\documentclass[11pt]{article}
\usepackage[activeacute,english]{babel} 
\usepackage[utf8]{inputenc}
\usepackage{verbatim} 
\usepackage{graphicx} 
\usepackage{adjustbox}
\usepackage[usenames,dvipsnames]{color} 
\usepackage{enumitem,mdwlist}  
\usepackage{multirow} 
\usepackage{amsmath,amsfonts,amssymb,amsthm}
\usepackage{empheq} 
\usepackage{bbm,dsfont} 
\usepackage{mathrsfs} 
\numberwithin{equation}{section} 
\usepackage{tikz}
\usetikzlibrary{patterns,calc}
\usepackage{float} 
\usepackage[labelfont=bf,labelsep=space]{caption} 
\usepackage{empheq} 
\usepackage{booktabs} 
\usepackage{stackrel}
\usepackage{array}
\usepackage{longtable}
\usepackage{tabularx}

\tikzset{
  symbol/.style={
    draw=none,
    every to/.append style={
      edge node={node [sloped, allow upside down, auto=false]{$#1$}}}
  }
}

\usepackage{subcaption}

\usepackage{geometry}
\usepackage{color}
\definecolor{red}{rgb}{.7,0,0}
\definecolor{blue}{rgb}{0,0,1}

\newcolumntype{H}{>{\setbox0=\hbox\bgroup}c<{\egroup}@{}}

\def\mcC{\mathcal{C}}

\def\mcD{\mathcal{D}}

\def\I{\mathcal{I}}
\def\J{\mathcal{J}}
\def\mcK{\mathcal{K}}
\def\mcL{\mathcal{L}}

\def\M{\mathcal{M}}

\def\mcP{\mathcal{P}}
\def\mcR{\mathcal{R}}

\def\mcX{\mathcal{X}}
\def\mcU{\mathcal{U}}

\def\bbI{\mathbb{I}}

\def\bbR{\mathbb{R}}
\def\bbZ{\mathbb{Z}}
\def\bbN{\mathbb{N}}

\def\bbP{\mathbb{P}}

\def\bbF{\mathbb{F}}

\def\bbX{\mathbb{X}}
\def\OA{\mathrm{OA}}

\def\Oh{\mathcal{O}}
\def\id{\mathrm{id}}
\def\g{\mathrm{K}}
\def\rr{\mathrm{r}}
\newcommand{\auto}[2]{\left[\,#1\mid #2\,\right]}
\DeclareMathOperator{\Aut}{\mathrm{Aut}}
\DeclareMathOperator{\Unb}{U}
\DeclareMathOperator{\Tol}{Tol}

\def\Sym{\Sigma}
\DeclareMathOperator{\im}{im}
\DeclareMathOperator{\HammingSimilarity}{H^\perp}
\DeclareMathOperator{\dist}{dist}
\DeclareMathOperator{\Hamming}{H}

\usepackage{ifpdf}
\ifpdf
\usepackage[pdfencoding=auto,unicode]{hyperref} 
\hypersetup{
 pdftoolbar=true,
 pdfmenubar=true,
 pdfstartview={FitV},
 pdftitle={Almost Orthogonal Arrays: Search Three Ways},
 pdfauthor={Luis Martínez, María Merino, Juan Manuel Montoya, Josué Tonelli-Cueto},
 pdfsubject={orthogonal arrays,experiment design},
 pdfkeywords={othogonal arrays},
 pdfcreator={overleaf}, 
 pdfproducer={overleaf},
 linkcolor=red     
 citecolor=green   
 urlcolor=cyan     
 filecolor=magenta 
}
\fi
\usepackage{bookmark}

\usepackage{authblk}

\title{Almost Orthogonal Arrays:\\Search Three Ways}
\author[1]{Luis~Martínez\thanks{E-mail: {\tt luis.martinez@ehu.eus}.}}
\author[1,2]{María~Merino\thanks{E-mail: {\tt maria.merino@ehu.eus}.}}
\author[1,3]{Juan~Manuel~Montoya\thanks{E-mail: {\tt jmontoya006@ikasle.ehu.eus}.}}
\author[4]{Josué~Tonelli-Cueto\thanks{Corresponding author: {\tt  josue.tonelli.cueto@bizkaia.eu}.}}
\affil[1]{Dept. of Mathematics\\ 
Univ. of the Basque Country (EHU)\\ 
Leioa, SPAIN}
\affil[2]{Basque Center for Applied Mathematics (BCAM), Bilbao,  SPAIN} 
\affil[3]{
University of Valle, Faculty of Natural and Exact Sciences, Cali, COLOMBIA}
\affil[4]{Dept. of Quantitative Methods\\ 
CUNEF Universidad\\
Madrid, SPAIN}
\date{}

\makeatletter
\def\th@plain{%
  \thm@notefont{}
  \slshape 
}
\def\th@definition{%
  \thm@notefont{}
  \normalfont 
}
\makeatother
\theoremstyle{plain}
\newtheorem{lem}{Lemma}[section]
\newtheorem{prop}[lem]{Proposition}
\newtheorem{thm}[lem]{Theorem}
\newtheorem*{theo*}{Theorem}
\newtheorem*{temptheo*}{Template Theorem}
\newtheorem{cor}[lem]{Corollary}

\newtheorem{construct}[lem]{Construction}

\theoremstyle{definition}
\newtheorem{dfn}[lem]{Definition}

\theoremstyle{remark}

\newtheorem{ex}[lem]{Example}

\newtheorem{remark}[lem]{Remark}

\usepackage{multicol}
\usepackage[linesnumbered,ruled,algosection,vlined]{algorithm2e}

\SetCommentSty{mycommfont}
\let\oldnl\nl
\newcommand{\nonl}{\renewcommand{\nl}{\let\nl\oldnl}}
\SetKwRepeat{Do}{do}{while}

\usepackage{refcount}

\usepackage{csquotes}
\usepackage[
backend=biber,
defernumbers=true,
style=apa,
citestyle=numeric,
sorting=nyt,
maxbibnames=99
]{biblatex}
\DeclareNameAlias{apaauthor}{given-family}
\DeclareFieldFormat
  [article,inbook,incollection,inproceedings,patent,thesis,unpublished]
  {addendum}{#1.}
\DeclareLanguageMapping{english}{american-apa}
\newcounter{pubnumber}
\defbibenvironment{bibliography}
  {\list
     {\printtext{\stepcounter{pubnumber}%
        [\thepubnumber]%
        }}
     {\setlength{\topsep}{0pt}
      \setlength{\labelwidth}{5 pt}%
      \setlength{\labelsep}{3pt}%
      \leftmargin\labelwidth%
      \advance\leftmargin\labelsep}%
      \setlength{\leftmargin}{\bibhang}%
      \setlength{\itemsep}{\bibitemsep}%
      \setlength{\parsep}{\bibparsep}%
      \sloppy\clubpenalty4000\widowpenalty4000}
  {\endlist}
  {\item}
\allowdisplaybreaks 

\begin{document}
\maketitle
\begin{abstract}
Orthogonal arrays play a fundamental role in many applications. However, constructing orthogonal arrays with the required parameters for an application usually is extremely difficult and, sometimes, even impossible. Hence there is an increasing need for a relaxation of orthogonal arrays to allow a wider flexibility. The latter has lead to various types of arrays under the name of ``nearly-orthogonal arrays'', and less often ``almost orthogonal arrays''. In this paper, we explore how to find almost orthogonal arrays three ways: using integer programming, local search meta-heuristics and algebraic methods. We compare all our search results with the ones existing in the literature, and we show that they are competitive, improving some of the existing arrays for many non-orthogonality measures. All our found almost orthogonal arrays are available at a public repository.
\end{abstract}

\noindent\textbf{keyword:} orthogonal arrays, non-orthogonality measures, discrepancy measures, integer programming, heuristics, finite algebraic geometry

\section{Introduction}

Orthogonal arrays of strength $t$ and index $\lambda$ are arrays of symbols from a finite alphabet such that, in each set of $t$ columns, every $t$-tuple of symbols appears exactly $\lambda$ times. These arrays are a central object of mathematics \cite{heslst1999}. Furthermore, they play a fundamental role in many applications, such as design of experiments \cite{bj2007,wj1995} (cf.~\cite[Ch.~11]{heslst1999} and \cite[Ch.~11]{hinkelmankempthorne2008}), error-correcting codes \cite[Ch.~4]{heslst1999} and derandomization of algorithms~\cite{gopalakrishnan2006applications,owen1992}, among many others. See a recent review in \cite{lin2025orthogonal}.

Although the importance of orthogonal arrays is enormous, it is extremely difficult to construct orthogonal arrays with given parameters and even to determine if they exist. Their construction requires many mathematical tools: Hadamard matrices, finite fields, Latin squares, integer programming, among others. Despite all these tools available today, the range of orthogonal arrays at our disposal is insufficient to meet all practical needs for designing experiments and other applications. Moreover, the `brute force' search for orthogonal arrays has a high combinatorial complexity. 

This practical need for more arrays has motivated the relaxation of the notion of orthogonal array in several contexts. Under the name of ``nearly-orthogonal array''---and way less often ``almost orthogonal array''---a wide family of relaxations is available~\cite{booth1962,wu1992,lin1993,lin1995,denglinwang1999,fang2000,ma2000new,xu2002algorithm,yamadalin2002,yamadamatsui2002,UDpaper2,UDpaper13,UDpaper10,lu2006class,linphoakao2017}, having some of them usefulness for practical applications such as functional magnetic resonance imaging~\cite{linphoakao2017} and manufacturing and high-technology industries~\cite{xu2002algorithm}. Moreover, another family of relaxations emerged from the design of numerical integration schemes~\cite{hickernell1998,hickernell1998b,UDpaper8hickernellliu,UDpaper8liuhickernell,zhou2013,ke2015}.

In this paper, we explore, in this new setting, three ways to produce almost orthogonal arrays: a meta-heuristic local search algorithm and the two usual methods for producing orthogonal arrays, integer programming (IP)~\cite{rosenber1995,carlini2007,bulutoglu2008,vieira2011,geyer2014,geyer2019} and algebraic constructions~\cite[Ch.~3]{heslst1999}. We compare all the obtained arrays with all those existing in the literature, whenever this was possible\footnote{In this work, we consider the so-called fixed-level or unmixed arrays. Hence we can only compare against those, and not against the existing mixed-level almost orthogonal arrays.}.

Moreover, in our search for new almost orthogonal arrays, we perform extensive computation using IP and meta-heuristics local search. Additionally, we consider symmetric arrays as a tool to reduce the search space. Even though, the search for almost orthogonal arrays is more flexible, this search is still a challenging combinatorial optimization problem. In this way, by restricting to arrays with symmetries, we are able to bring the search methods to a computationally more tractable point. 

In our comparisons, we show that our methods are competitive. The arrays that we obtain are always comparable to the existing arrays in the literature, being better for a wide set of non-orthogonality measures. Moreover, all our arrays are available at a public repository\footnote{\url{https://github.com/tonellicueto/AOAs}} and we have optimality guarantees for some of our arrays.

\noindent
{\bf Organization.} 
In Section \ref{sec:concepts}, we introduce the measures of non-orthogonality that we use (\emph{tolerance} and \emph{unbalance}) as well as almost orthogonal arrays and our strategy to search for them. In Section~\ref{sec:concepts2}, we introduce symmetric arrays and showcase the symmetries we use to make our computations more tractable. Then, in Section~\ref{sec:overview}, we report, discuss and compare all our found arrays, leaving the detailed description of our search/construction methods for the latter sections: IP formulation, in Section \ref{sec:integerprogramming}; heuristic methods, in Section~\ref{sec:heuristic}; and algebraic constructions, in Section~\ref{sec:algebraic}.

Further, in Appendix~\ref{sec:unbalancebounds}, we provide some optimality results guaranteeing the optimality of some of our found arrays, and in Appendix~\ref{sec:comparative}, we do an extensive survey and comparison between all existing relaxations of orthogonal arrays and measures of non-orthogonality. Finally, proofs are delegated to the final appendices: proofs of Sections~\ref{sec:concepts} and~\ref{sec:concepts2} to Appendix~\ref{secproofA}, proofs of Section~\ref{sec:algebraic} to Appendix~\ref{secproofB} and proofs from Appendices~\ref{sec:unbalancebounds} and~\ref{sec:comparative} to Appendix~\ref{sec:unbalanceproofs}.

\noindent
{\bf Acknowledgments.} The authors thank for technical and human support provided by SGIker (UPV/EHU/ ERDF, EU) of the University of the Basque Country.

L.M. was supported by the University of the Basque Country  (EHU) and Basque Center of Applied Mathematics, grant US21/27. 

M.M. has been partially supported by Spanish Ministry of Science and Innovation through the 
grant PID2023-147410NB-100 funded by MICIU/AEI/ 10.13039/501100011033 and by ERDF/EU and BCAM Severo Ochoa accreditation CEX2021-001142-S; and by the Basque Government through the program BERC 2022-2025 and the project IT1494-22.

J.T.-C. did part of this work on a stay partially financed with funds from a 2023 AMS-Simons Travel Grant, while employed at Johns Hopkins University. J.T.-C. was also partially supported by a postdoctoral fellowship of the 2020 ``Interaction'' program of the \emph{Fondation Sciences Mathématiques de Paris}, by the ANR JCJC
GALOP (ANR-17-CE40-0009), the PGMO grant ALMA, and the PHC GRAPE, while at Inria Paris and the IMJ-PRG; and by start-up funds of Alperen A. Ergür and NSF CCF 2110075, while at the University of Texas at San Antonio. Finally, J.T.-C. is thankful to Jazz G. Suchen for suggestions regarding Theorems~\ref{thm:AKconstruction} and \ref{thm:AKconstructionB}.

\section{Almost-Orthogonal Arrays}\label{sec:concepts}

Recall that, due to the connection of orthogonal arrays (OA for short) with experimental designs, we will refer to the rows of an array as \emph{runs}, to its columns as \emph{factors} and to the symbols in its entries as \emph{levels}.

\begin{dfn} \cite[Ch.~1]{heslst1999}
An \emph{orthogonal array with $N$ runs, $k$ factors, $s$ levels and strength $t$}, or an \emph{$\OA(N,k,s,t)$} for short, is an $N\times k$ array $A\in S^{N\times k}$ with entries in $S$ of size $s$ such that whenever we delete $k-t$ columns of $A$ every $t$-tuple with entries in $S$ appears the same number of times, $N/s^t$. The \emph{index} of such an orthogonal array, denoted by $\lambda(A)$ or simply $\lambda$, is this last number.
\end{dfn}

\subsection{Measures of Non-Orthogonality: Tolerance and Unbalance}

We will introduce two measures of non-orthogonality for an array. We introduce first some notation. To parameterize the possible ways of choosing $t$ columns in an array with $k$, we consider the set
\begin{equation}\label{eq:tsubset}
T_{t,k}:=\{(j_1,\dots,j_t)\in [k]^t\,:\,1\leq j_1<\dots < j_t\leq k\}
\end{equation}
where $[k]$ denotes the set $\{1,2,\dots,k\}$. Note that $T_{t,k}$ is in bijection to $\binom{[k]}{t}$, the set of size $t$ subsets of $[k]$, and so $\#T_{t,k}=\binom{k}{t}$. Here, we need to count how many times $x\in S^t$ appears as a row in the $N\times t$ subarray of the array $A\in S^{N\times k}$ selected by $j\in T_{t,k}$, this count is given by
\begin{equation}\label{eq:countingnumber}
   n(A,x,j):=\#\{i\in [N]\,:\,\forall r\in [t],\,a_{i,j_r}=x_r\}. 
\end{equation}
More explicitly, $n(A,x,j)$ counts the number of times the $t$-tuple $x$ appears as a row in the subarray 
\begin{equation}
A[j]:=A[j_1,\ldots,j_t]:=\begin{pmatrix}
A_{1,j_1}&\cdots&A_{1,j_t}\\
\vdots&\ddots&\vdots\\
A_{N,j_1}&\cdots&A_{N,j_t}
\end{pmatrix}
\end{equation}
of $A$, obtained after selecting columns $j_1,\ldots,j_t$ of $A$. Now, the following characterization is immediate.

\begin{prop}\label{prop:OAchar1}
Let $A\in S^{N\times k}$ with $\#S=s$. Then $A$ is an $OA(N,k,s,t)$ if and only if for all $x\in S^t$ and all $j\in T_{t,k}$,
\begin{equation}
    n(A,x,j)=N/s^t.
\end{equation}
\end{prop}

In this way, to measure how far is an array from being orthogonal, we just need a measure of how far are the $n(A,x,j)-N/s^t$ from being zero. We do this by the means of the usual norms, obtaining the tolerance and the unbalance.

\begin{dfn}\label{def:unbal}
Let $A\in S^{N\times k}$ with $\#S=s$ and $t\in \bbN$. The \emph{tolerance of strength $t$ of $A$} is
\begin{equation}
    \Tol_{t}(A):=\max\left\{\vert n(A,x,j)-N/s^t\vert\,:\, x\in S^t,\,j\in T_{t,k}\right\},
\end{equation}
and, for $p\in[1,\infty)$, the \emph{$p$-unbalance of strength $t$ of $A$} is
\begin{equation}
    \Unb_{p,t}(A):=\sum_{x\in S^t}\sum_{j\in T_{t,k}} \vert n(A,x,j)-N/s^t\vert^p.
\end{equation}
When $p$ and/or $t$ are clear from context, we will omit them from the notation.
\end{dfn}

\begin{remark}
The notion of unbalance is not inherently new. In \cite{lu2006class} and \cite{UDpaper2}, their balance, $B(t)$, is equal to our $2$-unbalance up to a factor of $\binom{k}{t}$. In our opinion, `unbalance' is a better term, as $\Unb_{p,t}$ measures how non-orthogonal is an array. Further, there are many notions on the literature that turn out to be equivalent to our unbalance and tolerance. We discuss them in detail in Appendix~\ref{sec:comparative}.
\end{remark}

We note that our definitions above can be normalized to obtain the following measure:
\begin{equation}\label{eq:normunbtol}
    \widehat\Unb_{p,t}(A):=\begin{cases}
\Tol_t(A)&\text{if }p=\infty\\
\left(\frac{\Unb_{p,t}(A)}{s^t\binom{k}{t}}\right)^{\frac{1}{p}}&\text{if }p\in[1,\infty)
\end{cases}
\end{equation}
which is useful when comparing how orthogonal are arrays of different numbers or runs, factors or levels. Noticing that this is just a $p$-mean gives some standard inequalities.

\begin{prop}\label{prop:unbtolrelations}
Let $A\in S^{N\times k}$ be an $N\times k$ array with entries in a set $S$ of size $s$, $t\in \bbN$ and $p,q\in[1,\infty)$ such that $p\leq q$. Then
\begin{equation}
    \Tol_t(A)^p\leq \Unb_{p,t}(A)\leq s^t\binom{k}{t}\Tol_t(A)^p,
\end{equation}
and
\begin{equation}
   \Unb_{q,t}(A)^p\leq \Unb_{p,t}(A)^{q}\leq \left(s^t\binom{k}{t}\right)^{q-p}\Unb_{q,t}(A)^p.
\end{equation}
\end{prop}

\subsection{AOAs and their Search through Minimization: MTP and MUP}

By putting our focus on the tolerance, we define almost orthogonal arrays (AOA for short). However, this is not intended as a rigid notion.

\begin{dfn}\label{def:aoa} 
An \emph{almost-orthogonal array with $N$ runs, $k$ factors, $s$ levels, strength $t$ and tolerance $\epsilon$}, or \emph{$AOA(N,k,s,t,\epsilon)$} for short, is an $N\times k$ array $A\in S^{N\times k}$ with entries in $S$ of size $s$ such that
\[
\Tol_t(A)\leq\epsilon.
\]
The \emph{index} of such an almost-orthogonal array $A$, denoted by $\lambda(A)$ or simply $\lambda$, is $N/s^t$. 
\end{dfn}
\begin{remark}
In \cite{linphoakao2017}, the term almost-orthogonal array was introduced. In Appendix~\ref{sec:comparative}, and Proposition~\ref{prop:bwvstol} therein, we show that both notions are equivalent.
\end{remark}

As we have several non-orthogonality measures, we can exploit them to either find AOAs by considering the minimization problems for the tolerance and the unbalance. In this way, we obtain the \emph{Minimum Tolerance Problem} (MTP) given by
\begin{equation}
    \min\{\Tol_t(A)\mid A\in S^{N\times k}\}
\end{equation}
and the \emph{Minimum Unbalance Problem} (MUP) given by
\begin{equation}
    \min\{\Unb_{p,t}(A)\mid A\in S^{N\times k}\}
\end{equation}
depending on whether we minimize the tolerance or one of the unbalances. Both of these are combinatorial optimization problems that we handle using IP methods (Section~\ref{sec:integerprogramming}) and heuristic methods (Section~\ref{sec:heuristic}). Further, to reduce the search space we will be considering arrays with symmetries as described in Section~\ref{sec:concepts2}.

\subsection{Trivial construction of AOAs out of OAs for strength two} \label{sec:trivial}

The following proposition provides a ``trivial'' construction of an AOA out of an OA. The main idea is to repeat one of the factors of the OA. It gives us a baseline to compare our searches again.

\begin{prop}\label{prop:trivialconstruction}
Let $N,k,s\in\bbN$ with $k$ be the biggest integer such that an $OA(N,k-1,s,2)$ $B\in [s]^{N\times (k-1)}$ exists. Fix $\lambda=N/s^2$. Then the array
\[
A=\left(\begin{array}{c|c}  
    B_1 &B
\end{array}\right)\in [s]^{N\times k},
\]
where $B_1$ is the first factor of $B$, satisfies   
\begin{equation}\label{eq:defiepsbar}
    \Tol_{2}(A)\leq \lambda(s-1)=: \overline{\epsilon}
\end{equation}
and, for $p\in[1,\infty)$,
\begin{equation}\label{eq:defiUbar}
\Unb_{p,2}(A)=\lambda^p s(s-1)\left[1+(s-1)^{p-1}\right]=:\overline{U}.
\end{equation}
\end{prop}
\begin{remark}
Interestingly, when $\lambda=1$ and $s$ is a prime power, Corollary~\ref{cor:HamminSimilarityvsUnbalance2} guarantees that the trivial construction gives an optimal $2$-unbalance. However, it does not give neither an optimal tolerance nor 1-unbalance. 
\end{remark}
\begin{remark}
When repeating one of the factors, the obtained array is useless for factorial design because we cannot distinguish the effect of the repeated factor. 
\end{remark}

\section{Symmetries for searching faster AOAs}\label{sec:concepts2}

Searching for AOAs in the space of all arrays is computationally expensive. To reduce the search space, we can restrict our search to arrays that have symmetries. In this section, we recall the symmetries of an array and introduce the three families of symmetric arrays we consider. All relevant proofs can be found in the Appendix~\ref{secproofA}.

\subsection{Symmetries of an array} \label{subsec:threesymmetries}

We are not interested in symmetries up to equality of arrays, but in symmetries up to permutation of the rows. The reason for this is that in experimental design, the order of the runs---the order in which we run our experiments---is irrelevant (at least in principle).

\begin{dfn}
Given arrays $A,\tilde{A}\in S^{N\times k}$, we say that \emph{$A$ and $\tilde{A}$ are equivalent}, $A\cong \tilde{A}$, if there is a row permutation $P$ such that
\[
\tilde{A}=PA.
\]
In other word, $A\cong \tilde{A}$ if and only if the multisets of their rows are equal (meaning they have the same number of each rows counting with multiplicity).
\end{dfn}

Let
\[
\Sym(S,k):=\Sym(S)\times \Sym_k
\]
be the product of the group of permutations of $S$, $\Sym(S)$, and the group of permutations of $[k]:=\{1,\ldots,k\}$, $\Sym_k$. We will denote its elements by
\begin{equation}
    \auto{g}{\sigma}
\end{equation}
where $g\in \Sym(S)$ and $\sigma\in \Sym_k$. In this way, we will write\footnote{In what follows, recall that $(u_1,u_2,\ldots,u_c)$,
with the $u_i$ distinct, denotes the cyclic permutation sending $u_1$ to $u_2$, $u_2$ to $u_3$,\ldots,$u_{c-1}$ to $u_c$, and $u_c$ to $u_1$. } expression like $\auto{(1,2,3)}{(1,2)(3,4)}$ instead of $((1,2,3),(1,2)(3,4))$ for an element of $\Sym(\{1,2,3\},4)$. Now, $ \Sym(S,k)$ acts naturally on $S^{N\times k}$ as follows: given $\auto{g}{\sigma}\in \Sym(S,k)$ and $A\in S^{N\times k}$,
\begin{equation}
    \auto{g}{\sigma}A:=\begin{pmatrix}gA_{i,\sigma^{-1}(j)}\end{pmatrix}_{\substack{i\in [N]\\j\in[k]}},
\end{equation}
i.e., $g$ permutes the symbols in the array and $\sigma$ the columns\footnote{The $\sigma^{-1}$ in the subindex is needed to make $\sigma$ send the $i$th column to the place of the $\sigma(i)$th column, and to guarantee the axioms of group action.}.

Once we have introduced the equivalence of arrays and the group action on arrays, we are ready for the definition of symmetry---automorphism---of an array. 

\begin{dfn}
Given an array $A\in S^{N\times k}$, an \emph{automorphism (or symmetry) of $A$} is an element $\auto{g}{\sigma}\in\Sym(S,k)$ such that
\[
\auto{g}{\sigma}A\cong A.
\]
\emph{The automorphism group of $A$}, $\Aut(A)$, is the group of automorphisms of $A$, i.e., 
\begin{equation}
    \Aut(A):=\left\{\auto{g}{\sigma}\in\Sym(S,k): \auto{g}{\sigma}A\cong A\right\}.
\end{equation}
\end{dfn}
\begin{remark}
Note that \emph{an automorphism group} of $A$ means a subgroup of $\Aut(A)$. Therefore it should be clear that when we say that $A$ has $G\leq\Sym(S,k)$ as an automorphism group, we mean that $G\leq\Aut(A)$ allowing the inclusion to be strict.
\end{remark}
\begin{ex}
Consider the arrays
\[
A=\begin{pmatrix}
a&b&c\\
b&c&a\\
c&a&b
\end{pmatrix} \text{ and }\ B=\begin{pmatrix}
a&b&c&a\\
b&a&b&a
\end{pmatrix}.
\]

Regarding $A$, we can see that both $\auto{(a,b,c)}{\id}$ and $\auto{\id}{(1,2,3)}$ are automorphisms of $A$---in both cases, we need to permute the rows to see this. However, note that there are more automorphisms such as $\auto{(a,b)}{(1,2)}$, even though neither $\auto{(a,b)}{\id}$ nor $\auto{\id}{(1,2)}$ are so.

Regarding $B$, we can prove that $\Aut(B)=\{\auto{\id}{\id}\}$, and so that $B$ has no symmetries.
\end{ex}
\begin{ex}
Consider the following array
\[
A=\begin{pmatrix}1&1&2&2\\1&2&1&2\end{pmatrix}.
\]
We can see that $\auto{(1,2)}{(1,4)(2,3)}$ is a non-trivial symmetry of $A$. Moreover, it is the generator of the automorphism group of $A$.
\end{ex}

\begin{prop}\label{prop:unbsymmetries}
Let $A\in S^{N\times k}$ be an array, $t\in\bbN$ and $p\in[1,\infty)$. Then, for every $\auto{g}{\sigma}\in \Sigma(S,k)$,
\begin{equation}
\Unb_{p,t}(\auto{g}{\sigma}A)=\Unb_{p,t}(A)~\text{ and }~\Tol_t(\auto{g}{\sigma}A)=\Tol_t(A).
\end{equation}
In particular, the action of $\Sigma(S,k)$ on $S^{N\times k}$ preserves OAs and AOAs. 
\end{prop}

\subsection{Three families of symmetric arrays}\label{subsec:symmetries}

In our search for AOAs, we will consider symmetric arrays that lie in three big families: bi-cyclic, quasi- and semi-cyclic and Klein. IP methods (Section~\ref{sec:integerprogramming}) exploit the first two, and heuristic methods (Section~\ref{sec:heuristic}) the first and the third. We will comment briefly how these symmetries allow us to reduce the search space.

\subsubsection{Quasi- and semi-cyclic arrays} \label{sec:qc}

A semi-cyclic array is an array where an automorphism permuts cyclically a subset of the levels. This notion generalizes the well-known case of cyclic arrays in the literature~\cite{bosebush1952} \cite[Part~VI, Ch.~17]{colbourne2007}. The quasi-cyclic case was already exploited in~\cite{montoya2022}.

\begin{dfn}
A \emph{semi-cyclic array} is an array $A\in [s]^{N\times k}$ that has an automorphism, called \emph{the semi-cyclic automorphism of $A$}, of the form
\[
\auto{(a,\ldots,s)}{\id}
\]
where $a\in [s]\setminus \{s\}=\{1,\ldots,s-1\}$. We will say \emph{quasi-cyclic} instead of semi-cyclic when $a=2$, and \emph{cyclic} when $a=1$.
\end{dfn}

In general, when the semi-cyclic automorphism acts on the runs of the corresponding semi-cyclic array, we have orbits of size $1$, when all the levels of the run are in $\{1,\ldots,a-1\}$; and orbits of size $s-a+1$, otherwise. This allows one to compress the search space significantly, particularly in the case of quasi-cyclic arrays, where a quasi-cyclic array with $N$ runs, $k$ factors and $s$ levels can be encoded into an array with $\left\lceil N/(s-1)\right\rceil$ runs and $k$ factors.

\begin{ex}
Consider the quasi-cyclic array
\[
\begin{pmatrix}
1&1&1&1&1\\
1&1&2&3&2\\
1&1&3&2&3\\
1&3&2&1&1\\
1&2&3&1&1
\end{pmatrix}
\]
with quasi-cyclic automorphism $\auto{(2,3)}{\id}$, where the first run is fixed and the next two pairs are the orbits of the action on runs. Instead of encoding this full array, we can just give the automorphism and the following array
\[
\begin{pmatrix}
1&1&2&3&2\\
1&3&2&1&1
\end{pmatrix}
\]
that has a member of each orbit of the action on rows and where the fixed run is implicitly given by the number of runs of the encoded array.
\end{ex}

The above feature is explore in two ways. In the IP search (Section~\ref{sec:integerprogramming}), we exploit both the semi- and quasi-cyclic to identify variables of the IP through simple additional linear constraints. In the heuristic methods (Section \ref{sec:heuristic}), we exploit (only) quasi-cyclic arrays to encode a bigger array into a smaller one---so instead of searching inside a search space of size $s^{Nk}$, we search in one of size $s^{(N-\lambda)k/(s-1)}$.

\subsubsection{Klein arrays}  \label{sec:klein}

A Klein array is an array where we can permute non-cyclically the first four factors. A class of interest is that of semi-cyclic Klein arrays.

\begin{dfn}
A \emph{Klein array} is an array $A\in [s]^{N\times k}$ that has an automorphism, called \emph{the Klein automorphism of $A$}, of the form
\[
\auto{\id}{(1,2)(3,4)}.
\]
A \emph{semi-cyclic Klein array} is an array that is both semi-cyclic and Klein.
\end{dfn}

When acting on the runs of a Klein array, the Klein automorphism will have orbits of size one and two. This allows roughly to encode a Klein array into an array with half the number of runs. 

\begin{ex}
The array 
\[
\begin{pmatrix}
1&1&2&3&2\\
1&1&3&2&2\\
1&1&3&2&3\\
1&1&2&3&3\\
1&3&2&1&1\\
3&1&1&2&1\\
1&2&3&1&1\\
2&1&1&3&1
\end{pmatrix}
\]
is a Klein quasi-cyclic array. 
\end{ex}

The above feature is exploited in the search of AOAs using IP (Section~\ref{sec:integerprogramming}), by identifying variables of the IP through simple additional linear constraints. 

\subsubsection{Bi-cyclic array}\label{sec:bc}

A bi-cyclic array is an array admitting an automorphism that permutes cyclically at the same time the levels and the first columns of the array with some additional restrictions.

\begin{dfn}
A \emph{bi-cyclic array} is an array $A\in [s]^{N\times k}$ that has an automorphism, called \emph{the bi-cyclic automorphism of $A$}, of the form
\[
\auto{(1,\ldots,s)}{(1,\ldots,r)}
\]
where $r$ is a divisor of $s$ such that $r\leq k$.
\end{dfn}

\begin{remark}
In the heuristic methods, where this symmetry is exploited, the parameter $r$ will always be the maximum divisor of $s$ that does not exceed $k$. 
\end{remark}

In general, if a bi-cyclic array with $s$ levels has more factors than levels, then we have that its cyclic automorphism acts freely on the set of runs of the array. Because of this, the action of this automorphism on the set of rows produces orbits of size $s$. This fact allows us to encode a bi-cyclic array with $N$ runs and $k$ factors into an array with $N/s$ runs and $k$ factors. 

\begin{ex}
Consider the bi-cyclic array
\[
\begin{pmatrix}
1&1&2&3&2\\
3&2&2&1&3\\
3&1&3&2&1\\
1&3&2&1&1\\
3&2&1&2&2\\
2&1&3&3&3
\end{pmatrix}
\]
with bi-cyclic automorphism $\auto{(1,2,3)}{(1,2,3)}$, where the first three runs belong to one orbit and the second three runs to another. Instead of encoding this full array, we can just give the automorphism and the following array
\[
\begin{pmatrix}
1&1&2&3&2\\
1&3&2&1&1
\end{pmatrix}
\]
that has a member of each orbit of the action on rows.
\end{ex}

The above feature of bi-cyclic arrays is exploited in the heuristic methods (Section \ref{sec:heuristic}) to reduce the search space dramatically---instead of the search space of size $s^{Nk}$ of arrays without symmetries, we pass to the search space of size $s^{Nk/s}$.

\section{Search/Construction Results} \label{sec:overview}

In this section, we report the results of our search through IP (Section~\ref{sec:integerprogramming}) and heuristic methods (Section~\ref{sec:heuristic}) as well as our algebraic constructions (Section~\ref{sec:algebraic}). In general, we will focus on the MTP and MUP at strength $t=2$ for arrays with $N$ runs, $k$ factors and $s\leq 10$ levels such that:
\begin{enumerate}
\item[(1)] The index $\lambda=N/s^2$ is either $1$ or $2$.
\item[(2)] There is an $OA(N,k-1,s,2)$, but no known $OA(N,k,s,2)$ (see \cite[Ch.~6, III]{colbourne2007}).
\end{enumerate}
For these arrays, we will report the arrays which among the ones we found/constructed have the best $1$-unbalance in Tables~\ref{tb:l1p1} and~\ref{tb:l2p1}; the best $2$-unbalance in Tables~\ref{tb:l1p2} and~\ref{tb:l2p2}; and the best tolerance in Tables~\ref{tb:l1e1} and~\ref{tb:l2e1}. We discuss our results regarding the unbalance in Subsection~\ref{sub:u}, and those regarding the tolerance in Subsection~\ref{sub:e}. We specify the used notation in Table~\ref{tb:not}, where in all tables, we specify in bold the dominating values.

Additionally, in Subsection~\ref{sec:sl}, we compare the arrays that we found with respect to others found in the literature, such as those of \cite{ke2015}, \cite{UDrepository} and \cite{ma2000new}. However, as these other arrays have been built/found minimizing other non-orthogonality measures, we do the comparison with respect the tolerance and six other non-orthogonality measures.

First, we consider the $\mcD$-value introduced Wang and Wu~\cite[(4.4)]{wang1992nearly}. Being more concrete, we use $\mcD_f$, with $f(a)=a-\frac{s+1}{2}$, which we define precisely, as the definition is quite lengthy, in Definition~\ref{defi:Dvalue} in Appendix~\ref{sec:comparative}, where a full discussion on this non-orthogonality measure is done. Second, we consider the so-called \emph{$D_1$ and $D_2$ criteria} of Ma, Fang and Liski~\cite{ma2000new} which, for an $N\times k$-array $A$, can be defined as
\[
D_i(A)=\binom{k}{2}^{-1}\Unb_{i,2}(A).
\]
The original definition, as well as further comparisons and discussions, can be found around Definition~\ref{defi:D1D2} in Appendix~\ref{sec:comparative}.

Finally, we consider the three discrepancy measures: the centralized $L_2$-discrepancy, CD; the wrap-around $L_2$-discrepancy, WD; and the mixture $L_2$-discrepancy, MD. These three measures of non-orthogonality appear in the study of quadrature discrepancy for families of points and were introduced by Hickernell~\cite{hickernell1998,hickernell1998b} (CD and WD) and Zhou, Fang and Ning~\cite{zhou2013} (MD). For the sake of the flow of the text, we give their definitions in Definition~\ref{defi:CDWDMD} in Appendix~\ref{sec:comparative}, together with a discussion on the relationship of these discrepancy measures to the unbalance. We computed these measures in our experiments with the DiceDesign package \cite{dicedesign} from R software \cite{R}.

We report these comparative results in Tables~\ref{tab:udl1} and \ref{tab:udl2}. Moreover, in Table~\ref{tab:make}, we further report our comparative results against arrays with parameters beyond our initial consideration. We discuss these comparative results in Subsection~\ref{sec:sl}. We specify the notation used in Table~\ref{tb:not2}.

All the arrays, whose results we report, are available at the following repository:
\begin{center}
    \url{https://github.com/tonellicueto/AOAs}
\end{center}
In each of the tables mentioned above, we mark in boldface the best value attained for the array.

\begin{remark}
All the searches, either through IP (Section~\ref{sec:integerprogramming}) and heuristic methods (Section~\ref{sec:heuristic}), were conducted in the ARINA computational cluster from SGI/IZO-SGIker at the University of the Basque Country \cite{arina}. In total, we run each of our searches up to a maximum of 100Gb of memory and 72 hours of computation time. Because of this, the search has been limited to arrays with the number of runs, factors and levels that we have considered.
\end{remark}

\begin{remark}
We note that the algebraic constructions (Section~\ref{sec:algebraic}) produce arrays only for $s$ a prime power. Due to this, non-prime-power values of $s$ are missing from the tables regarding this methodology. However, this method, unlike the other two methods, produces infinite families of AOAs.
\end{remark}

\begin{table}[p]  \caption{Notation (Not.) for Tables~\ref{tb:l1p1}, \ref{tb:l1p2}, \ref{tb:l1e1}, \ref{tb:l2p1}, \ref{tb:l2p2},  and~\ref{tb:l2e1}} \label{tb:not}
\begin{tabularx}{\linewidth}{rXr}
Not. & Description & Ref. \\
\toprule
   $s$  & number of levels (symbols) & \\ 
   $k$   &number of factors (columns)& \\ \midrule
   $\overline{U}$ &  unbalance of the trivial construction \eqref{eq:defiUbar}& \ref{sec:trivial}\\ 
   $\overline{\epsilon}$ &  tolerance of the trivial construction \eqref{eq:defiepsbar}&  \ref{sec:trivial}\\ \midrule
   $\mathfrak u$ &  unbalance of the found array& Def.~\ref{def:unbal}\\ 
    $\mathfrak{e}$ &  tolerance of the found array& Def.~\ref{def:unbal}\\  \midrule\midrule
   st &  solving status of the IP& \ref{sec:integerprogramming}\\\midrule
  \quad $OP$  & the optimal solution has been found\\
  \quad $ML$  & the 100Gb memory limit has been exceeded \\
  \quad $TL$  & the 72 hours time limit has been exceeded  \\
  \quad $^i$  & an OA with less than $k$ factors considered as input,\\&design obtained with the DoE.base package \cite{doebase} from R  \cite{R}\\ 
  \quad $^j$  & the IP solution for alternative parameters configuration\\&  fully or partially considered as input\\      \midrule\midrule
   $G$  & the automorphism group & \ref{subsec:symmetries}\\\midrule
   $\auto{1}{\rr}$ & bi-cyclic automorphisms $\auto{(1,2,\ldots,s)}{(1,2,\cdots, r)}$ & \ref{sec:bc}\\
   $\auto{a}{\cdot}$& quasi- and semi-cyclic automorphisms $\auto{(a,a+1,\ldots,s)}{\cdot}$ & \ref{sec:qc}\\
   $\auto{\cdot}{\g}$& Klein automorphism $\auto{\cdot}{(1,2)(3,4)}$ & \ref{sec:klein}\\ \bottomrule
   \end{tabularx}
\bigskip
\caption{Notation (Not.) for Tables~\ref{tab:make}, ~\ref{tab:udl1} and \ref{tab:udl2}} \label{tb:not2}
\begin{tabularx}{\linewidth}{rXr}
Not. & Description & Ref. \\
\toprule
$s$  & number of levels (symbols) & \\ 
$k$   &number of factors (columns)& \\ \midrule
$\mathfrak{e}$ &  tolerance of the found array& Def.~\ref{def:unbal}\\
$\mcD_f$ &  $\mcD$-value with respect $f$& Def.~\ref{defi:Dvalue}\\
$D_1$ & $D_1$-criterion & Def.~\ref{defi:D1D2}\\
$D_2$ &  $D_2$-criterion& Def.~\ref{defi:D1D2}\\
$CD$ &  centralized $L_2$-discrepancy& Def.~\ref{defi:CDWDMD}\\
$WD$ &  wrap-around $L_2$-discrepancy& Def.~\ref{defi:CDWDMD}\\
$MD$ &  mixture $L_2$-discrepancy& Def.~\ref{defi:CDWDMD}\\\midrule\midrule
Met. &   the methodology or source used for obtaining the results \\\midrule
IP & the Integer Programming model &  \ref{sec:integerprogramming}\\
BC & the two-stage local search   heuristic for bi-cyclic AOAs (TSBC)& \ref{sec:heuristic}\\
QC & the two-stage local search   heuristic for quasi-cyclic AOAs (TSQC)& \ref{sec:heuristic}\\
AC & the algebraic construction &  \ref{sec:algebraic} \\ \midrule
Ke & array obtained from~\cite{ke2015} &   \\ 
Ma & array obtained from~\cite{ma2000new} &  \\ 
CD & array obtained from~\cite{UDrepository} optimal with respect CD &  \\ 
WD & array obtained from~\cite{UDrepository} optimal with respect WD&  \\ 
\bottomrule
\end{tabularx}
\end{table}

\begin{table}[tp]
	\centering
	\caption{Minimum unbalance results for $\lambda=1$ and $p=1$ } \label{tb:l1p1}
	\begin{tabular}{rrr rlrr rrrH rr}
		\toprule
		Instance & \multicolumn{2}{c}{Bounds} & \multicolumn{4}{c}{Integer Programming} & 
		\multicolumn{4}{c}{Heuristic} &  \multicolumn{2}{c}{AC}   \\ \cmidrule(l){2-3}\cmidrule(l){4-7} \cmidrule(l){8-11} \cmidrule(l){12-13} 
		$(s,k)$ & $\overline{U}$ & $\overline{\epsilon}$ & $\mathfrak u$ & st & $\mathfrak{e}$ & $G$ &   $\mathfrak u$ & $\mathfrak{e}$ & $G$ & met.  & $\mathfrak u$ & $\mathfrak{e}$ \\
		\midrule
		$(2,4)$ & {\bf 4} & {\bf 1}& {\bf 4} & OP & {\bf 1} & $\auto{\id}{\g}$ & {\bf 4}& {\bf 1} & $\auto{1}{2}$ & BC & {\bf 4} & {\bf 1}\\ 
		$(3,5)$ &  {\bf 12} & 2& {\bf 12} & OP & { 2} & $\auto{\id}{\g}$  & {\bf 12} & { 2} &$\auto{1}{3}$ & BC & 18 & {\bf 1}\\ 
		$(4,6)$ &  {\bf 24} & 3& {\bf 24} & OP & { 3} & $\auto{\id}{\g}$  & {\bf 24} & { 3} & $\auto{2}{\id}$ & QC & 48 & {\bf 1}\\ 
		$(5,7)$ &  {\bf 40} & 4& {\bf 40} & OP & { 4} & $\auto{2}{\id}$  & {\bf 40} & { 4} &$\auto{1}{5}$ & BC & 100 & {\bf 1}\\  		
  $(6,4)$ & 60 & 5& {\bf 4} & OP & {\bf 1} & $\auto{3}{\id}$  &  12 &  {\bf 1}  & $\auto{1}{3}$ & BC & $-$ & $-$\\ 
		$(7, 9)$ &	 {\bf 84} & 6& {\bf 84} & TL & { 6} & $\auto{2}{\id}$   & {\bf 84} & { 6} &$\auto{2}{\id}$  & QC & 294 & {\bf 1}\\	
		$(8, 10)$ &	 {\bf 112} &7& {\bf 112} & OP & { 7} & $\auto{2}{\id}$  & {\bf 112} & { 7} & $\auto{2}{\id}$ & QC & 448 & {\bf 1}\\
		$(9, 11)$ &	{\bf 144} &8& {\bf 144} &  OP$^i$ & { 8} & $\auto{\id}{\id}$ &  976 & 8 & $\auto{2}{\id}$ & QC & 648 & {\bf 1}\\	
		$(10, 5)$	& 180 &9& {\bf 36}  & TL & {\bf 1} & $\auto{2}{\g}$  & {\bf 36} & {\bf 1} & $\auto{2}{\id}$ & QC & $-$  & $-$\\
		\bottomrule
	\end{tabular}
 \bigskip

 \caption{Minimum unbalance results for $\lambda=1$ and $p=2$ } \label{tb:l1p2}
\begin{tabular}{rrr rlrr rrrH rr}
		\toprule
  Instance & \multicolumn{2}{c}{Bounds} & \multicolumn{4}{c}{Integer Programming} & 
		\multicolumn{4}{c}{Heuristic} &  \multicolumn{2}{c}{AC}   \\ \cmidrule(l){2-3}\cmidrule(l){4-7} \cmidrule(l){8-11} \cmidrule(l){12-13} 
		$(s,k)$ & $\overline{U}$ & $\overline{\epsilon}$ & $\mathfrak u$ & st & $\mathfrak{e}$ & $G$ &   $\mathfrak u$ & $\mathfrak{e}$ & $G$ & met.  & $\mathfrak u$ & $\mathfrak{e}$ \\
		\midrule
		$(2,4)$ & {\bf 4}& {\bf 1}& {\bf 4}&  OP&{\bf  1}& $\auto{\id}{\g}$ & {\bf 4} & {\bf 1} &$\auto{1}{2}$ & BC & {\bf 4} & {\bf 1} \\ 
		$(3,5)$ & {\bf 18}&2& {\bf 18} & OP & {\bf 1} & $\auto{\id}{\g}$   & {\bf 18} & {\bf 1} &$\auto{1}{3}$ & BC & {\bf 18} & {\bf 1}\\ 
		$(4,6)$ & {\bf 48}  &3& {\bf 48}  & OP & {\bf 1} & $\auto{2}{\id}$   & {\bf 48} & 3 &$\auto{2}{\id}$ & QC & {\bf 48} & {\bf 1}\\ 
		$(5,7)$ & {\bf 100} &4& {\bf 100}  & OP & 4 & $\auto{2}{\id}$   & {\bf 100} & {\bf 1} & $\auto{1}{5}$& BC & {\bf 100} & {\bf 1}\\ 
		$(6,4)$ & 180 & 5&{\bf 4} &  OP & {\bf 1} & $\auto{3}{\id}$ & 12 & {\bf 1} & $\auto{1}{3}$& BC & $-$ & $-$\\ 
		$(7, 9)$ &	{\bf 294} &6&  {\bf 294} & OP$^j$  & { 6} & $\auto{2}{\id}$ & {\bf 294} & { 6} & $\auto{2}{\id}$& QC & {\bf 294} & {\bf 1}\\	
		$(8, 10)$ &	 {\bf 448}  &7 & {\bf 448}  & OP$^j$ & { 7} & $\auto{2}{\id}$  & {\bf 448} & { 7} & $\auto{2}{\id}$& QC& {\bf 448}& {\bf 1}\\
		$(9, 11)$ &	{\bf 648} & 8&{\bf 648}  & ML$^i$ & { 4} & $\auto{\id}{\id}$ &  1144 & 8 & $\auto{2}{\id}$& QC & {\bf 648} & {\bf 1}\\	
		$(10, 5)$	&900 & 9&{\bf 36}  & OP$^i$ & {\bf 1} & $\auto{2}{\g}$  & {\bf 36} & {\bf 1} & $\auto{2}{\id}$& QC & $-$ & $-$\\
		\bottomrule
	\end{tabular}

 \bigskip
 
	\caption{Unbalance results for $\lambda=1$ and minimum  tolerance ($\mathfrak{e}=1$)} \label{tb:l1e1}
	\begin{tabular}{r rlHlH rHrHr}
		\toprule
		Instance & 
		\multicolumn{5}{c}{Integer Programming} & \multicolumn{4}{c}{Heuristic} & AC\\ \cmidrule(l){2-6} \cmidrule(l){7-10} \cmidrule(l){11-11} 		$(s,k)$ &   $\mathfrak u$ & st & $\mathfrak{e}$ & $G$ & met. & $\mathfrak u$ &  $\mathfrak{e}$ & $G$ & met.  &   $\mathfrak u$  \\
		\midrule
		$(2,4)$ &  {\bf 4}  & OP & 1 & $\auto{\id}{\g}$ & IP & {\bf 4} & 1 & $\auto{1}{2}$ & BC & {\bf 4}\\ 
		$(3,5)$ & {\bf 18}  & OP & 1 & $\auto{\id}{\id}$ & IP  & {\bf 18} & 1 & $\auto{1}{3}$ & BC & {\bf 18}\\ 
		$(4,6)$ & {\bf 48} & TL & 1 & $\auto{\id}{\id}$ & IP  & 60 & 1 & $\auto{2}{\id}$ & QC & {\bf 48}\\ 
		$(5,7)$ &  128 & TL & 1 & $\auto{\id}{\id}$ & IP  & {\bf 100} & 1 & $\auto{1}{5}$ & BC & {\bf 100}\\    
		$(6,4)$ &  {\bf 4}  & OP & 1 & $\auto{3}{\id}$ & IP  & 12 & 1 & $\auto{1}{3}$ & BC & $-$ \\    
		$(7, 9)$ &	 576 & TL & 1 & $\auto{2}{\id}$ & IP  & 336 & 1 & $\auto{2}{\id}$ & QC & {\bf 294}\\     
		$(8, 10)$ &	 1022 & TL & 1 & $\auto{2}{\id}$ & IP  & 504 & 1 & $\auto{2}{\id}$ & QC & {\bf 448}\\     
		$(9, 11)$ & 2272 & TL & 1 & $\auto{2}{\id}$ & IP  & 1216 & 1 & $\auto{2}{\id}$ & QC & {\bf 648}\\     
		$(10, 5)$	& {\bf 36} & TL & 1 & $\auto{2}{\g}$ & IP  & {\bf 36} & 1 & $\auto{2}{\id}$ & QC & $-$ \\     
		\bottomrule
\end{tabular}
\end{table}

\subsection{Results for the minimum unbalance} \label{sub:u}

Tables \ref{tb:l1p1} and \ref{tb:l1p2} show the unbalance results for index $\lambda=1$. For $p=1$, we can see that IP produces the best arrays, which improve significantly the trivial constructions when $s$ is not a prime power and match them otherwise. For $p=2$, IP and the algebraic construction---the latter only for $s$ a prime power---produce the best arrays. Although heuristic methods are not optimal, they turn out to be pretty competitive obtaining arrays comparable to those of IP in many cases. 

\begin{table}[tp]
	\centering
	\caption{Minimum unbalance results for $\lambda=2$ and $p=1$ } \label{tb:l2p1}
\begin{tabular}{rrr rlrr rrrH rr}
		\toprule
  Instance & \multicolumn{2}{c}{Bounds} & \multicolumn{4}{c}{Integer Programming} & 
		\multicolumn{4}{c}{Heuristic} &  \multicolumn{2}{c}{AC}   \\ \cmidrule(l){2-3}\cmidrule(l){4-7} \cmidrule(l){8-11} \cmidrule(l){12-13} 
		$(s,k)$ & $\overline{U}$ & $\overline{\epsilon}$ & $\mathfrak u$ & st & $\mathfrak{e}$ & $G$ &   $\mathfrak u$ & $\mathfrak{e}$ & $G$ & met.  & $\mathfrak u$ & $\mathfrak{e}$ \\
		\midrule
		$(2,8)$ & {\bf 8}&{\bf 2} & {\bf 8}& OP & {\bf 2} & $\auto{\id}{\g}$     & 24 & {\bf 2} &$\auto{1}{2}$ & BC& $-$& $-$\\  
		$(3,8)$ & 24 &4& { 18} &  OP & {\bf 1} &  $\auto{\id}{\id}$ & 24 & 4 & $\auto{1}{3}$& BC & {\bf 12} & 2\\ 
		$(4,10)$& 48 &6& {\bf 32} &   OP$^i$ & {\bf 2} &  $\auto{\id}{\id}$  & 156 & 3 & $\auto{2}{\id}$& QC & {\bf 32}& {\bf 2}\\ 
		$(5, 12)$ & 80 &8& {\bf 60} &  OP$^i$ & {\bf 3}  &  $\auto{\id}{\id}$ & 472 & 8 & $\auto{2}{\id}$& QC & {\bf 60} & {\bf 3}\\ 
		$(6, 8)$ &  120 &10& {\bf 48}&  OP$^j$ & {\bf 1} & $\auto{\id}{\id}$  & 140 & 5 &$\auto{2}{\id}$ & QC & $-$& $-$\\ 
		$(7, 16)$ & 168 & 12&{\bf 140} & ML$^i$ & {\bf 5} & $\auto{\id}{\id}$  & 2232 & 6 & $\auto{2}{\id}$& QC& {\bf 140} & {\bf 5}\\ 
		$(8, 18)$ &  224 &14&{\bf 192}  & TL$^i$ & {\bf 6} & $\auto{\id}{\id}$  & 4102 & 7 & $\auto{2}{\id}$& QC& {\bf 192}&{\bf 6} \\ 
		$(9, 20)$ &  {288} & 16&{1244} & TL$^i$ & 10 & $\auto{\id}{\id}$&  6896 & { 8} & $\auto{2}{\id}$& QC & {\bf 252} & {\bf 7} \\ 	
		$(10, 11)$&  {\bf 360} &  18&{3564}&  TL$^j$& 9 & $\auto{2}{\g}$ & { 1386} & {\bf 2} &$\auto{2}{\id}$ & QC& $-$& $-$\\ 
		\bottomrule
	\end{tabular}
\bigskip
 	\caption{Minimum unbalance results for $\lambda=2$ and $p=2$} \label{tb:l2p2}
\begin{tabular}{rrr rlrr rrrH rr}
		\toprule
  Instance & \multicolumn{2}{c}{Bounds} & \multicolumn{4}{c}{Integer Programming} & 
		\multicolumn{4}{c}{Heuristic} &  \multicolumn{2}{c}{AC}   \\ \cmidrule(l){2-3}\cmidrule(l){4-7} \cmidrule(l){8-11} \cmidrule(l){12-13} 
		$(s,k)$ & $\overline{U}$ & $\overline{\epsilon}$ & $\mathfrak u$ & st & $\mathfrak{e}$ & $G$ &   $\mathfrak u$ & $\mathfrak{e}$ & $G$ & met.  & $\mathfrak u$ & $\mathfrak{e}$ \\
		\midrule
		$(2,8)$ & {\bf 16} & 2&{\bf 16} &   OP & {\bf 1} &$\auto{\id}{\id}$&   48 & {\bf 1} & $\auto{1}{2}$& BC&  $-$& $-$\\ 
		$(3,8)$ &  72 & 4&{\bf 18} & TL  & {\bf 1}  & $\auto{\id}{\id}$&  54 & 2 & $\auto{1}{3}$& BC& {\bf 18} & 2\\ 
		$(4,10)$& 192 & 6&{\bf 64} & TL$^i$ & {\bf 2} & $\auto{\id}{\id}$  & 208 & {\bf 2} & $\auto{1}{4}$& BC& {\bf 64}&{\bf 2}\\ 
		$(5, 12)$ & 400 &8& {\bf 150} & ML$^i$ & 3 & $\auto{\id}{\id}$& 640 & {\bf 1} & $\auto{1}{5}$& BC& {\bf 150} & 3\\ 
		$(6, 8)$ & 720 & 10&{\bf 48}  & TL$^i$ & {\bf 1} & $\auto{\id}{\id}$  & 144 & {\bf 1} & $\auto{1}{6}$& BC & $-$ & $-$\\ 
		$(7, 16)$ & 1176 &12& {\bf 490} & OP$^j$ & {\bf 5} & $\auto{\id}{\id}$ & 2958 & 6 &$\auto{2}{\id}$ & QC & {\bf 490} & {\bf 5}\\ 
		$(8, 18)$ & 1792 & 14&{\bf 768}  & OP$^j$ & {\bf 6} & $\auto{\id}{\id}$ & 5362 & 7 &$\auto{2}{\id}$ & QC & {\bf 768} &  {\bf 6}\\ 
		$(9, 20)$ & 2592 &  16&{ 2102}  & TL$^j$ & {\bf 2} &$\auto{\id}{\id}$& 9072 & 8 &$\auto{2}{\id}$ & QC & {\bf 1134} & 7\\	
		$(10, 11)$& 3600 & 18&6668&  OP$^j$& 2 & $\auto{2}{\g}$ &  {\bf 1458} & {\bf 1} & $\auto{2}{\id}$& QC & $-$ & $-$\\ 
		\bottomrule
	\end{tabular}
 \bigskip

\caption{Unbalance results for $\lambda=2$ and minimum tolerance} \label{tb:l2e1} 
 \makebox[\textwidth][c]{
	\begin{tabular}{r rrlrlH rrrrH rrr}
		\toprule
		Instance & 
		\multicolumn{6}{c}{Integer Programming} & \multicolumn{5}{c}{Heuristic} & \multicolumn{3}{c}{AC}\\ \cmidrule(l){2-7} \cmidrule(l){8-12} 	 \cmidrule(l){13-15} 		$(s,k)$ &   $\mathfrak u_1$ & $\mathfrak u_2$ & st & $\mathfrak{e}$ & $G$ & met. & $\mathfrak u_1$ & $\mathfrak u_2$  &  $\mathfrak{e}$ & $G$ & met. & $\mathfrak u_1$ & $\mathfrak u_2$  &  $\mathfrak{e}$ \\
		\midrule
		$(2,8)$ & {\bf  16}  & {\bf 16} & OP & {\bf 1} & $\auto{\id}{\id}$ & IP & 48 & 48 & {\bf 1} & $\auto{1}{2}$ & BC & $-$& $-$& $-$\\ 
		$(3,8)$ & 18 & {\bf 18} & OP & {\bf 1} & $\auto{\id}{\id}$ & IP  & 66 & 66 & {\bf 1} & $\auto{1}{3}$ & BC & {\bf 12}& {\bf 18}& 2\\ 
		$(4,10)$ & 410 & 410 & TL &{\bf 1} & $\auto{1}{\id}$ & IP  & { 224} & { 224} & {\bf 1} & $\auto{1}{4}$ & BC &{\bf 32} &{\bf 64} & 2\\ 
		$(5,12)$ &  160 & 238 & OP$^i$ & 2 & $\auto{\id}{\id}$ & IP  & 640 & 640 & {\bf 1} & $\auto{1}{5}$ & BC & {\bf 60}&{\bf 150} & 3\\    
		$(6,8)$ &  {\bf 48} & {\bf 48} & TL & {\bf 1} & $\auto{\id}{\id}$ & IP  & 144 & 144 & {\bf 1} & $\auto{1}{6}$ & BC & $-$ & $-$ & $-$ \\    
		$(7, 16)$ &	 588 & 880 & TL$^i$ & {\bf 2} & $\auto{\id}{\id}$  & IP  & 2716 & 3094 & {\bf 2} & $\auto{1}{7}$ & BC &{\bf 140} &{\bf 490} & 5\\     
		$(8, 18)$ &	{ 908} & { 1366}&   TL$^i$ & {\bf 2} & $\auto{\id}{\id}$  & IP  & 4912 & 5592& {\bf 2} &  $\auto{1}{8}$ & BC & {\bf 192}&{\bf 768} & 6\\     
		$(9, 20)$ & 1408 & 2102 & TL$^i$ & {\bf 2} & $\auto{\id}{\id}$ & IP  & 8352 & 9540& {\bf 2} &  $\auto{1}{9}$& BC &{\bf 252} & {\bf 1134}& 7\\     
		$(10, 11)$	& 3892 & 6668& OP$^j$ & 2 & $\auto{2}{\g}$   & IP  & {\bf 1458} & {\bf 1458}& {\bf 1} & $\auto{2}{\id}$ & QC & $-$ & $-$ & $-$\\     
		\bottomrule
	\end{tabular}
 }
\end{table}

Tables~\ref{tb:l2p1} and~\ref{tb:l2p2} show unbalance results for index $\lambda=2$. For both $p=1$ and $p=2$, algebraic constructions are the dominating method when $s$ a prime power. In general, IP produces the best search results, but the heuristic method seems to have an advantage for $s=10$. However, as Table~\ref{tb:l2p2} suggests, this is due to the fact that the symmetry group used for the IP search is more restrictive.

Regarding $p=2$ and $s$ prime power, we can see that the results in Tables~\ref{tb:l1p2} and~\ref{tb:l2p2} are optimal due to Corollary~\ref{cor:HamminSimilarityvsUnbalance2}. In Table~\ref{tb:l1p2}, this has further interest as the arrays obtained by IP search have non-trivial symmetries. We note that Corollary~\ref{cor:HamminSimilarityvsUnbalance2} applies to the arrays obtained through algebraic construction beyond the values shown in the table.

\subsection{Results for the minimum tolerance} \label{sub:e}

Tables~\ref{tb:l1e1} and~\ref{tb:l2e1} show the results for the minimum tolerance. 

For the index $\lambda=1$, we have that that all methods produce arrays with the minimum tolerance: $\mathfrak{e}=1$. Because of this, Table~\ref{tb:l1e1} does not show the different tolerances. For $s$ a prime power, the algebraic construction produces the best unbalances dominating over the other two methodologies. Moreover, by Corollary~\ref{cor:HamminSimilarityvsUnbalance2}, these unbalances are optimal. For $s$ not a prime power, IP produces the best results.

For the index $\lambda=2$, the heuristic methodology obtains the best tolerance. Since the optimal tolerance is not always one, Table~\ref{tb:l2e1} indicates both the $1$-unbalance and $2$-unbalance of the best tolerance array. In this way, we can see how the best unbalances are not necessarily achieved by the best tolerance AOAs.

\begin{table}[tp]
	\centering
	\caption{Comparison results for instances in \cite{ke2015}, \cite{UDrepository} and \cite{ma2000new}} \label{tab:make}
	\begin{tabular}{rrrH HHHr rrr rrr Hl}
		\toprule
  $s$ & $k$ &  $\lambda$ & $p$ &  $\mathfrak u_1$ &  $\mathfrak{e}$  &   $\mathfrak u_2$ &  $\mathfrak{e}$ & $\mathcal{D}_f$ & $D_1$ & $D_2$ & $CD$ & $WD$ & $MD$ & $OA(1)$  & Met. \\
  \midrule
   3 & 8 & 1 & 1 & 68 & 2 & 72 & 2 & 0.8409 & 2.4286 & {\bf 2.5714} & {\bf 0.2242} & {\bf 1.5638} & {\bf 4.2253} & 0  & Ke \cite{ke2015}\\ 
   &  &  & 1 & 66 & 2 & 72 & 2 &0.0084& 2.3571 & {\bf 2.5714} & 0.2351 & {\bf 1.5638} & 4.2667 & 0  & CD \cite{UDrepository}\\ 
   &  &  & 1 & 66 & 2 & 72 & 2 &0.0084 & 2.3571 & {\bf 2.5714} & 0.2351 & {\bf 1.5638} & 4.2667 & 0  & WD \cite{UDrepository}\\ \cmidrule(l){4-16} 
       &  &  & 1 & 72 & 1 & 72 & {\bf 1} &  0.0056 & 2.5714 & {\bf 2.5714} & 0.2392 & {\bf 1.5638} & 4.2908 & 0  &IP\\ 
   &  &  & 1 & 48 & 2 & 72 & 2 & {\bf 0.8660} & {\bf 1.7143} & {\bf 2.5714} & 0.2266 & {\bf 1.5638} & 4.2282 & 0   &IP \\
   &  &  & 1 & 48 & 2 & 72 & 2  &0.0000 & {\bf 1.7143} & {\bf 2.5714} & 0.2443 & {\bf 1.5638} & 4.2936 & 0 &QC\\ 
  \midrule \midrule
  6 & 7 & 1 & 1 & 140 & 2 & 144 & 2 & 0.9775& 6.6667 & 6.8571 & 0.0419 & 0.2206 & 0.4674 & 0  &Ma \cite{ma2000new}\\ 
   &  &  & 1 & 276 & 2 & 286 & 2  &  {\bf 0.9987}& 13.1429 & 13.6190 & {\bf 0.0359} & 0.2247 & 0.4548 & 0& CD \cite{UDrepository}\\  \cmidrule(l){4-16} 
   &  &  & 1 & 100 & 1 & 100 & {\bf 1} &  0.9913& {\bf 4.7619} & {\bf 4.7619} & 0.0391 & {\bf 0.2154} & {\bf 0.4507} & 0 & IP\\  
   &  &  & 1 & 100 & 1 & 100 & {\bf 1} &0.9935& {\bf 4.7619} & {\bf 4.7619} & 0.0402 & 0.2209 & 0.4612 & 0  & QC\\ 
   \bottomrule
  \end{tabular}

  \bigskip

  \caption{Comparison results for web uniform designs \cite{UDrepository} and $\lambda=1$} \label{tab:udl1}
	\begin{tabular}{rrrH HHHr rrr rrr Hl}
		\toprule
  $s$ & $k$ &  $\lambda$ & $p$ &  $\mathfrak u_1$ &  $\mathfrak{e}$  &   $\mathfrak u_2$ &  $\mathfrak{e}$ &$\mathcal{D}_f$  & $D_1$ & $D_2$ & $CD$ & $WD$ & $MD$ & $OA(1)$ & Met. \\
  \midrule
   3 & 5 & 1 & 1 & 20 & 1 & 20 & {\bf 1} &{\bf 0.9747}& 2.0000 & 2.0000 & {\bf 0.0773} & 0.3412 & {\bf 0.5246} & 0  & CD \cite{UDrepository}\\
   &  & 1 & 1 & 16 & 2 & 18 & 2 & 0.8027& 1.6000 & {\bf 1.8000} & 0.0841 & {\bf 0.3386} & 0.5323 & 0 & WD \cite{UDrepository}\\  \cmidrule(l){3-16} 
   &  & 1 & 1 & 18 & 1 & 18 & {\bf 1} & 0.9642& { 1.8000} & {\bf 1.8000} & 0.0814 & {\bf 0.3386} & 0.5298 & 0 & IP\\ 
   &  & 1 & 1 & 12 & 2 & 18 & 2 & 0.9441  & {\bf 1.2000} & {\bf 1.8000} & 0.0817 & {\bf 0.3386} & 0.5316 & 0 & IP\\ 
    &  & 1 & 1 & 12 & 2 & 18 & 2 & 0.0000 & {\bf 1.2000} & {\bf 1.8000} & 0.0869 & {\bf 0.3386} & 0.5375 & 0&BC\\ 
       &  & 1 & 1 & 18 & 1 & 18 & {\bf 1} & 0.8706 & 1.8000 &  {\bf 1.8000} & 0.0841 & {\bf 0.3386} & 0.5341 & 0 &AC\\ 
  \midrule  \midrule
4 & 6 & 1 & 1 & 48 & 1 & 48 & {\bf 1} & {\bf  {1.0000}}& 3.2000 & {\bf 3.2000} & {\bf 0.0603} & 0.3243 & 0.5439 & 0   & CD \cite{UDrepository}\\ 
   &  & 1 & 1 & 48 & 1 & 48 & {\bf 1} & 0.9635& 3.2000 & {\bf 3.2000} & 0.0652 & {\bf 0.3021} & {\bf 0.5371} & 0 & WD \cite{UDrepository}\\  \cmidrule(l){3-16}
   &  & 1 & 1 & 48 & 1 & 48 & {\bf 1} & 0.9184& 3.2000 & {\bf 3.2000} & 0.0703 & 0.3129 & 0.5599 & 0  & IP\\ 
   &  & 1 & 1 & 24 & 3 & 48 & 3  & 0.9283& {\bf 1.6000} & {\bf 3.2000} & 0.0697 & 0.3145 & 0.5602 & 0 & IP\\ 
    &  & 1 & 1 & 24 & 3 & 48 & 3 & 0.0000& {\bf 1.6000} & {\bf 3.2000} & 0.0771 & 0.3147 & 0.5690 & 0 & QC\\
      &  & 1 & 1 & 48 & 1 & 48 & {\bf 1} & 0.8270 & 3.2000 &  {\bf 3.2000} & 0.0751 & 0.3143 & 0.5651 & 0 & AC\\ 
 \midrule \midrule
  5 & 7 & 1 & 1 & 180 & 2 & 186 & {\bf 2} &0.9956& 8.5714 & 8.8571 & {\bf 0.0516} & 0.3293 & {\bf 0.6845} & 0  & CD \cite{UDrepository}\\  \cmidrule(l){3-16}
 %
 %
   &  & 1 & 1 & 40 & 4 & 100 & 4  & 0.9754& {\bf 1.9048} & {\bf 4.7619} & 0.0583 & {\bf 0.3217} & 0.6952 & 0 & IP\\ 
   &  & 1 & 1 & 40 & 4 & 100 & 4 &{\bf {1.0000}}& {\bf 1.9048} & {\bf 4.7619} & 0.0582 & 0.3279 & 0.7070 & 0&   BC\\
       &  & 1 & 1 & 100 & 1 & 100 & {\bf 1} & 0.9686 & 4.7619 & {\bf 4.7619} & 0.0612 & 0.3259 & 0.7057 & 0 & AC\\ 
 \midrule \midrule
   6 & 4 & 1 & 1 & 16 & 1 & 16 & {\bf 1} &{\bf 0.9999}& 2.6667 & 2.6667 & {\bf 0.0124} & 0.0454 & {\bf 0.0567} & 0  & CD \cite{UDrepository}\\ \cmidrule(l){3-16}
   &  & 1 & 1 & 4 & 1 & 4 & {\bf 1} &{\bf 0.9999}& {\bf 0.6667} & {\bf 0.6667} & 0.0135 & {\bf 0.0452} & 0.0571 & 0 &  IP \\
   &  & 1 & 1 & 12 & 1 & 12 & {\bf 1} & 0.9948& 2.0000 & 2.0000 & 0.0131 & 0.0459 & 0.0578 & 0  &BC\\
      \bottomrule
  \end{tabular}
\end{table}

   \begin{table}[tp]
	\centering
	
  \end{table}

   \begin{table}[tp]
	\centering
	\caption{Comparison results for web uniform designs \cite{UDrepository} and $\lambda=2$} \label{tab:udl2}
    \begin{tabular}{rrrH HHHr rrr rrr Hl}
		\toprule
  $s$ & $k$ &  $\lambda$ & $p$ &  $\mathfrak u_1$ &  $\mathfrak{e}$  &   $\mathfrak u_2$ &  $\mathfrak{e}$ &$\mathcal{D}_f$ & $D_1$ & $D_2$ & $CD$ & $WD$ & $MD$ & $OA(1)$ & Met. \\
  \midrule \midrule
  3 & 8 & 2 & 1 & 16 & 2 & 18 & 2 & {\bf 0.9974}& 0.5714 & \bf{0.6429} & \bf{0.1492} & \bf{1.2517} &\bf{3.3878} & 0 &CD \cite{UDrepository}\\ 
   &  & 2 & 1 & 16 & 2 & 18 & 2 &0.9947& 0.5714 & \bf{0.6429} & 0.1522 & \bf{1.2517} & 3.3936 & 0  &WD \cite{UDrepository}\\   \cmidrule(l){3-16}
   &  & 2 & 1 & 18 & 1 & 18 & \bf{1} &0.9744 & {0.6429} & \bf{0.6429} & 0.1545 & \bf{1.2517} & 3.4052 & 0 &IP\\ 
    &  & 2 & 1 & 24 & 4 & 72 & 4  &0.9647& 0.8571 & 2.5714 & 0.1610 & 1.2962 & 3.5160 & 0  &BC\\ 
      &  & 2 & 1 & 66 & 1 & 66 & \bf{1} &0.9322& 2.3571 & 2.3571 & 0.1614 & 1.2891 & 3.5009 & 0  &BC\\
    &  & 2 & 1 & 12 & 2 & 18 & 2 & 0.9920 &\bf{0.4286} & \bf{0.6429} & 0.1532 & \bf{1.2517} & 3.3957 & 0 & AC \\ 
  \midrule \midrule
   4 & 10 & 2 & 1 & 272 & 2 & 310 & 2 & {\bf 0.9944}& 6.0444 & 6.8889 & \bf{0.1588} & 1.7177 & 5.8851 & 0   &CD \cite{UDrepository}\\ 
   &  & 2 & 1 & 206 & 2 & 210 & 2 &0.9850& 4.5778 & {4.6667} & 0.1693 & {1.6747} & {5.8821} & 0  &WD \cite{UDrepository}\\  \cmidrule(l){3-16}
   &  & 2 & 1 & 410 & 1 & 410 & \bf{1} &0.9255& 9.1111 & 9.1111 & 0.1907 & 1.7750 & 6.3064 & 0 &IP \\
 &  & 2 & 1 & 32 & 2 & 64 & 2  & 0.9906& {\bf 0.7111} & {\bf 1.4222} & 0.1625 & {1.6286} & {\bf 5.7027} & 0 & IP \\ 
  &  & 2 & 1 & 156 & 3 & 222 & 3 &0.8998& {3.4667} & 4.9333 & 0.2052 & 1.7113 & 6.1255 & 0  &QC\\ 
   &  & 2 & 1 & 224 & 1 & 224 & \bf{1} &0.9441& 4.9778 & 4.9778 & 0.1732 & 1.6993 & 5.9848 & 0  &BC\\ 
      &  & 2 & 1 & 32 & 2 & 64 & 2 & 0.9827 & \bf{0.7111} & \bf{1.4222} & 0.1668 & \bf{1.6260} & 5.7239 & 0 &AC\\ 
  \midrule \midrule
   5 & 12 & 2 & 1 & 944 & 3 & 1166 & 3 & {\bf 0.9965}& 14.3030 & 17.6667 & \bf{0.1516} & 2.5771 & 12.6809 & 0  &CD \cite{UDrepository}\\  \cmidrule(l){3-16}
   &  & 2 & 1 & 60 & 3 & 150 & 3&0.9958 & \bf{0.9091} & \bf{2.2727} & 0.1621 & \bf{2.3661} & {12.1654} & 0  &IP\\
      &  & 2 & 1 & 472 & 8 & 976 & 8 &0.9527& 7.1515 & 14.7879 & 0.2162 & 2.6099 & 13.6464 & 0 &QC\\ 
   &  & 2 & 1 & 640 & 1 & 640 & \bf{1} &0.9533& 9.6970 & 9.6970 & 0.1766 & 2.5171 & 13.0467 & 0   &BC\\ 
       &  & 2 & 1 & 60 & 3 & 150 & 3 & 0.9946 & \bf{0.9091} & \bf{2.2727} & 0.1601 & \bf{2.3661} & \bf{12.1637} & 0 &AC \\ 
  \midrule \midrule
   7 & 16 & 2 & 1 & 4268 & 4 & 5786 & \bf{4} &0.9978 & 35.5667 & 48.2167 & \bf{0.2008} & 6.7195 & 62.7258 & 0  &CD \cite{UDrepository}\\  \cmidrule(l){3-16}
   &  & 2 & 1 & 140 & 5 & 490 & 5 &0.9921& \bf{1.1667} & \bf{4.0833} & 0.2433 & \bf{6.1279} & \bf{59.9444} & 0  &IP\\
    &  & 2 & 1 & 2232 & 6 & 3396 & 6 &0.9536& 18.6000 & 28.3000 & 0.3330 & 6.8577 & 69.1272 & 0  &QC\\ 
      &  & 2 & 1 & 140 & 5 & 490 & 5 & \bf{0.9990} & \bf{1.1667} & \bf{4.0833} & 0.2305 & 6.1482 & 60.3086 & 0 &AC \\ 
      \bottomrule
  \end{tabular}
\end{table}

\begin{figure}[t]
\centering
\begin{tabular}{r c c}
&$\widehat\Unb_{1,2}$ & $\widehat\Unb_{2,2}$\\
\rotatebox[origin=c]{90}{$\lambda=1$}&\adjincludegraphics[valign=M,width=0.45\textwidth]{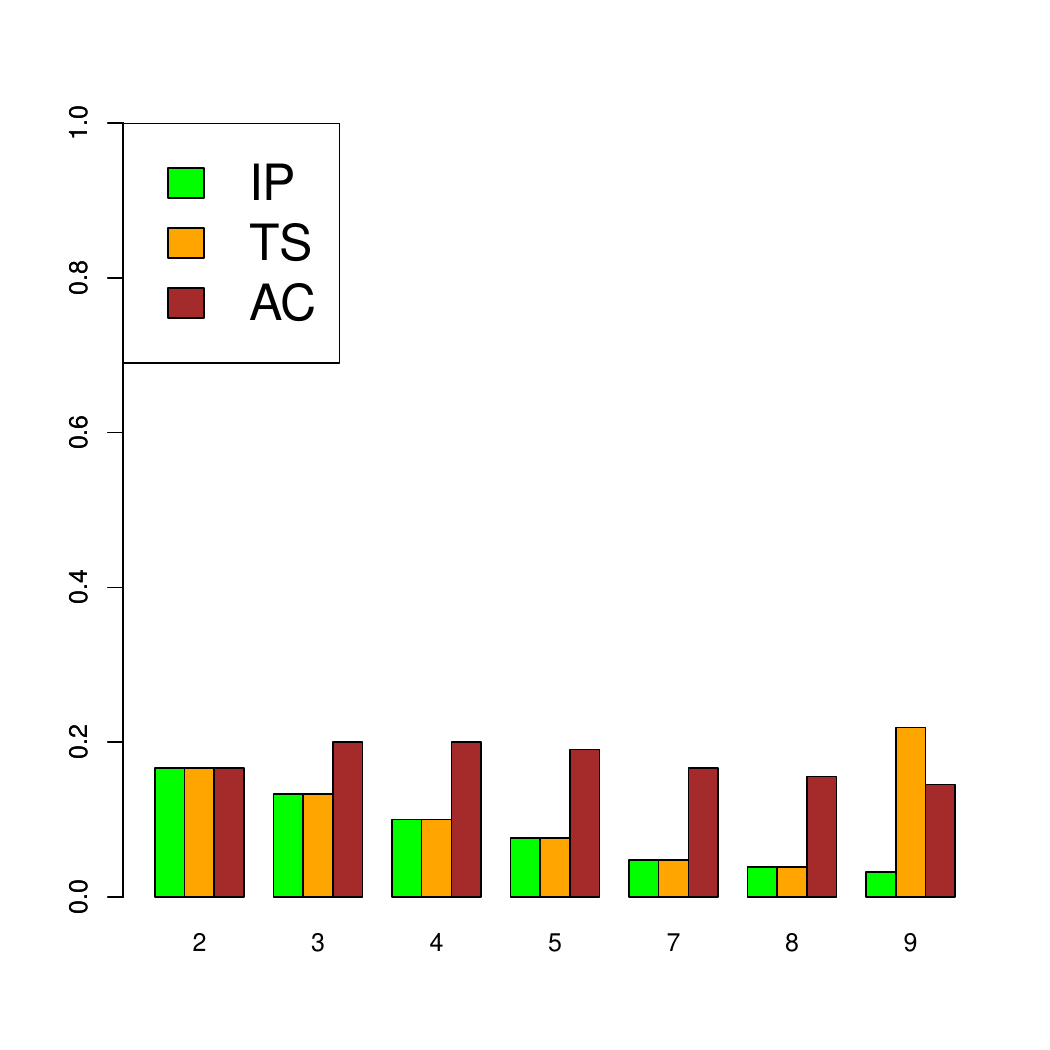} &\adjincludegraphics[valign=M,width=0.45\textwidth]{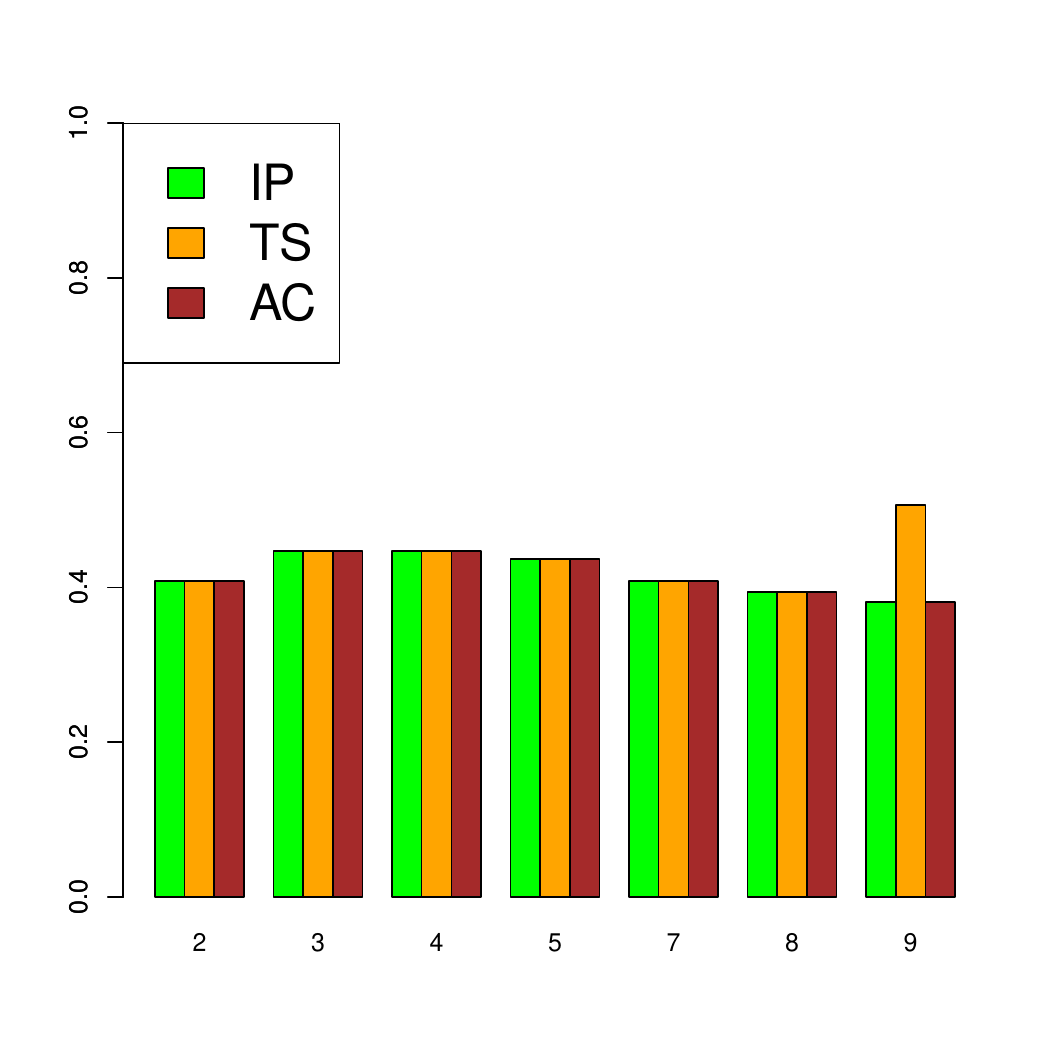}\\
\rotatebox[origin=c]{90}{$\lambda=2$}&\adjincludegraphics[valign=M,width=0.45\textwidth]{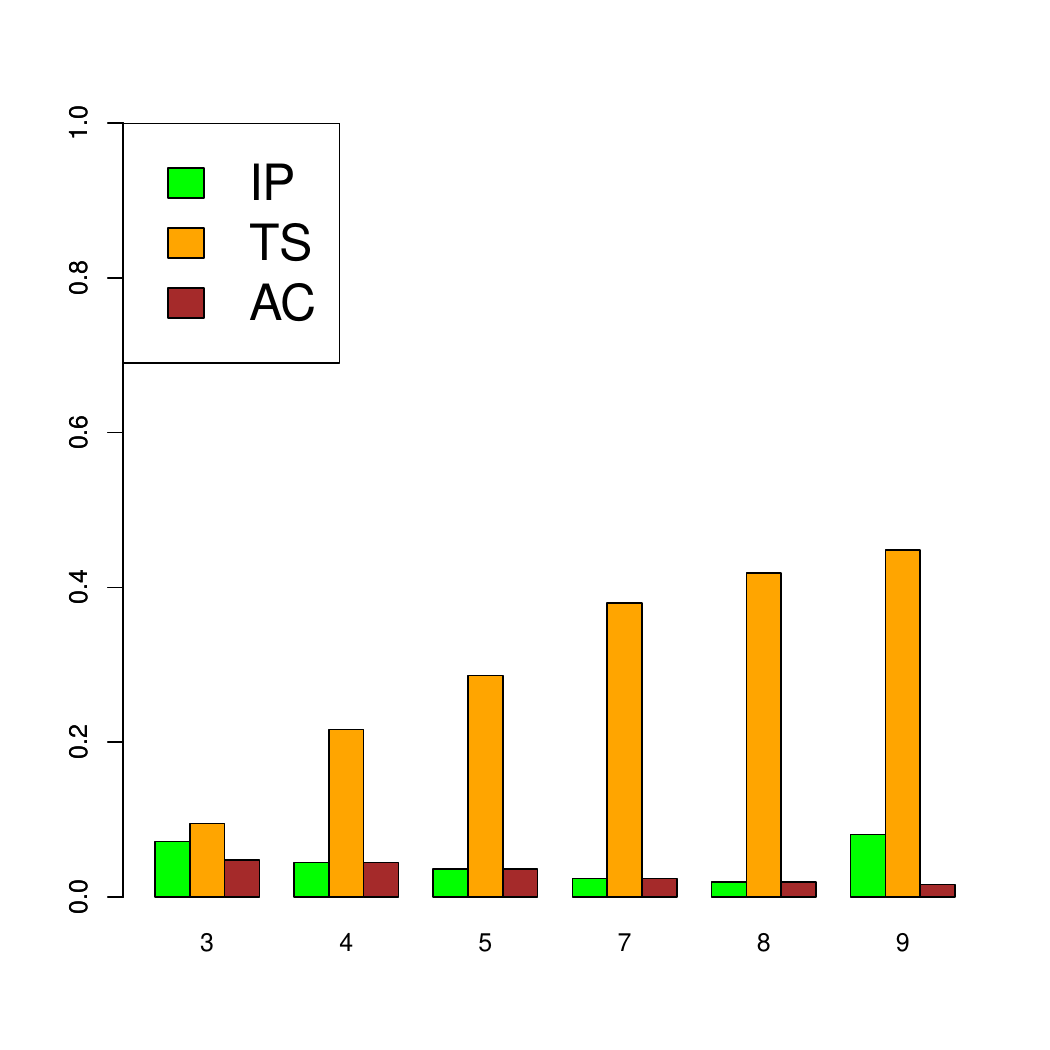} &\adjincludegraphics[valign=M,width=0.45\textwidth]{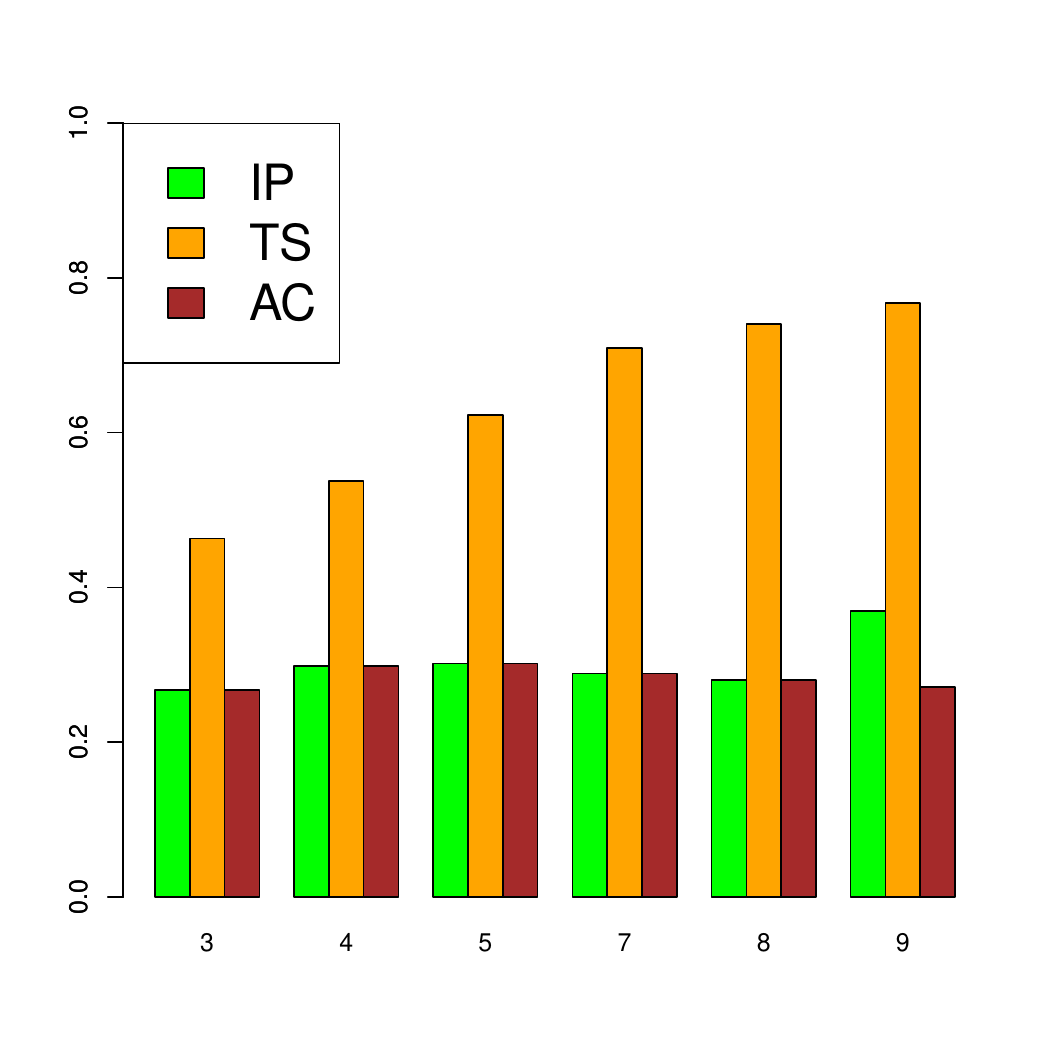}
 \end{tabular}
 
\caption{Best normalized unbalances $\widehat\Unb_{1,2}$ (left) and $\widehat\Unb_{2,2}$ (right, see~\eqref{eq:normunbtol}) by $\lambda=1$ (top) and $\lambda=2$ (bottom)}
    \label{fig:l1l2norm}
\end{figure}

\subsection{Comparison with scientific literature} \label{sec:sl}

In Tables~\ref{tab:make}, \ref{tab:udl1} and \ref{tab:udl2}, we show the comparison of our results with the ones in the literature. Tables \ref{tab:udl1} and \ref{tab:udl2} deal with the cases initially considered by us, while Table~\ref{tab:make} handles two cases not initially considered by us.

In general, our methods outperform the existing results for $D_1$ and $D_2$. This fact is not surprising, since $D_1$ and $D_2$ are just scaled version of the unbalance (see equation above for $D_1$ and $D_2$), which we are trying to optimize in our search for AOAs. Moreover, with the exception of $s=6$, Corollary~\ref{cor:HamminSimilarityvsUnbalance2} in Appendix~\ref{sec:unbalancebounds} guarantees that the obtained arrays do not only have the best $D_2$, but the optimal $D_2$.

Regarding the $\mcD$-value (Definition~\ref{defi:Dvalue}), which we consider for $f(a)=a-\frac{s+1}{2}$, and the discrepancy measures CD, WD and MD (Definition~\ref{defi:CDWDMD}), our methods (particularly IP and the algebraic constructions) produce results almost identical to those of the literature, even if they might not outperform them. For the discrepancy measures, Theorem~\ref{thm:discrepancyvsunbalance} together with Corollary~\ref{cor:discrepancyvsunbalance}, in Appendix~\ref{subsec:quadrature} provide an explanation of this phenomenon. Corollary~\ref{cor:discrepancyvsunbalance} guarantee that our arrays have optimal discrete discrepancy and Theorem~\ref{thm:discrepancyvsunbalance} shows how this optimal discrete discrepancy leads to good values of CD, WD and MD. Moreover, in the case of $s=3$ and WD, this fully justify why the arrays with best optimal $D_2$ have also optimal WD.
  
\subsection{Discussion} \label{sec:dis}

Figure \ref{fig:l1l2norm} shows the comparable best results obtained with the integer programming (green), heuristics (orange) and algebraic construction (red) for the normalized unbalance (see~\eqref{eq:normunbtol} for the precise formula) for $p=1$ (left) and $p=2$ (right), when $\lambda=1$ (top) and $\lambda=2$ (bottom), respectively. We use the normalized unbalance of~\eqref{eq:normunbtol} as it allows not only comparisons of the methods within the same parameters, but also across different choices of the parameters. 

In general, the IP and the algebraic construction (with the exception of $\lambda=1$ and $p=1$) produce the arrays with the best normalized unbalances. However, in terms of scalability, the algebraic constructions outperform the IP search; therefore, we consider that exploring new algebraic constructions might be a promising direction for the future. However, we should not ignore that the heuristic methodology produced the best tolerances for $\lambda=2$. In this sense, each methodology seems to have an advantage at hand depending on what we want to optimize when searching/constructing arrays.

\section{Integer Programming for Almost-Orthogonal Arrays} \label{sec:integerprogramming}

We present an integer programming (IP) formulation to construct AOAs. IP methods have been successfully applied to find OAs~\cite{rosenber1995,carlini2007,bulutoglu2008,vieira2011,geyer2014,geyer2019}, particularly using symmetry-exploiting techniques~\cite{margot2002, margot2003a, margot2003b, margot2007}. 
For classical references regarding IP, see \cite{luenberger,minoux, winston, wolsey1999} and references therein. 
Our formulation is implemented in IBM ILOG CPLEX 20.1 \cite{IBMcplex} with up to 8 threads.

\subsection{IP model}

We model the Minimum Unbalance Problem (MUP), whose optimal solution is an array $A \in S^{N \times k}$ forming an $AOA(N,k,s,t,\epsilon)$ with $N=\lambda s^t$, $k$ factors, $s$ levels, strength $t=2$, index $\lambda$, and tolerance $\epsilon$. We assume that the first $t$ columns contain all $\lambda s^t$ ordered combinations. The indices are the following ones:
\begin{itemize}
   \item $i \in \I=\{1,\ldots,N\}$: runs,
   \item $j \in \J=\{t+1,\ldots,k\}$: factors,
   \item $c \in \mcC=\{1,\ldots, \binom{k-t}{t}\}$: $t$-tuples of columns,
   \item $m \in \M=\{1,\ldots,s\}$: levels,
   \item $l \in \mcL=\{1,\ldots,s^t\}$: lexicographic combinations.
\end{itemize}
And the variables are the following ones:
\begin{itemize}
\item $x_{i,j}^m \in \{0,1\}$: equals 1 if entry $(i,j)$ takes value $m$;
\item $z_{i,c}^l \in \{0,1\}$: equals 1 if row $i$ in $t$-tuple $c$ corresponds to combination $l$;
\item $\delta_c^{0,l} \in \mathbb{Z}$: deviation of the frequency of combination $l$ in $c$ from $\lambda$;
\item $\delta^{1,m} \in \mathbb{Z}$: deviation from level balance ($\lambda s$ occurrences);
\item $\delta_{m',j}^{2,m}, \delta_{m',j}^{3,m} \in \mathbb{Z}$: deviations from pairwise balance constraints for column pairs $(1,j)$ and $(2,j)$.
\end{itemize}
Then, the mathematical modelling of MUP, with $\Oh(s^4 k^2)$ variables 
is as follows:  

\begin{subequations}\label{eq:AOA}
	\begin{align}
	  \min & \displaystyle \sum\limits_{c \in \mcC,  l \in \mcL} |\delta_{c}^{0,l}|^p + \sum\limits_{m \in \M} |\delta^{1,m}|^p  + \sum\limits_{m,m' \in \M,  j \in \J} |\delta_{m',j}^{2,m}|^p +  |\delta_{m',j}^{3,m}|^p  \label{eq:fo} \\
	s.t. ~ & \sum_{i \in \I} x_{i,j}^m = \lambda ~ s,  \quad  \forall j\in \J\setminus\{|\J|\},~ m \in \M \label{eq:aoa1}\\
  & \sum_{i\in \I} x_{i,|\J|}^m = \lambda ~ s + \delta^{1,m},  \quad \forall m \in \M \label{eq:aoa1k}\\
	& \sum_{m \in \M} x_{i,j}^m =1,  \quad \forall i \in \I,~ j\in \J \label{eq:aoa2}\\
	&  \sum_{\ell=1}^\lambda \sum_{i=(\ell-1)s^2+(m'-1)s+1}^{(\ell-1)s^2+m's} x_{i,j}^m =\lambda + \delta^{2,m}_{m',j},  \quad  \forall  j\in \J,~ m,m' \in \M \label{eq:aoa31}\\
	& \sum_{\ell=1}^\lambda \sum_{i\in \I:i\equiv(\ell-1)s^2+ m'\pmod s} x_{i,j}^m =\lambda + \delta^{3,m}_{m',j},  \quad \forall j \in \J,~ m,m' \in \M \label{eq:aoa32}\\
	& \sum_{l \in \mcL} l ~ z_{i,c}^l =s \sum_{m \in \M} m ~x_{i,j_1}^m + \sum_{m  \in \M} m ~x_{i,j_2}^m-s, \quad  \forall i \in \I,~c=(j_1,j_2) \in \mcC  \label{eq:aoaz1}\\	
	& \sum_{l \in \mcL} z_{i,c}^l =1,  \quad  \forall i \in \I,~c \in \mcC  \label{eq:aoaz2}\\	
	& \sum_{i \in \I} z_{i,c}^l =\lambda+\delta_{c}^{0,l}, \quad  \forall l \in \mcL,~c \in \mcC  \label{eq:aoaz3}\\		
	&  x_{i,j}^m, ~ z_{i,c}^l \in\{0,1\},   \quad \forall i \in \I,~ j\in \J,~c \in \mcC,~l\in \mcL,~m\in \M  \label{eq:aoavar} \\
& \delta_{c}^{0,l},\delta_{m',j}^{2,m},\delta_{m',j}^{3,m} \in \{\max(-\lambda,-\epsilon),\ldots,\epsilon\},      \quad \forall j\in \J,~c \in \mcC,~l\in \mcL,~m,m'\in \M \label{eq:aoaint}\\
& \delta^{1,m} \in \{-\lambda s,\ldots,\lambda s^t - \lambda s \},     \quad \forall m\in \M. \label{eq:aoaint1}
	\end{align}
\end{subequations} 
The objective function \eqref{eq:fo} minimizes the $p$-unbalance of $A$, where $$U_{p,2}(A)=\displaystyle \sum\limits_{c \in \mcC,  l \in \mcL} |\delta_{c}^{0,l}|^p + \sum\limits_{m,m' \in \M,  j \in \J} |\delta_{m',j}^{2,m}|^p +  |\delta_{m',j}^{3,m}|^p.$$ 
The variables $\delta_{c}^{0,l}$, $\delta_{m',j}^{2,m}$ and $\delta_{m',j}^{3,m}$ represent deviations $n(A,x,c)-\lambda$, where $x$ is the $l$-th combination and $c$ is the corresponding $t$-tuple column.
Constraints \eqref{eq:aoa1} impose level balance ($\lambda s$ occurrences per column), except for the last column, relaxed via $\delta^{1,m}$ in \eqref{eq:aoa1k}. Constraints \eqref{eq:aoa2} ensure a unique level per cell. Constraints \eqref{eq:aoa31}–\eqref{eq:aoa32} enforce pairwise balance for $(1,j)$ and $(2,j)$ up to deviations $\delta_{m',j}^{2,m}$ and $\delta_{m',j}^{3,m}$.
Constraints \eqref{eq:aoaz1}–\eqref{eq:aoaz2} assign exactly one combination of $S^t$ to each row and column pair, while \eqref{eq:aoaz3} bounds their frequencies by $\lambda+\epsilon$. Constraints \eqref{eq:aoavar}–\eqref{eq:aoaint1} define binary and integrality conditions.

Once solved, $A$ is recovered as $a_{i,j}=\sum_m m,x_{i,j}^m$. An optimal value of 0 yields an OA; otherwise, an optimal AOA or an upper bound on $\mathfrak{u}(\lambda s^2,k,2,p)$ is obtained. The MUP is linear for $p=1$ and quadratic for $p=2$.

\subsection{Exploiting symmetry} \label{sec:symip}

To reduce the search space, we restrict to symmetric arrays (Subsection~\ref{subsec:symmetries}) by adding constraints \eqref{eq:sim1}–\eqref{eq:sim3}, which enforce the automorphism $\auto{(\overline{m},\ldots,s)}{\id}$. This fixes symbols $\{1,\ldots,\overline{m}-1\}$ and cyclically permutes $\{\overline{m},\ldots,s\}$. For each run $i$, $\sigma_{\overline{m}}(i)$ denotes the mapped row.
Special cases include: $\overline{m}=1$ (full cycle), $\overline{m}=2$ (difference matrices \cite{montoya2022}), $\overline{m}=s-1$ (minimal cycle), and $\overline{m}=s$ (no symmetry).



\begin{subequations}\label{eq:AOAs}
	\begin{align}
		& x_{i,j}^{m}= x_{\sigma_{\overline{m}}(i),j}^{m}, &\forall i,j,~\forall m \in \{1,\ldots,\overline{m}-1\} \label{eq:sim1}\\
		& x_{i,j}^m= x_{\sigma_{\overline{m}}(i),j}^{m+1}, &\forall i,j,~\forall m\in \{\overline{m},\ldots,s-1\}\label{eq:sim2}\\		
		& x_{i,j}^{s}= x_{\sigma_{\overline{m}}(i),j}^{\overline{m}}, &\forall i,j \label{eq:sim3}		
	\end{align}
\end{subequations} 


For instances with $k\ge 4$, we may also consider the Klein automorphism $\auto{\id}{(1,2)(3,4)}$. 
Constraints \eqref{eq:sim03}–\eqref{eq:sim04} enforce this symmetry, where $\sigma_0(i)$ denotes the image of run $i$.

\begin{subequations}\label{eq:AOAs0}
	\begin{align}
		& x_{i,3}^{m}= x_{\sigma_0(i),4}^{m}, &\forall i,m  \label{eq:sim03}\\
		& x_{i,4}^{m}= x_{\sigma_0(i),3}^{m}, &\forall i,m  \label{eq:sim04}\\
			& x_{i,k}^{m}= x_{\sigma_0(i),k}^{m}, &\forall i,m,~ \forall k>4  \label{eq:sim034}
	\end{align}
\end{subequations} 


Section \ref{sec:overview} reports the resulting AOAs obtained with this IP approach, with the applied symmetries detailed in the $G$ columns of the tables.

\section{Heuristic Methods for Almost-Orthogonal Arrays} \label{sec:heuristic}

Heuristic methods are uncommon for the search of OAs, being the only exception the use of genetic algorithms for the search for binary orthogonal arrays~\cite{mariot2018evolutionary}. In the context of AOAs, heuristic methods are more common, since to successfully solve the underlying combinatorial optimization problem we don't need to fully minimize the objective function. In this setting, the so-called threshold-accepting heuristics of Ma, Fang and Liski~\cite{ma2000new} is the more common one, having been used extensively in the context of quadrature discrepancy measures~\cite{UDpaper1,UDpaper2,UDpaper3,fang2005,ke2015}.

In our search, we use local search meta-heuristics~\cite{kaveh2018new,kaveh2014advances,kaveh2018meta} based on a local version of Pareto fronts. This leads to a whole family of heuristic algorithms, which are competitive in many of the considered cases. 

\subsection{Local Pareto Fronts}

In a metric space $(\bbX,\dist)$, we want to optimize (`minimize') the multi-objective continuous cost function
\[
f:\bbX\rightarrow \bbR^m.
\]
This problem induces a partial order $\preccurlyeq_f$ on $\bbX$ given by the component-wise partial order of the $f$-values, i.e.,
\begin{equation}
    x\preccurlyeq_f y\text{ if and only if }f(x)\leq f(y)
\end{equation}
where the right-hand side `$\leq$' is the component-wise order which satisfies $f(x)\leq f(y)$ if and only if for all $i$, $f_i(x)\leq f_i(y)$. 

In general, finding the $f$-optimal solution to this problem is the same as finding the minimum element in $\bbX$ with respect $\preccurlyeq_f$. However, in general, we cannot expect such a minimum to exist. The existence of trade-offs between the different components of $f$ leads to the definition of Pareto front.

\begin{dfn}
A \emph{Pareto front} of $f$ in $\bbX$ is a set $F\subset \bbX$ such that for all $x\in \bbX$, there is $x_\ast\in F$ such that $x_\ast\preccurlyeq_f x$, i.e., $f(x_{\ast})\leq f(x)$.
\end{dfn}

Unfortunately, being a Pareto set is a global property. To make it local, we only require for $F$ to be a Pareto front in its \emph{$r$-neighborhood} $\mcU(F,r):=\{x\in\bbX\mid \dist(x,F)\leq r\}$. 

\begin{dfn}
A \emph{local Pareto front} of $f$ in $(\bbX,\dist)$ with radius $r$ is a set $F\subset \bbX$ such that $F$ is a Pareto front of $f$ in its $r$-neighborhood $\mcU(F,r)$.
\end{dfn}

Using the notion of local Pareto fronts, we can design heuristic algorithms for finding local Pareto fronts with radius $r$ in discrete spaces as follows:

\begin{algorithm}[h]
\caption{Local Pareto Front Heuristic}\label{alg:heuristic}
\DontPrintSemicolon 
\Repeat{$F$ unchanged}{
\While{$F$ unchanged}{
take $u\in \mcU(F,r)$\tcp*{We loop over the elements of $\mcU(F,r)$}
\uIf{for all $x\in F$, $f(x)\leq f(u)$}{
Do nothing\tcp*{$u$ is not a new minimal element}
}
\Else{
$F\leftarrow F\cup \{u\}$\tcp*{$u$ is a new minimal element, add it}
\tcc{Removal of non-minimal elements after adding $u$}
\For{$x\in F$}{
\If{$f(u)\leq f(x)$}{
$F\leftarrow F\setminus \{x\}$ \tcp*{$x$ is not minimal anymore, remove it}
}
}
}
}
}
\end{algorithm}

In many practical situations, $\bbX$ will take the form of the space of sequences of length $\ell$ over a certain finite alphabet $S$ and $\dist$ will be the Hamming distance $\Hamming$ given by the number of positions that differ. In this way, looping over $\mcU(F,r)$ will be the same as looping over those sequences obtained from those of $F$ with at most $r$ changes. It is common then to loop first over all those sequences obtained with one change, then over all sequences obtained with two changes, and so on.

However, we should not assume that the setting of sequences above is the only one with practical interest. For example, we might consider as our search space a (weighted) graph with its corresponding geodesic metric.

\subsection{Local Pareto Fronts in the context of AOAs}

In the setting of AOAs, we will search over the space of $N\times k$ arrays with symbols in the finite set $S$ of size $s$, $S^{N\times k}$, using as distance the Hamming distance---recall that, in some sense, an array is just a sequence written in the form of a table. We will consider mainly the multi-objective cost function
\begin{equation}
   A\mapsto (\Unb_{p,2}(A),\Tol_2(A)),
\end{equation}
for $p=1$ and $p=2$, constructing this way a local Pareto front for the $p$-unbalance and the tolerance. To initialize the algorithm, we use a singleton with a randomly generated array.

A priori, we could consider our heuristic algorithms for any $r$. However, from our practical computations, we have that these algorithms are not competitive for $r\neq 2$. If $r<2$, then the algorithms terminates but it does not produce competitive outputs; and if $r>2$, then the algorithm does not terminate at all---within the considered time limit of 72 hours in the used computational cluster. Hence we limit ourselves to the case $r=2$. We refer to this as \emph{Two-Stage Local Search} (TS) since the search occurs in two stages: a first stage in which we loop over all arrays obtained by one change, and a second stage in which we loop over all arrays obtained by two changes.

\begin{remark}
Observe that, for $r\in\bbN$, the size of an $r$-ball in $(S^\ell,\Hamming)$ is
\begin{equation}
\sum_{k=0}^r\binom{\ell}{k}(s^\ell-1)^k.
\end{equation}
Hence we should expect the size of the $r$-neighborhood to be exponential in $r$. In the particular case $r=2$, we have
\begin{equation}
1+\ell (s^{\ell}-1)+\binom{\ell}{2}(s^{\ell}-1)^2\sim \ell^2 s^{2\ell}.
\end{equation}
\end{remark}

\subsection{Exploiting symmetry}

In order to reduce the search space, we can restrict ourselves to the symmetric arrays introduced in Subsection~\ref{subsec:symmetries} with which we reduce the search space significantly. As said in this subsection, for the heuristic method, we focus on two kinds of symmetric arrays: the bi-cyclic (TSBC or, simply, BC) and the quasi-cyclic (TSQC, or, simply, QC). 

From Subsection~\ref{subsec:symmetries}, we have that TSBC and TSQC will run faster as they significantly reduce the size of the considered arrays. In this way, the size of the neighborhoods that appears during the execution of the search is significantly smaller---by the above remark, the dependence of the size of Hamming 2-balls on the size of the array is roughly exponential.

\section{Algebraic Construction of Almost-Orthogonal Arrays} \label{sec:algebraic}

To construct optimal orthogonal arrays with an odd primer power number of level sets, constructions rely on finite fields. To do this, there are three main constructions: Bush construction~\cite[\S3.2]{heslst1999}, the Addelman-Kempthorne construction~\cite[\S3.3]{heslst1999} and the Rao-Hamming construction~\cite[\S3.4]{heslst1999}. 

In this section, we will use the variations of the Addelman-Kempthorne construction to show the following two theorems. The first one covers the case of index one and the second one the case of index two. The constructions are given after the theorems and the proofs can be found in Appendix~\ref{secproofB}. 

\begin{thm}\label{thm:AKconstruction}
Let $s$ be a prime power and $\tilde{\kappa}\leq s\,\frac{s^{\ell-1}-1}{s-1}$ and $\ell\geq 2$ positive integers. There is an array $A\in M_{s^{\ell}\times \frac{s^{\ell}-1}{s-1}+\tilde{\kappa}}(\{1,\ldots,s\})$ with $s^\ell$ runs, $\frac{s^{\ell}-1}{s-1}+\tilde{\kappa}$ factors and $s$ levels such that:
\begin{enumerate}
    \item[(0)] $A$ is an OA of strength one.
    \item[(1a)] The first $\frac{s^{\ell}-1}{s-1}$ columns of $A$ form an OA of strength two.
    \item[(1b)] The last $\tilde{\kappa}$ columns of $A$ form an OA of strength two.
    \item[(2)] The tolerance of strength two of $A$ satisfies:
    \begin{equation}
    \Tol_2(A)=s^{\ell-2}.
    \end{equation}
    \item[(3)] For every $p\in[1,\infty)$,
    \begin{equation}
    \Unb_{p,2}(A)=\tilde{\kappa}s^2(s-1)s^{(\ell-2)p}.
    \end{equation}
\end{enumerate}
\end{thm}

\begin{thm}\label{thm:AKconstructionB}
Let $s>2$ be a prime power and $\tilde{\kappa}\leq s-1$ and $\ell\geq 2$ positive integers. There is an array $A\in M_{2s^{\ell}\times 2\frac{s^{\ell}-1}{s-1}-1+\tilde{\kappa}}(\{1,\ldots,s\})$ with $2s^\ell$ runs, $2\frac{s^{\ell}-1}{s-1}-1+\tilde{\kappa}$ factors and $s$ levels such that:
\begin{enumerate}
    \item[(0)] $A$ is an OA of strength one.
    \item[(1a)] The first $\frac{2s^{\ell}-1}{s-1}-1$ columns of $A$ form an OA of strength two.
    \item[(1b)] If we remove the first column of $A$ and any of the $\tilde{\kappa}-1$ last columns of $A$, we have an OA of strength two.
    \item[(2)] The tolerance of strength two of $A$ satisfies:
    \begin{equation}
    \Tol_2(A)=\max\{2,(s-2)\}s^{\ell-2}.
    \end{equation}
    \item[(3)] For every $p\in[1,\infty)$,
    \begin{equation}
    \Unb_{p,2}(A)=\binom{\tilde{\kappa}+1}{2}2s(s-2)\left((s-2)^{p-1}+2^{p-1}\right)s^{(\ell-2)p}.
    \end{equation}
\end{enumerate}
\end{thm}
\begin{remark}
We exclude the case $s=2$ from Theorem~\ref{thm:AKconstructionB}, since we get an OA in that case.
\end{remark}
\begin{remark}
For $p=2$ and $\ell=2$, Theorem~\ref{thm:AKconstruction} attains the optimal lower bound of Corollary~\ref{cor:HamminSimilarityvsUnbalance2} for any choice of $\tilde{\kappa}$ and Theorem~\ref{thm:AKconstructionB} does so too for $\tilde{\kappa}=1$. Hence, for these parameters, the AOAs of Theorems~\ref{thm:AKconstruction} and~\ref{thm:AKconstructionB} (Constructions~\ref{constructionAK},~\ref{constructionAKB} and \ref{constructionAKBeven}) are optimal.
\end{remark}

In what follows, $s$ will be a prime power, $\bbF_s$ will denote the finite field with $s$ elements, and $\ell$ and $\tilde{\kappa}$ fixed as in the statement of the theorem. Also, we will denote vectors of $\bbF_s^\ell$ as $(x,y)$ with $x\in \bbF_s$ and $y\in\bbF_s^{\ell-1}$, and we will use the notation
\begin{equation}
    \gamma^Ty:=\sum_{i=1}^{\ell-1}\gamma_iy_i
\end{equation}
for $\gamma,y\in\bbF_s^{\ell-1}$. Now, we can provide the constructions for Theorems~\ref{thm:AKconstruction} and~\ref{thm:AKconstructionB}.

\begin{construct}[Addelman-Kempthorne half-construction]\label{constructionAK}
Take a subset
\[
\Gamma\subset\bbF_s^{\ell-1}
\]
maximal with respect the property that every two elements in $\Gamma$ are $\bbF_s$-linearly independent, i.e., every $\bbF_s$-linear line in $\bbF_s^{\ell-1}$ is spanned by one and only one element in $\Gamma$.

On the one hand, consider the array $L\in M_{s^\ell\times 1+s\frac{s^{\ell-1}-1}{s-1}}(\bbF_s)$, with rows indexed by $\bbF_s^\ell$ and columns indexed by the disjoint union of $\{\ast\}$ and $\bbF_s\times \Gamma$, given by
\begin{equation}
L_{(x,y),z}=\begin{cases}
x ,&\text{if }z=\ast\\
ax+\gamma^Ty ,&\text{if }z=(a,\gamma)\in\bbF_s\times \Gamma 
\end{cases}
\end{equation}
where $(x,y)\in\bbF_s\times \bbF_s^{\ell-1}\cong \bbF_s^{\ell}$ and $z\in \{\ast\}\sqcup \bbF_s\times \Gamma$.

On the other hand, consider the array $Q\in M_{s^\ell\times \tilde{\kappa}}(\bbF_s)$, with rows indexed by $\bbF_s^\ell$ and columns by a subset $\Xi\subseteq \bbF_s\times \Gamma$ of size $\tilde{\kappa}$, given by
\begin{equation}
Q_{(x,y),(\beta,\gamma)}=x^2+\beta x+\gamma^Ty
\end{equation}
where $(x,y)\in\bbF_s\times \bbF_s^{\ell-1}\cong \bbF_s^{\ell}$ and $(\beta,\gamma)\in\Xi\subset \bbF_s\times \Gamma$.

The array
\begin{equation}
A=\begin{pmatrix}
    L&Q
\end{pmatrix}\in M_{s^{\ell}\times \frac{s^{\ell}-1}{s-1}+\tilde{\kappa}}(\bbF_s)
\end{equation}
satisfies the claims of Theorem~\ref{thm:AKconstruction}.
\end{construct}

\begin{construct}[Extension of Addelman-Kempthorne for odd $s$]\label{constructionAKB}
Take a non-square $\omega\in\bbF_s$ and a subset
\begin{equation}
\Gamma\subset\bbF_s^{\ell-1}
\end{equation}
maximal with respect the property that every two elements in $\Gamma$ are $\bbF_s$-linearly independent, as we did in the construction above.

First, consider the arrays $L,\tilde{L}\in M_{s^\ell\times 1+s\frac{s^{\ell-1}-1}{s-1}}(\bbF_s)$, with rows indexed by $\bbF_s^\ell$ and columns indexed by the disjoint union of $\{\ast\}$ and $\bbF_s\times \Gamma$, given by
\begin{equation}
L_{(x,y),z}=\begin{cases}
x ,&\text{if }z=\ast\\
ax+\gamma^Ty ,&\text{if }z=(a,\gamma)\in\bbF_s\times \Gamma 
\end{cases}
\end{equation}
and
\begin{equation}
\tilde{L}_{(x,y),z}=\begin{cases}
x ,&\text{if }z=\ast\\
ax+\gamma^Ty+\left(1-\frac{1}{\omega}\right)\frac{a^2}{4} ,&\text{if }z=(a,\gamma)\in\bbF_s\times \Gamma 
\end{cases}
\end{equation}
where $(x,y)\in\bbF_s\times \bbF_s^{\ell-1}\cong \bbF_s^{\ell}$ and $z\in \{\ast\}\sqcup \bbF_s\times \Gamma$.

Second, consider the arrays $Q,\tilde{Q}\in M_{s^\ell\times \tilde{\kappa}}(\bbF_s)$, with rows indexed by $\bbF_s^\ell$ and columns by $\bbF_s\times \Gamma$, given by
\begin{equation}
Q_{(x,y),(\beta,\gamma)}=(x-\beta)^2+\gamma^Ty
\end{equation}
and
\begin{equation}
\tilde{Q}_{(x,y),(\beta,\gamma)}=\omega(x-\beta)^2+\gamma^Ty
\end{equation}
where $(x,y)\in\bbF_s\times \bbF_s^{\ell-1}\cong \bbF_s^{\ell}$ and $(\beta,\gamma)\in \bbF_s\times \Gamma$.

Third, consider the arrays $E,\tilde{E}\in M_{s^\ell\times \tilde{\kappa}}(\bbF_s)$, with rows indexed by $\bbF_s^\ell$ and columns by a subset $\Xi\subset (\bbF_s\setminus 0)\times \{\beta_0\}\subset (\bbF_s\setminus 0)\times \bbF_s$, for some $\beta_0\in\bbF_s$, given  by
\begin{equation}
E_{(x,y),(a,\beta)}=(x-\beta)^2-ax
\end{equation}
and
\begin{equation}
\tilde{E}_{(x,y),(a,\beta)}=\omega(x-\beta)^2-ax-\left(1-\frac{1}{\omega}\right)\frac{a^2}{4}
\end{equation}
where $(x,y)\in\bbF_s\times \bbF_s^{\ell-1}\cong \bbF_s^{\ell}$ and $(a,\beta)\in\Xi\subset \{(a,\beta_0)\in \bbF_s^2\mid a\neq 0\}$.

The array
\begin{equation}
A=\begin{pmatrix}
    L&Q&E\\
    \tilde{L}&\tilde{Q}&\tilde{E}
\end{pmatrix}\in M_{2s^{\ell}\times 2\frac{s^{\ell}-1}{s-1}-1+\tilde{\kappa}}(\bbF_s)
\end{equation}
satisfies the claims of Theorem~\ref{thm:AKconstructionB} for odd $s$.
\end{construct}

\begin{construct}[Extension of Addelman-Kempthorne for even $s$]\label{constructionAKBeven}
Take $\zeta\in\bbF_s$ such that $\zeta^{s/2}\neq \zeta$ (i.e., $\zeta\notin\bbF_{s/2}$ when $s>2$) and a subset of
\begin{equation}
\Gamma\subset\bbF_s^{\ell-1}
\end{equation}
maximal with respect the property that every two elements in $\Gamma$ are $\bbF_s$-linearly independent, as we did in the construction above.

First, consider the arrays $L,\tilde{L}\in M_{s^\ell\times 1+s\frac{s^{\ell-1}-1}{s-1}}(\bbF_s)$, with rows indexed by $\bbF_s^\ell$ and columns indexed by the disjoint union of $\{\ast\}$ and $\bbF_s\times \Gamma$, given by
\begin{equation}
L_{(x,y),z}=\begin{cases}
x ,&\text{if }z=\ast\\
ax+\gamma^Ty ,&\text{if }z=(a,\gamma)\in\bbF_s\times \Gamma 
\end{cases}
\end{equation}
and
\begin{equation}
\tilde{L}_{(x,y),z}=\begin{cases}
x ,&\text{if }z=\ast\\
ax+\gamma^Ty+a^2\zeta ,&\text{if }z=(a,\gamma)\in\bbF_s\times \Gamma 
\end{cases}
\end{equation}
where $(x,y)\in\bbF_s\times \bbF_s^{\ell-1}\cong \bbF_s^{\ell}$ and $z\in \{\ast\}\sqcup \bbF_s\times \Gamma$.

Second, consider the arrays $Q,\tilde{Q}\in M_{s^\ell\times \tilde{\kappa}}(\bbF_s)$, with rows indexed by $\bbF_s^\ell$ and columns by $\bbF_s\times \Gamma$, given by
\begin{equation}
Q_{(x,y),(\beta,\gamma)}=x^2+\beta x+\gamma^Ty
\end{equation}
and
\begin{equation}
\tilde{Q}_{(x,y),(\beta,\gamma)}=x^2+\beta x+\gamma^Ty+\beta^2\zeta
\end{equation}
where $(x,y)\in\bbF_s\times \bbF_s^{\ell-1}\cong \bbF_s^{\ell}$ and $(\beta,\gamma)\in \bbF_s\times \Gamma$.

Third, consider the arrays $E,\tilde{E}\in M_{s^\ell\times \tilde{\kappa}}(\bbF_s)$, with rows indexed by $\bbF_s^\ell$ and columns by a subset $\Xi\subset \bbF_s\setminus 0$ of size $\tilde{\kappa}$ given by
\begin{equation}
E_{(x,y),\delta}=x^2+\delta x
\end{equation}
and
\begin{equation}
\tilde{E}_{(x,y),\delta}=x^2+\delta x+\delta^2\zeta
\end{equation}
where $(x,y)\in\bbF_s\times \bbF_s^{\ell-1}\cong \bbF_s^{\ell}$ and $\delta\in\Xi\subset \bbF_s\setminus 0$.

The array
\begin{equation}
A=\begin{pmatrix}
    L&Q&E\\
    \tilde{L}&\tilde{Q}&\tilde{E}
\end{pmatrix}\in M_{2s^{\ell}\times 2\frac{s^{\ell}-1}{s-1}-1+\tilde{\kappa}}(\bbF_s)
\end{equation}
satisfies the claims of Theorem~\ref{thm:AKconstructionB} for even $s$.
\end{construct}

Recall that the proofs of Theorems~\ref{thm:AKconstruction} and~\ref{thm:AKconstructionB} are in Appendix~\ref{secproofB}.

\section*{Conflict of Interest Statement}

The authors declare than none of the authors have a conflict of interest with respect to this paper.

{\small 
\printbibliography[heading=bibliography]
}

\appendix 

\section{Lower Bounds on the \texorpdfstring{$2$}{2}-Unbalance}\label{sec:unbalancebounds}

In the special case of the $2$-unbalance, an alternative formula can be given in terms of how similar the runs of the array are. This generalizes~\cite[Lemma 2.7]{heslst1999} to general arrays and was already proven in \cite[Lemma 3.1]{UDpaper2}. Moreover, it plays a fundamental role in obtaining lower bounds for the $2$-unbalance, as shown in \cite{UDpaper8liuhickernell,UDpaper10}. We move all proofs to Appendix~\ref{sec:unbalanceproofs} to ease the reading of this appendix.

\begin{thm}\label{thm:HamminSimilarityvsUnbalance}
Let $A\in S^{N\times k}$ be an $N\times k$ array with entries in a set $S$ of size $s$ and $t\in \bbN$. Denote by
\begin{equation}
    \HammingSimilarity(A;r,\tilde{r}):=\#\{j\in [k]\mid A_{r,j}=A_{\tilde{r},j}\}
\end{equation}
the \emph{Hamming similarity} between the $r$th and the $\tilde{r}$th runs of $A$. Then
\begin{equation}\label{eq:hammingformula}
    \Unb_{2,t}(A)=\sum_{1\leq r,\tilde{r}\leq N}\binom{\HammingSimilarity(A;r,\tilde{r})}{t}-\binom{k}{t}\frac{N^2}{s^t}.
\end{equation}
\end{thm}
\begin{remark}
If we compute $\Unb_{2,t}$ directly using the definition, we need to add a total of $\binom{k}{t}s^t$ numbers, but, if we use~\eqref{eq:hammingformula}, then we have to add $N^2$ numbers. In general, the latter will give an smaller sum than the former---specially for large values of $t$. 
\end{remark}
\begin{cor}\label{cor:HamminSimilarityvsUnbalance1}
Let $A\in S^{N\times k}$ be an $N\times k$ array with entries in a set $S$ of size $s$. Then
\begin{equation}\label{eq:U2lowerbound}
\Unb_{2,2}(A)\geq \lambda s^2(\lambda s^2-1)\left(\binom{\left\lfloor\gamma\right\rfloor}{2}+\left\lfloor\gamma\right\rfloor\left(\gamma-\left\lfloor\gamma\right\rfloor\right)\right)-\lambda(\lambda-1)s^2\binom{k}{2}
\end{equation}
where $\lambda=N/s^2$ is the index (of strength two) of $A$ and $\gamma=(\lambda s-1)k/(\lambda s^2-1)$. Moreover, the lower bound is attained if and only if $A$ is an $OA(N,k,s,1)$ and if there is a subset $R\subset T_{N,2}$ of size $\binom{N}{2}(1+\left\lfloor\gamma\right\rfloor-\gamma)$ such that for $(r,\tilde{r})\in R$, $\HammingSimilarity(r,\tilde{r})=\left\lfloor\gamma\right\rfloor$ and for $(r,\tilde{r})\in T_{N,2}\setminus R$, $\HammingSimilarity(A;r,\tilde{r})=\left\lfloor\gamma\right\rfloor+1$.
\end{cor}
\begin{cor}\label{cor:HamminSimilarityvsUnbalance2}
Let $A\in S^{N\times k}$ be an $N\times k$ array with entries in a set $S$ of size $s$. If $\lambda=1$ and $k=\alpha(s+1)+\tilde{\kappa}$ with $\alpha\in\bbN$ and $0\leq\tilde{\kappa}\leq s$, then \eqref{eq:U2lowerbound} takes the form
\begin{equation}\label{eq:U2lowerboundlambdaoneA}
    \Unb_{2,2}(A)\geq \alpha\tilde{\kappa} s^2(s-1)+\binom{\alpha}{2}s^2(s^2-1)=\begin{cases}0&\text{if }\alpha=0\\\tilde{\kappa} s^2(s-1)&\text{if }\alpha=1\end{cases}
\end{equation}
If $\lambda=2$ and $k=2s+1+\tilde{\kappa}$ with $1\leq \tilde{\kappa}\leq s$, then \eqref{eq:U2lowerbound} takes the form
\begin{equation}\label{eq:U2lowerboundlambdatwoA}
    \Unb_{2,2}(A)\geq 2s^2\left((2\tilde{\kappa}-1)(s-1)-\binom{\tilde{\kappa}+1}{2}\right);
\end{equation}
and, if $k\leq 2s+1$, then the lower bound of \eqref{eq:U2lowerbound} is negative and so unattainable.
\end{cor}
\begin{remark}
The lower bound \eqref{eq:U2lowerbound} (and its consequences) is not necessarily attainable. For example, if $A$ is an $OA(2s^2,2s+1,s,2)$, which exists for $s>2$ prime power by~\cite[Theorem~3.16]{heslst1999}, then the lower bound is $-2s^2(s-1)$ and so negative. Hence, the bound does not need to be tight, even in the case of OAs. Interestingly, this shows that for an $OA(2s^2,2s+1,s,2)$ the distribution of the $\HammingSimilarity(A;r,\tilde{r})$ is not as simple as the one given by Corollary~\ref{cor:HamminSimilarityvsUnbalance2}.
\end{remark}
\begin{remark}
Corollary~\ref{cor:HamminSimilarityvsUnbalance1} is nothing more than \cite[Theorem 2]{UDpaper10}. However, unlike \cite{UDpaper10}, we don't focus on the case in which $\gamma$ is an integer.
\end{remark}

\section{Comparison with other relaxations of OAs}\label{sec:comparative}

Ours is not the first relaxation of OA. Our notion of AOAs, unbalance and tolerance is similar and related to other existing relaxations of OAs. We discuss briefly the relationship between our notion and two families in the literature: the relaxations motivated by experimental designs, which do not lead to a characterization of OAs; and the more combinatorial relaxations, which do lead to characterizing measures of being an OA. As far as we know, the inequalities presented here are novel. We move all proofs to Appendix~\ref{sec:unbalanceproofs} to make the flow of text better.

\subsection{Relaxations in the context of experimental designs}

In the context of experimental designs, relaxations of OAs have been around for a long time. The main observation is that after assigning coefficients for the estimation of main effects to the levels of the array, one obtains an orthogonal matrix. This result can be summarized in the following way. Below, $\bbI$ is the identity matrix of the appropriate order.

\begin{prop}\label{prop:columnorthogonality}
Let $S$ be a set of size $s$ and non-constant $f:S\rightarrow \bbR$ such that
\begin{equation}\label{eq:fnullaverage}
    \sum_{x\in S}f(x)=0,
\end{equation}
and let $A\in S^{N\times k}$ be an $N\times k$ array with entries in $S$. Consider the real matrix $X(A,f)\in \bbR^{N\times k}$ is the matrix whose entries are given by
\begin{equation}\label{eq:Xmatrix}
X_{i,j}(A,f):=\frac{f(A_{i,j})}{\sqrt{\sum_{i=1}^Nf(A_{i,j})^2}}.
\end{equation}
If $A$ is an $OA(N,k,s,2)$, then $X(A,f)^TX(A,f)=\bbI$, i.e., the columns of $X(A,f)$ are orthogonal.  
\end{prop}
\begin{remark}
The choice of $f$ depends generally on the experimental design that one is considering. In~\cite[Ch.~11]{hinkelmankempthorne2008}, one can see the choice $f(x)=2x-3$ when $S=\{1,2\}$ and the choices $f(x)=x-2$ and $f(x)=3(x-2)^2-2$ when $S=\{1,2,3\}$. In all cases, \eqref{eq:fnullaverage} holds.
\end{remark}

Based on this result, in the case of $s=2$, Booth and Cox~\cite{booth1962} proposed to look at how near of being zero
$X(A,f)^TX(A,f)-\bbI$
is. Depending on the norm used, the latter leads to $\mathrm{Ave}(s^2)$ (also known as $E(s^2)$), if we use the Frobenius norm; $\mathrm{Ave}(|s|)$, using the $L_{1,1}$-norm; and $s_{\max}$, using the max-norm. All these measures led to a fruitful search and discussion regarding AOAs of strength two and two levels~\cite{lin1993,lin1995,wu1993,denglinwang1999}.

Based on the above, generalizing the above beyond the two level case, Wang and Wu~\cite{wang1992nearly} introduced for the first time the term of ``nearly-orthogonal array''. To measure non-orthogonality, however, Wang and Wu used the so called $\mcD$-criterion, based on the $\mcD$-value defined below.

\begin{dfn}\label{defi:Dvalue}
Let $S$ be a set of size $s$ and non-constant $f:S\rightarrow \bbR$ such that~\eqref{eq:fnullaverage} holds and $A\in S^{N\times k}$ be an $N\times k$ array with entries in $S$. The \emph{$\mathcal{D}$-value of $A$ with respect $f$} is the number
\begin{equation}
\mathcal{D}_f(A):=\det(X(A,f)^TX(A,f))^{\frac{1}{k}}\in[0,1]
\end{equation}
where $X(A,f)\in \bbR^{N\times k}$ is the matrix defined by \eqref{eq:Xmatrix}.
\end{dfn}
\begin{remark}\label{rem:Dcritbound}
Note that $\mathcal{D}_f(A)$ is a real number in $[0,1]$, because $X(A,f)^TX(A,f)$ is a positive semi-definite matrix with ones in the diagonal. To see this, we only have to apply the arithmetic-geometric inequality to the eigenvalues. Moreover, by the equality case, $\mathcal{D}_f(A)= 1$ if and only if $X(A,f)$ has orthogonal columns.
\end{remark}
\begin{prop}[$\mathcal{D}$-criterion]\label{prop:Dcriterion}
Let $A\in S^{N\times k}$ be an $N\times k$ array with entries in a set $S$ of size $s$ and non-constant $f:S\rightarrow\mathbb{R}$ such that \eqref{eq:fnullaverage} holds. If $A$ is an $OA(N,k,s,2)$, then $\mathcal{D}_f(A)=1$.
\end{prop}

We can see that the unbalance and tolerance give us good measures of how far is $X(A,f)$ of being orthogonal. Moreover, we obtain an explicit bound on the $\mcD$-value---unfortunately, the existing bounds on the determinant of a perturbed identity matrix~\cite{brentosbornsmith2015det,rump2018det} don't allow us to obtain stronger bounds than the stated.

\begin{thm}\label{theo:DcriterionvsUnbalance}
Let $A\in S^{N\times k}$ be an $N\times k$ array with entries in a set $S$ of size $s$ and non-constant $f:S\rightarrow\mathbb{R}$ such that \eqref{eq:fnullaverage}. If $A$ is an $OA(N,k,s,1)$, then
\begin{equation}\label{eq:XorthvsUnb}
\left\|X(A,f)^TX(A,f)-\bbI\right\|_F\leq \frac{\sqrt{2\Unb_{2,2}(A)}}{\lambda s},
\end{equation}
where $X(A,f)\in \bbR^{N\times k}$ is given by \eqref{eq:Xmatrix}, $\|\,\cdot\,\|_F$ is the Frobenius norm and $\lambda=N/s^2$ the index (of strength two) of $A$. Now, let
\begin{equation}\label{eq:sigmas}
    \sigma_1(f):=\min_{\alpha}\frac{1}{s}\sum_{x\in S}|f(x)-\alpha|~\text{ and }~\sigma_2(f):=\min_{\alpha}\sqrt{\frac{1}{s}\sum_{x\in S}|f(x)-\alpha|^2}
\end{equation}
be the \emph{absolute deviation} and \emph{typical deviation} of $f$ (with respect the uniform probability measure on $S$), respectively. Then
\begin{equation}\label{eq:XorthvsTolgen}
    \left\|X(A,f)^TX(A,f)-\bbI\right\|_{\max}\leq \frac{1}{\lambda}\left(\frac{\sigma_1(f)}{\sigma_2(f)}\right)^2\Tol_2(A)\leq \frac{1}{\lambda}\Tol_2(A)
\end{equation}
where $\|\,~\,\|_{\max}$ is the max-norm.
\end{thm}
\begin{remark}
By the definitions of $\mathrm{Ave}(s^2)$ and $s_{\max}$, we have that inequalities~\eqref{eq:XorthvsUnb} and~\eqref{eq:XorthvsTolgen} give bounds of these invariants in terms of the unbalance and the tolerance.
\end{remark}
\begin{remark}
We observe that the inequalities are only in one direction, because the columns of $X(A,f)$ are orthogonal if $A$ is an $OA(N,k,s,2)$ but not the other way around.
\end{remark}
\begin{remark}
Proposition~\ref{prop:sigmaquotient} below is to give an idea on the size of $\left(\sigma_1(f)/\sigma_2(f)\right)^2$ in Theorem~\ref{theo:DcriterionvsUnbalance} and Corollary~\ref{cor:DcriterionvsUnbalance}.
\end{remark}
\begin{cor}\label{cor:DcriterionvsUnbalance}
Let $A\in S^{N\times k}$ be an $N\times k$ array with entries in a set $S$ of size $s$ and non-constant $f:S\rightarrow\mathbb{R}$ such that \eqref{eq:fnullaverage}. Assume that (a) $A$ is an $OA(N,k,s,1)$, and (b) if we remove the last factor of $A$ we get an $OA(N,k-1,s,2)$ and if we remove some other $r$ factors, different from the last one, we get an $OA(N,k-r,s,2)$. If
\begin{equation}\label{eq:DcritvsTolcond}
   \frac{\sqrt{r}}{\lambda} \left(\frac{\sigma_1(f)}{\sigma_2(f)}\right)^2\Tol_2(A)<1,
\end{equation}
where $\lambda=N/s^2$ is the index (of strength two) of $A$ and $\sigma_1$ and $\sigma_2$ where given in~\eqref{eq:sigmas}; then
\begin{equation}\label{eq:DcritvsTol}
  \mathcal{D}_f(A)\geq \left(1-\frac{r}{\lambda^2}\left(\frac{\sigma_1(f)}{\sigma_2(f)}\right)^4\Tol_2(A)^2\right)^{\frac{1}{k}}.
\end{equation}
\end{cor}
The following proposition gives a general idea regarding the range of values of $\sigma_1(f)/\sigma_2(f)$.
\begin{prop}\label{prop:sigmaquotient}
Let $S$ be a set of size $s$ and $f:S\rightarrow\mathbb{R}$ a non-constant function such that \eqref{eq:fnullaverage}. Then
\begin{equation}\label{eq:sigmaquotient}
    \frac{1}{s-1}\leq \left(\frac{\sigma_1(f)}{\sigma_2(f)}\right)^2\leq 1.
\end{equation}
Moreover, both inequalities are tight.
\end{prop}
\begin{remark}
In the case described by Corollary~\ref{cor:DcriterionvsUnbalance}, we have, by Rao's bound~\cite[Theorem~2.1]{heslst1999}, that $k\leq \lambda (s+3)$. Thus $\sqrt{r}\leq 2\sqrt{\lambda(s-1)}$ and so, to make Corollary~\ref{cor:DcriterionvsUnbalance} usable, we need
\begin{equation}\label{eq:userangeofsigmas}
    \left(\frac{\sigma_1(f)}{\sigma_2(f)}\right)^2\leq \frac{1}{(s-1)^{\frac12+\varepsilon}}
\end{equation}
for some $\varepsilon>0$. Otherwise, the condition would not be satisfied at all---note that this is within the range of values given by Proposition~\ref{prop:sigmaquotient}.

Unfortunately, in the common case in which the values of $f:S\rightarrow \bbR$ form an arithmetic sequence of length $s$, we have $2/3\leq \left(\sigma_1(f)/\sigma_2(f)\right)^2\leq 1$ with $\lim_{s\to\infty}\left(\sigma_1(f)/\sigma_2(f)\right)^2=3/4$. Hence Corollary~\ref{cor:DcriterionvsUnbalance} is unusable in this case.
\end{remark}
\begin{remark}
Note that the condition~\eqref{eq:DcritvsTolcond} does not hold for the trivial construction of Proposition~\ref{prop:trivialconstruction} below for any $f$. Moreover, for the array $A$ of Proposition~\ref{prop:trivialconstruction} below, we indeed have $\mcD_f(A)=0$, since $X(A,f)$ has repeated columns.
\end{remark}
Later on, Xu~\cite{xu2002algorithm} proposed another measure of non-orthogonality $J_2$. However, to achieve a purely combinatorial relaxation of OAs, we have to choose the weights in~\cite[\S2.1]{xu2002algorithm} uniformly equal to $s$. In this way, due to \cite[Lemma~2]{xu2002algorithm}, we would get the following definition for OAs of strength one---the ones for which Xu~\cite{xu2002algorithm} considers $J_2$.

\begin{dfn}\cite[Lemma~2]{xu2002algorithm}
Let $A\in S^{N\times k}$ be an $N\times k$ array with entries in a set $S$ of size $s$ such that $A$ is an $OA(N,k,s,1)$ and $f:S\rightarrow \bbR$ such that \eqref{eq:fnullaverage}. The \emph{$J_2$-value of $A$} is
\begin{equation}
    J_2(A):=\frac{N^2}{2}\left\|X(A,f)^TX(A,f)-\bbI\right\|^2_F+\frac{N}{2}\left(Nk(k-1)+Nks-k^2s^2\right)
\end{equation}
where $X(A,f)$ is given by \eqref{eq:Xmatrix}, $\bbI$ the identity matrix and $\|\,\cdot\,\|_F$ is the Frobenius norm.
\end{dfn}

In this way, due to Theorem~\ref{theo:DcriterionvsUnbalance}, $J_2$ is also related to the unbalance.

\subsection{Combinatorial relaxations}

In the more combinatorial setting, the relaxations are based on measures of non-or\-tho\-go\-na\-li\-ty that recover OAs in the limit case. As we will show, all these measures are equivalent to the unbalance and the tolerance.

In the more general setting of mixed arrays, Ma, Fang and Liski~\cite{ma2000new} introduced the $(\phi,\theta)$ criterion, which generalizes the $p$-unbalance of strength $2$. This notion was important, because it unified all previous relaxations in the literature: the $E(s^2)$ of Booth and Cox~\cite{booth1962} in the case $s=2$, the $\chi^2$ of Yamada and Lin~\cite{yamadalin2002} for $s=3$ and of Yamada and Matsui~\cite{yamadamatsui2002} for general $s$, the $\mathrm{Ave}(|f|)$ and $\mathrm{Ave}(f^2)$ of Fang, Lin and Ma~\cite{fang2000}, and the $B(2)$ of Lu, Li and Xie~\cite{lu2006class} (also in~\cite{UDpaper2}). In this way, the $(\phi,\theta)$ criterion provided the best way of defining nearly-orthogonal arrays

\begin{dfn}\cite{ma2000new}\label{defi:D1D2}
Let $A\in S^{N\times k}$ be an $N\times k$ array with entries in a set $S$ of size $s$, and $\phi:[0,\infty)\rightarrow [0,\infty)$ and $\theta:[0,\infty)\rightarrow [0,\infty)$ monotonic increasing functions such that $\phi(0)=\theta(0)=0$. The \emph{$(\phi,\theta)$ criterion for the non-orthogonality of $A$} is the quantity
\begin{equation}
    D_{\phi,\theta}(A):=\sum_{j\in T_{2,k}}\theta\left(\sum_{x\in S^2} \phi\left(\vert n(A,x,j)-N/s^2\vert\right)\right).
\end{equation}
In particular, the \emph{$D_1$ and $D_2$ criteria for the non-orthogonality of $A$} are
\begin{equation}
    D_1(A):=\binom{k}{2}^{-1}D_{p_1,\mathrm{id}}(A)~\text{ and }~D_2(A):=\binom{k}{2}^{-1}D_{p_2,\mathrm{id}}(A)
\end{equation}
where $\mathrm{id}$ denotes the identity and $p_{\alpha}$ the map $x\mapsto p_{\alpha}(x):=x^{\alpha}$.
\end{dfn}

Despite the generality of their $(\phi,\theta)$ criterion, Ma, Fang and Liski~\cite{ma2000new} focused on the $D_1$ and $D_2$ criteria. However, these equal to the unbalance up to a multiplicative factor. More generally, the studied cases of the $(\phi,\theta)$ criterion have been only those in which $\phi$ and $\theta$ are of the form $x\mapsto x^\alpha$, with $\alpha\geq 1$. All these cases relate nicely to the unbalance.

\begin{prop}\label{prop:phithetacriterionvsunbalance}
Let $A\in S^{N\times k}$ be an $N\times k$ array with entries in a set $S$ of size $s$, $\alpha,\beta\geq 1$ and let $D_{\alpha,\beta}$ denote $D_{p_{\alpha},p_{\beta}}$ where $p_\tau:x\mapsto x^\tau$. Then
\begin{equation}
    \Unb_{\alpha\beta,2}(A)\leq D_{\alpha,\beta}(A)\leq s^{2(\beta-1)}\Unb_{\alpha\beta,2}(A).
\end{equation}
In particular,
\begin{equation}
    D_1(A)=\binom{k}{2}^{-1}\Unb_{1,2}(A)~\text{ and }~D_2(A)=\binom{k}{2}^{-1}\Unb_{2,2}(A).
\end{equation}
\end{prop}

Because of this proposition, the unbalance captures all important cases of the $(\phi,\theta)$ criterion, including the ones it generalizes, while having a simpler form.

Being more explicit, the $E(s^2)$ of Booth and Cox~\cite{booth1962} is $4D_2$; the $\chi^2$ of Yamada and Lin~\cite{yamadalin2002} and of Yamada and Mautsi~\cite{yamadamatsui2002} is $\frac{s^2}{N}D_2$; the $\mathrm{Ave}(|f|)$ and $\mathrm{Ave}(f^2)$ of Fang, Lin and Ma~\cite{fang2000} are, respectively, $D_1$ and $\binom{k}{2}^{-1}D_{1,2}$; and the $E(f_{NOD})$ of \cite{UDpaper10,UDpaper13} is exactly $D_2$. All of these, by Proposition~\ref{prop:phithetacriterionvsunbalance}, can be reduced to the unbalance.

We should mention very specially the \emph{balance} $B(t)$ introduced by Lu, Li and Xie~\cite{lu2006class}. In general, we have that
\begin{equation}
    B(t)=\binom{k}{t}^{-1}\Unb_{2,t},
\end{equation}
so their definition is almost identical to ours. However, we believe that the term unbalance captures better this quantity. Moreover, the notion of balance seems to play a role in quadrature discrepancies, as we will show later in Subsection~\ref{subsec:quadrature}.

Finally, the recent work by Lin, Phoa and Kao~\cite{linphoakao2017} introduced the term of ``almost-orthogonal array''. Our definition of AOA (Definition~\ref{def:aoa}) is closely related to their notion although they use the notion of bandwidth instead of the tolerance.

\begin{dfn}
Let $A\in S^{N\times k}$ be an $N\times k$ array with entries in a set $S$ of size $s$ and $t\in\bbN$. The \emph{bandwidth of strength $t$ of $A$} is 
\begin{equation}
    \mathrm{BW}_t(A):=\max\left\{\left|n(A,x,i)-n(A,y,j)\right|\mid x,y\in S^t,\,i,j\in T_{t,k}\right\}.
\end{equation}
\end{dfn}

We can see easily that the bandwidth is nothing more than the tolerance up to a fixed constant factor.

\begin{prop}\label{prop:bwvstol}
Let $A\in S^{N\times k}$ be an $N\times k$ array with entries in a set $S$ of size $s$ and $t\in\bbN$. Then
\begin{equation}
    \Tol_t(A)\leq \mathrm{BW}_t(A)\leq 2\Tol_t(A).
\end{equation}
\end{prop}

However, the notion of AOA in Lin, Phoa and Kao~\cite{linphoakao2017} is not equivalent to ours, since they require stronger conditions: they require the array to be circulant. The latter makes the $n(A,x,j)$ not to depend on $j$, but can be very restrictive of a condition depending on the setting.

\subsection{Comparison with quadrature discrepancy measures}\label{subsec:quadrature}

In numerical integration, we want to find families of points $\mcX\subseteq \bbX$ of a size $N$ for which
\begin{equation}
    \mathrm{Q}_{\mcX}(\varphi):=\frac{1}{N}\sum_{x\in \mcX}\varphi(x).
\end{equation}
gives a good approximation of the integral $\int_{\mcX}\,\varphi(x)\,\mathrm{d}\bbP(x)$ with respect some probability measure $\bbP$ on $\mcX$. In this context, two issues arises: How do we obtain such families of points? How do we control the error of such a measure?

After some precedent work by Wang and Fang~\cite{wangfang1981}, Hickernell~\cite{hickernell1998,hickernell1998b} (cf.~\cite{hickernell2000}) attacked this problem obtaining several notions of discrepancy that allow one to control the quadrature error. Later on, based on Hickernell's method, new notions of discrepancy were introduced by Zhou, Fang and Ning~\cite{zhou2013}, Ke, Zhang and Ye~\cite{ke2015} and Hickernell and Liu~\cite{UDpaper8hickernellliu,UDpaper8liuhickernell}. 

Given a probability space $\bbX$ with probability measure $\bbP$ and a reproducing kernel $\mcK:\bbX\times \bbX\rightarrow \mcR$ with respect the inner product $\langle\,\cdot\,,\,\cdot\,\rangle_{\mcK}$ on $L^2(\bbX)$, we can consider the following discrepancy measure for a collection of (not necessarily distinct) points $\mcX\subset \bbX$:
\begin{equation}
D_2(\mcX;\bbX,\mcK):=\left(\int_{\bbX\times \bbX}\mcK(x,y)\,\mathrm{d}(\bbP-\bbP_{\mcX})(x)\mathrm{d}(\bbP-\bbP_{\mcX})(y)\right)^{\frac{1}{2}}
\end{equation}
where $\bbP_{\mcX}$ is the uniform probability measure on $\mcX$. For such a discrepancy measure, we have the following version of the Koksma-Hlawka inequality:
\begin{equation}
    \left|\int_{\mcX}\,\varphi(x)\,\mathrm{d}\bbP(x)-\mathrm{Q}_{\mcX}(\varphi)\right|\leq D_2(\mcX;\bbX,\mcK)\omega_2(\varphi;\bbX,\mcK)
\end{equation}
where $\varphi\in L^2(\mcX)$ and $\omega_2(\varphi;\mcX,\mcK):=\min_{c\in\bbR}\|\varphi-c\|_{\mcK}$. See \cite{hickernell2000} for all the details.

Depending on the kernel considered, we get three discrepancy measures for $[0,1]^k$ and one discrepancy measure for $[s]^k$.

\begin{dfn}\label{defi:CDWDMD}
Let $\mcX\subset [0,1]^k$ be a collection of (not necessarily distinct) points.
\begin{enumerate}
    \item[(CD)] The \emph{centralized $L_2$-discrepancy, CD, of $\mcX$} is $CD(\mcX):=D_2(\mcX;[0,1]^k,\mcK_C)$ where
    \begin{equation}
        \mcK_C(x,y):=\prod_{j=1}^k\left(1-\frac{1}{2}\left|x_j-\frac{1}{2}\right|-\frac{1}{2}\left|y_j-\frac{1}{2}\right|-\frac{1}{2}|x_j-y_j|\right).
    \end{equation}
    \item[(WD)] The \emph{wrap-around $L_2$-discrepancy, WD, of $\mcX$} is $WD(\mcX):=D_2(\mcX,\;[0,1]^k,\mcK_W)$ where
    \begin{equation}
        \mcK_W(x,y):=\prod_{j=1}^k\left(\frac{3}{4}-|x_j-y_j|+|x_j-y_j|^2\right).
    \end{equation}
    \item[(MD)] The \emph{mixture $L_2$-discrepancy, MD, of $\mcX$} is $MD(\mcX):=D_2(\mcX;[0,1]^k,\mcK_M)$ where
    \begin{equation}
        \mcK_M(x,y):=\prod_{j=1}^k\left(\frac{15}{8}-\frac{1}{4}\left|x_j-\frac{1}{2}\right|-\frac{1}{4}\left|y_j-\frac{1}{2}\right|-\frac{3}{4}|x_j-y_j|+\frac{1}{2}|x_j-y_j|^2\right).
    \end{equation}
\end{enumerate}
\end{dfn}
\begin{dfn}\label{dfn:DD}
Let $\mcX\subset [s]^N$ be a collection of (not necessarily distinct) points.
The \emph{discrete discrepancy, DD, of $\mcX$} is $DD(\mcX;a,b):=D_2(\mcX;[s]^k,\mcK_{D;a,b})$ where
\begin{equation}
    \mcK_{D;a,b}(x,y):=a^{\#\{j\in [k]\mid x_j=y_j\}}b^{k-\#\{j\in [k]\mid x_j=y_j\}}
\end{equation}
and $a>b>0$. 
\end{dfn}
\begin{remark}
Note that CD, WD and MD are discrepancy measures for quadratures over $[0,1]^k$, while DD is a discrepancy measure for quadratures over $[s]^k$. In a certain sense, we can see DD also as a discrepancy measure over $[0,1]^k$, if we see as telling us how much does a subcollection of points of the uniform grid differ from the whole uniform grid.
\end{remark}
\begin{remark}
Although we are giving CD, WD, MD and DD in terms of their kernels, it is possible to give explicit formulas for them in terms of the collection of points $\mcX$ directly. For CD, WD and MD, see \cite[(9), (10) \& (18)]{zhou2013} for the explicit formulas; for DD, see below.
\end{remark}

Now, a common way to find to find collections of points in $[0,1]^k$ is to consider arrays $A\in [s]^{N\times k}$ with $N$ runs, $k$ factors and $s$ levels $1,\ldots,s$, and associate to them the following indexed family of points
\begin{equation}\label{eq:integrationpoints}
    \mcP(A):=\left\{\begin{pmatrix}\frac{2A_{i,j}-1}{2s}\end{pmatrix}_{i=1}^N\in [0,1]^N\mid j\in [k]\right\}.
\end{equation}
Using this identification---which is sort of arbitrary---of arrays with collections of points, we have that we can reduce the problem of looking for points to average over to a combinatorial search of arrays for which the discrepancy measures are small. We will abuse notation, and write $CD(A)$, $MD(A)$ and $WD(A)$ instead of $CD(\mcP(A))$, $MD(\mcP(A))$ and $WD(\mcP(A))$. Observe that for $DD$, we can just write $DD(A)$ as we can see the runs of $A$ as a collection of $N$ points in $[s]^k$.

Because of the above, it has been common to focus on arrays in order to find collections of points minimizing CD, MD and WD. These techniques have been later extended to the latter DD. In general, we have two families of searches: those based on the threshold accepting heuristics~\cite{UDpaper1,UDpaper2,UDpaper3,fang2005,ke2015}, and those based on combinatorial constructions~\cite{UDpaper9,UDpaper10,UDpaper11,UDpaper14}. Only recently, there has been a motivation for non-array collections of points using gradient descent~\cite{UDpaper14}.

Intuitively, the reason why OA and, therefore also AOAs, enter into play for quadrature discrepancy estimate can be seen as follows. Let $A$ be an $OA(N,k,s,t)$, then we have that for any projection
$
    x\mapsto x_J
$
onto the coordinates indexed by $J\subseteq [k]$, the family
$
\mcP(A)_J:=\{p_J\mid p\in \mcP(A)\}\subset [0,1]^J
$
would just be uniform grid 
$
\left\{\frac{1}{2}+\frac{a}{s}\mid a\in [s]\right\}^J
$
with every point repeated the same number of times (precisely, $N/s^{|J|}$). In other words, if $\varphi$ depends only on $c\leq t$ coordinates, then averaging over $\mcP(A)$ would be the result of applying the quadrature rule resulting of the uniform grid.

The above can be made explicit as follows.

\begin{thm}\label{thm:discrepancyvsunbalance}
Let $A\in [s]^{N\times k}$ be an array with $N$ runs, $k$ factors and $s$ levels, and $a>b>0$. Then:
\begin{align}
    DD(A;a,b)^2&=\frac{1}{N^2}\sum_{t=1}^k\Unb_{2,t}(A)(a-b)^tb^{k-t}\label{eq:discrepancyvsunbalance1}\\
    &=-\left(\frac{a-b}{s}+b\right)^k+\frac{b^k}{N^2}\sum_{1\leq r,\tilde{r}\leq N}\left(\frac{a}{b}\right)^{\HammingSimilarity(A;r,\tilde{r})}.\label{eq:discrepancyvsunbalance2}
\end{align}
Moreover, we have the following:
\begin{enumerate}[wide, labelwidth=!, labelindent=0pt]
    \item[(CD)] For $s>2$,
    \begin{equation}
        CD(A)^2\leq \left(\frac{13}{12}\right)^k-2\left(\frac{9s^2-1}{8s^2}\right)^k+\left(\frac{3s^2-6s+2}{2s^2}\right)^k+DD\left(A;\frac{3s-1}{2s},\frac{3s-3}{2s}\right)^2;
    \end{equation}
    and, for $s=2$,
    \begin{equation}
        CD(A)^2=\left(\frac{13}{12}\right)^k-2\left(\frac{35}{32}\right)^k+\left(\frac{9}{8}\right)^k+DD\left(A;\frac{5}{4},1\right)^2.
    \end{equation}
    \item[(WD)] For $s\geq 2$,
    \begin{equation}
        WD(A)^2\leq -\left(\frac{4}{3}\right)^k+\left(\frac{3s^3-2s^2+4s-2}{2s^3}\right)^k+DD\left(A;\frac{3}{2},\frac{3s^2-2s+2}{2s^2}\right)^2,
    \end{equation}
    with equality if $s\leq 3$.
    \item[(MD)] For $s$ odd,
    \begin{multline}
        MD(A)^2\leq\left(\frac{19}{12}\right)^k-\frac{2}{N}\left(\frac{71s^2+12s-3}{48s^2}\right)^k\\+\left(\frac{15s^3-8s^2+12s-4}{8s^3}\right)^k+DD\left(A;\frac{15}{8},\frac{15s^2-8s+4}{8s^2}\right)^2;
    \end{multline}
    and, for $s$ even,
    \begin{multline}
        MD(A)^2\leq\left(\frac{19}{12}\right)^k-\frac{2}{N}\left(\frac{71s^2+12s-3}{48s^2}\right)^k\\+\left(\frac{15s^3-8s^2+10s-4}{8s^3}\right)^k+DD\left(A;\frac{15}{8}-\frac{1}{2s},\frac{15s^2-8s+4}{8s^2}\right)^2;
    \end{multline}
    with equality if $s=2$.
\end{enumerate}
\end{thm}
\begin{cor}\label{cor:discrepancyvsunbalance}
Let $A\in [s]^{N\times k}$ be an array with $N$ runs, $k$ factors and $s$ levels, and $a>b>0$. Then
\begin{equation}
    DD(A;a,b)\geq -\left(\frac{a-b}{s}+b\right)^k+\frac{a^k}{\lambda s^2}+b^k\left(1-\frac{1}{\lambda s^2}\right)\left(1+\left(\frac{a}{b}-1\right)(\gamma-\lfloor \gamma\rfloor)\right)\left(\frac{a}{b}\right)^{\lfloor\gamma\rfloor}
\end{equation}
where $\lambda=N/s^2$ is the index (of strength two) of $A$ and $\gamma=(\lambda s-1)k/(\lambda s^2-1)$. Moreover, the lower bound is attained if and only if $A$ is an $OA(N,k,s,1)$ and if there is a subset $R\subset T_{N,2}$ of size $\binom{N}{2}(1+\left\lfloor\gamma\right\rfloor-\gamma)$ such that for $(r,\tilde{r})\in R$, $\HammingSimilarity(r,\tilde{r})=\left\lfloor\gamma\right\rfloor$ and for $(r,\tilde{r})\in T_{N,2}\setminus R$, $\HammingSimilarity(A;r,\tilde{r})=\left\lfloor\gamma\right\rfloor+1$. In other words, the lower bound is attained if and only if $A$ attains the lower bound of Corollary~\ref{cor:HamminSimilarityvsUnbalance1}.
\end{cor}
\begin{remark}
The above theorem generalizes \cite[Theorem~3.1 \& 3.3]{UDpaper2} for CD and WD in the case $s=2$, \cite[Theorem~4.1]{UDpaper2} for WD in the case $s=3$,  and \cite[Theorem~2]{zhou2013} for MD in the case $s=2$ to arbitrary $s$. Unfortunately, for arbitrary $s$, we don't get an equality, but the inequality shows how OAs and AOAs fit in the picture of quadrature discrepancies.
\end{remark}
\begin{remark}
An alternative formulation of the formula for $DD$ of Theorem~\ref{thm:discrepancyvsunbalance} was given by Ma and Fang~\cite{mafang2001} and Xu and Wu~\cite{xuwu2001} in terms of the generalized word pattern, in the case $s=2$. However, both forms are equivalent.
\end{remark}

\section{Proofs for Section~\ref{sec:concepts} and Section~\ref{sec:concepts2}}\label{secproofA}

\begin{proof}[Proof of Proposition~\ref{prop:OAchar1}]
By its definition, $n(A,x,j)$ counts the number of times that $x\in S^t$ appears as a row in $A[j]:=A[j_1,\ldots,j_t]$. Hence
\begin{equation}\label{eq:countingsum}
\sum_{x\in S^t}n(A,x,j)=N,
\end{equation}
since $A[j]$ has $N$ rows. Now, $A$ is an OA if and only if for every $j$, all the $n(A,x,j)$ are equal. Since there are $s^t$ possible $x\in S^t$, the latter happens if and only if the value equals $N/s^t$.  
\end{proof}

\begin{proof}[Proof of Proposition~\ref{prop:unbtolrelations}]
In~\eqref{eq:normunbtol}, $\widehat\Unb_{p,t}(A)$is a $p$-mean of the $|n(A,x,j)-N/s^t|$. Hence the bounds follow from the usual inequalities between $p$-means.
\end{proof}

\begin{proof}[Proof of Proposition~\ref{prop:trivialconstruction}]
To compute the unbalance and tolerance of $A$, we only need to deal with the case in which we pick $j=(1,2)\in T_{2,k}$, as in all the other cases we are dealing with an orthogonal array by construction of $A$.

Now, take $j=(1,2)\in T_{2,k}$. Then 
\begin{equation}
    \sum_{x\in S^t} \vert n(A,x,j)-\lambda\vert^p =  \sum_{\substack{x\in S^t \\ x_1 = x_2}}\sum_{j\in T_{t,k}}  \vert n(A,x,j)-\lambda\vert^p + \sum_{\substack{x\in S^t \\ x_1 \neq x_2}}\sum_{j\in T_{t,k}} \vert n(A,x,j)-\lambda\vert^p.
\end{equation}
In the first summand, we have 
$
n(A,x,j)=n(A,(x_2),(2)),
$
since $x_1=x_2$ and the first two columns of $A$ are equal. Now, every $OA(N,k-1,s,2)$ is $OA(N,k-1,s,1)$, where the index is now $\lambda s$ instead of $\lambda$. Thus $n(A,x,j)=\lambda s$ in the first summand. In the second summand, we have that $n(A,x,j)=0$, because the first two columns of $A$ are equal, but $x_1\neq x_2$.

Hence, putting this together, we have
\begin{align*}
    \sum_{x\in S^t} \vert n(A,x,j)-\lambda\vert^p&=\sum_{\substack{x\in S^t \\ x_1 = x_2}}  \vert 
 \lambda(s-1)\vert^p + \sum_{\substack{x\in S^t \\ x_1 \neq x_2}} \vert 0-\lambda\vert^p\\
    &=s\lambda^p(s-1)^p+s(s-1)\lambda^p=\lambda^p s(s-1)\left[1+(s-1)^{p-1}\right].
\end{align*}
Moreover, we have shown that $\vert n(A,x,j)-\lambda\vert$ is at most $\lambda(s-1)$, so this is the value of the tolerance.
\end{proof}

For completeness, we add this proposition showing the effect of adding several factors at once.

\begin{prop}\label{prop:trivialconstructiongeneral}
Let $N,k,\ell,s\in\bbN$, $B\in [s]^{N\times (k-1)}$ be some array and $\tilde{\kappa}\leq k-1$. Fix $\lambda=N/s^2$. Then the array
\[
A=\left(\begin{array}{c|c}  
    B[[\tilde{\kappa}]] &B
\end{array}\right)\in [s]^{N\times k},
\]
where $B[[\tilde{\kappa}]]$ is the subarray of $B$ formed by its first $\tilde{\kappa}$ factors, satisfies
\[
\Tol_2(A)\leq \max\left\{\Tol_2(B),\lambda(s-1)\right\}
\]
and, for $p\in[1,\infty)$,
\[
\Unb_{p,2}(A)\leq 2\Unb_{p,2}(B)+2\Unb_{p,2}(B[[\tilde{\kappa}]])+\tilde{\kappa}\lambda^p s(s-1)\left(1+(s-1)^{p-1}\right).
\]
\end{prop}
\begin{remark}
If $\tilde{\kappa}=1$, note that $\Unb_{p,2}(B[[\tilde{\kappa}]])=0$, as the sum is empty.
\end{remark}
\begin{proof}[Proof of Proposition~\ref{prop:trivialconstructiongeneral}]
The proof is analogous to that of Proposition~\ref{prop:trivialconstruction}. However, there will be some extra cases, but the main idea is the same: the only pairs of factors of $A$ not covered by $B$ are those factors in which we are picking a factor of $B$ and its repetition in $B[[\tilde{\kappa}]]$.

For the tolerance, we consider a pair of factors. If the pair is analogous to pick a pair of factors of $B$, then we can bound the $|n(A,x,j)-\lambda|$ by $\Tol_2(B)$; if not, then we are picking the same factor twice, and so, using the proof of Proposition~\ref{prop:trivialconstruction}, we get the bound $\lambda(s-1)$. The maximum of these two bounds is the tolerance then.

For the unbalance, we will do the computation with more care. Given two factors $j_1$ and $j_2$ of $A$, we have four cases: a) both $j_1$ and $j_2$ are factors of $B[[\tilde{\kappa}]]$, b) both $j_1$ and $j_2$ are factors from $B$, c) $j_1$ is a factor from $B[\tilde{\kappa}]$ and $j_2$ one of the first $\tilde{\kappa}$ factors of $B$, and d) $j_1$ is a factor of $B[[\tilde{\kappa}]]$ and $j_2$ a factor from $B$ which is not one of the first $\tilde{\kappa}$ factors.

In the case (a), we get $\Unb_{p,2}(B[[\tilde{\kappa}]])$ after adding; in the case (b), we get $\Unb_{p,2}(B)$ after adding. The latter is so, because these cases reduce to computing the unbalance of $B[[\tilde{\kappa}]]$ and $B$ respectively. In the case (c), we have two kinds of summands: either $j_1$ and $j_2$ correspond to the same factor or they correspond to different ones. If they correspond to the same factor, then we have $\tilde{\kappa}$ summands, which are handled as in the proof of Proposition~\ref{prop:trivialconstruction}, getting $\tilde{\kappa} \lambda^ps(s-1)\left(1+(s-1)^{p-1}\right)$. If they correspond to different summand, then we are just computing the unbalance of $B[[\tilde{\kappa}]]$, but each pair of factors appears twice. Thus we get $2\Unb_{p,2}(B[[\tilde{\kappa}]])$. Summing up, in the case (c), we get
\[
2\Unb_{p,2}(B[[\tilde{\kappa}]])+\tilde{\kappa}\lambda^ps(s-1)\left(1+(s-1)^{p-1}\right).
\]
In the case (d), as we range over the pairs, we get all the pairs for computing the unbalance of $B$, except those use for computing the unbalance of $B[[\tilde{\kappa}]]$. Hence we get $\Unb_{p,2}(B)-\Unb_{p,2}(B[[\tilde{\kappa}]])$.

Putting all the cases above together, we get that $\Unb_{p,2}(A)$ equals
\begin{multline*}
  \Unb_{p,2}(B[[\tilde{\kappa}]])+\Unb_{p,2}(B)+\left(2\Unb_{p,2}(B[[\tilde{\kappa}]])+\tilde{\kappa}\lambda^ps(s-1)\left(1+(s-1)^{p-1}\right)\right)\\+\left(\Unb_{p,2}(B)-\Unb_{p,2}(B[[\tilde{\kappa}]])\right),  
\end{multline*}
which, after simplifying the sum, gives the desired result.
\end{proof}

\begin{proof}[Proof of Proposition~\ref{prop:unbsymmetries}]
We can see, after a simple computation, that for $A\in S^{N\times k}$, $x\in S^t$, $j\in T_{t,k}$ and $\auto{g}{\sigma}\in \Sym(S,k)$,
\begin{equation}
    n(\auto{g}{\sigma}A,x,j)=n(A,g^{-1}x,\sigma^{-1}j)
\end{equation}
where $g^{-1}x:=\left(g^{-1}x_1,\ldots,g^{-1}x_t\right)$ and $\sigma^{-1} j$ is equal to $\left(\sigma^{-1}j_1,\ldots,\sigma^{-1}j_t\right)$ up to reordering of the entries---the permutation might not give a tuple in $T_{t,k}$. Hence the action of $\Sigma(S,k)$ only permutes the summands of $\Unb_{p,t}$ or the values among which we take the maximum for computing $\Tol_t$. In either case, this does not change the resulting value.
\end{proof}

\section{Proofs for Section~\ref{sec:algebraic}}\label{secproofB}

\begin{proof}[Proof of Theorem~\ref{thm:AKconstruction} for $s$ odd]
Take $A$ as in Construction~\ref{constructionAK}. In order to prove the claims, we only need to choose all possible pair of columns and count how many times each $(u,v)\in\bbF_s^2$ appears. Now, since the rows are indexed by vectors $(x,y)\in\bbF_s^\ell$ being the entries of $A$ polynomials in those, this translates for each pair of columns to count the solutions of a system of the form
\[
\left\{
\begin{array}{rl}
     p_1(x,y)&=u\\
     p_2(x,y)&=v
\end{array}
\right.
\]
where $p_1$ and $p_2$ are polynomials of degree at most two. Note that the columns of $A$ are indexed by the disjoint union of $\{\ast\}$, $\bbF_s\times \Gamma$ and $\Xi$, so we have 5 possible combinations of column types.

\bigskip\noindent\emph{Case 1}. Column indexed by $\{\ast\}$ and column indexed by $\bbF_s\times \Gamma$ ($L$-column).

In this case, the system takes the form:
\begin{equation}\label{eq:sys1}
    \left\{
\begin{array}{rl}
     x&=u\\
     ax+\gamma^Ty&=v
\end{array}
\right.
\end{equation}
with $(a,\gamma)\in \bbF_s\times \Gamma$. Since $\gamma\neq 0$, this system has precisely $s^{\ell-2}$ solutions for each $(u,v)$, since it can be seen as an affine hyperplane in $\bbF_s^{\ell-1}$.

\bigskip\noindent\emph{Case 2}. Both columns indexed by $\bbF_s\times \Gamma$ ($L$-column).

In this case, the system takes the form:
\begin{equation}\label{eq:sys2}
    \left\{
\begin{array}{rl}
     ax+\gamma^Ty&=u\\
     \tilde{a}x+\tilde{\gamma}^Ty&=v
\end{array}
\right.
\end{equation}
with $(a,\gamma),(\tilde{a},\tilde{\gamma})\in \bbF_s\times \Gamma$ distinct. Note that the matrix
\[
\begin{pmatrix}
    a&\gamma^T\\
    \tilde{a}&\tilde{\gamma}^T
\end{pmatrix}
\]
has rank two. Otherwise, the two rows have to be proportional, but, by construction of $\Gamma$, this forces $\gamma=\tilde{\gamma}$ and so $a=\tilde{a}$, contrary to assumption. Hence this system has precisely $s^{\ell-2}$ solutions for every $(u,v)$, as it is the intersection of two non-parallel affine hyperplanes in $\bbF_s^\ell$.

\bigskip\noindent\emph{Case 3}. Column indexed by $\{\ast\}$ and column indexed by $\Xi$ ($Q$-column).

In this case, the system takes the form:
\begin{equation}\label{eq:sys3}
    \left\{
\begin{array}{rl}
     x&=u\\
     x^2+\beta x+\gamma^Ty&=v
\end{array}
\right.
\end{equation}
with $(\beta,\gamma)\in \Xi$. Arguing as in case 1, this system has precisely $s^{\ell-2}$ solutions for each $(u,v)$.

\bigskip\noindent\emph{Case 4}. Column indexed by $\bbF_s\times \Gamma$ ($L$-column) and column indexed by $\Xi$ ($Q$-column).

In this case, the system takes the form:
\begin{equation}\label{eq:sys4}
    \left\{
\begin{array}{rl}
     ax+\gamma^Ty&=u\\
     x^2+\beta x+\tilde{\gamma}^Ty&=v
\end{array}
\right.
\end{equation}
with $(a,\gamma)\in \bbF_s\times \Gamma$ and $(\beta,\tilde{\gamma})\in \Xi$. By construction of $\Gamma$, there are two cases: either $\gamma$ and $\tilde{\gamma}$ are distinct, and so linearly independent, or they are equal.

If they are distinct, then for every possible value of $x\in\bbF_s$, we get an affine system of codimension $2$ with $s^{\ell-3}$ solutions in $y\in\bbF_s^{\ell-1}$ for every $(u,v)$. Hence, in total, we have that there are $s^{\ell-2}$ solutions for each $(u,v)$.

If they are equal, then we can substitute $\gamma^Ty$ in the second equation, to get the following quadratic polynomial:
\[
x^2+(\beta-a)x+u-v=0.
\]
This polynomial has two, one or zero roots depending on whether its discriminant, given by $\Delta(u,v):=(\beta-a)^2-4(u-v)$, is a non-zero square, zero or a non-square in $\bbF_s$. Now, for each of its roots, there are $s^{\ell-2}$ possible solutions in $y$. Hence the whole system has $2s^{\ell-2}$, $s^{\ell-2}$ or $0$ solutions depending on whether $\Delta(u,v)$ is a non-zero square, zero or a non-square in $\bbF_s$.

Finally, since $\Delta(u,v)$ is an affine univariate map of $u-v$, we have that is zero for $s$ values of $(u,v)$, a non-zero square for $\frac{s(s-1)}{2}$ values of $(u,v)$ and a non-square for the remaining $\frac{s(s-1)}{2}$ values of $(u,v)$. Hence, for $s$ values of $(u,v)$, the system has $s^{\ell-2}$, and for the remaining $s(s-1)$ values of $(u,v)$, the system has either $2s^{\ell-2}$ or $0$ solutions.

\bigskip\noindent\emph{Case 5}. Both columns indexed by $\Xi$.

In this case, the system takes the form:
\begin{equation}\label{eq:sys5}
    \left\{
\begin{array}{rl}
     x^2+\beta x+\gamma^Ty&=u\\
     x^2+\tilde{\beta}x+\tilde{\gamma}^Ty&=v
\end{array}
\right.
\end{equation}
with $(\beta,\gamma),(\tilde{\beta},\tilde{\gamma})\in \Xi$. As in case 4, we distinguish two cases depending on whether $\gamma$ and $\tilde{\gamma}$ are distinct or equal.

If they are distinct, then we argue, as in case 4, making $x$ go trough all possible values. Thus we get $s^{\ell-2}$ solutions for each $(u,v)$.

If they are equal, then $\beta\neq \tilde{\beta}$. By substracting the first equation from the second, we arrive at a system of the form:
\begin{equation}\label{eq:sys5prime}
    \left\{
\begin{array}{rl}
     x^2+\beta x+\gamma^Ty&=u\\
     (\tilde{\beta}-\beta)x&=v-u
\end{array}
\right.
\end{equation}
which is analogous to the system considered in case 3. Hence the system has still $s^{\ell-2}$ solutions for each $(u,v)$.

Putting all these cases together, we have the only choice of columns of $A$ in which all possible $(u,v)$ appear the same ammount of times as a row is the case 4, which is the case in which we mix the columns of $L$ and $Q$. Hence $L$ and $Q$ are OA of strength two, and so $A$ an OA of strength one.

Now, from case $4$, we have that
\[
|n(A,((a,\gamma),(\beta,\tilde{\gamma})),(u,v))-s^{\ell-2}|
\]
differs from zero if and only if $\gamma=\tilde{\gamma}$. Note that given a column of $Q$ indexed by $(\beta,\gamma)$, this only happen for the $s$ columns of $L$ indexed by $(a,\gamma)$. And, in that case, it takes the value $0$ for $s$ possible values of $(u,v)$ and the value $s^{\ell-2}$ for the remaining $s(s-1)$ values of $(u,v)$.

Hence the claim about the tolerance of $A$ follows, and also
\[
\sum_{(a,\gamma)}\sum_{(u,v)}|n(A,((a,\gamma),(\beta,\gamma)),(u,v))-s^{\ell-2}|=s^2(s-1)s^{(\ell-2)p},
\]
from where the claim about the unbalance follows.
\end{proof}

\begin{proof}[Proof of Theorem~\ref{thm:AKconstruction} for $s$ even]
The proof is analogous to that of the odd case. We divide in the five possible cases, being the only different case the case 4, when $\gamma=\tilde{\gamma}$. Observe that this is the only case where the number of solutions of a quadratic equation mattered. 

Recall that the polynomial (rewritten in characteristic $2$) is of the form:
\[
x^2+(a+\beta)x+u+v=0.
\]
If $a=\beta$, then this polynomial has a unique zero for all $(u,v)$. If $a\neq \beta$, after a change of variables, we are reduced to the quadratic polynomial
\begin{equation}\label{eq:quadratic}
    Z^2+Z+\frac{u+v}{a+\beta}=0.
\end{equation}
Observe that this polynomial has either two roots in $\bbF_s$ or none at all---if $\zeta$ is a root, then $\zeta+1$ is also a root.

Now, consider the trace map,
\begin{align}\begin{split}\label{eq:trace}
    T:\bbF_s&\rightarrow \bbF_2\\
    x&\mapsto \sum_{i=0}^{\log_2 s-1}x^{2^i},
\end{split}\end{align}
and the map
\begin{align}\begin{split}\label{eq:cotrace}
    L:\bbF_s&\rightarrow \bbF_{s}\\
    x&\mapsto x+x^2.
\end{split}\end{align}
Both of this maps are $\bbF_2$-linear maps and we have that\footnote{To see this, note that $T\circ L=0$ (by direct computation, since $x^s=x$ for $x\in\bbF_s$) and that $\dim_{\bbF_2} \mathrm{im}\,L=s/2-1=\dim_{\bbF_2}\ker T$ (this follows from $\ker L=\bbF_2$, $\im T=\bbF_2$ and that the dimension of the domain of a linear map equals the sum of the dimension of the image and the kernel).}
\[
\im L=\ker T.
\]
Hence \eqref{eq:quadratic} has two solutions if and only if $T\left(\frac{u+v}{a+\beta}\right)=0$. But, using linearity with respect $u+v$, we have that this happens for half of the values of $(u,v)$.

Summing up, when $s$ is even, the system~\eqref{eq:sys4} has $s^{\ell-2}$ solution for all $(u,v)$ if $\gamma\neq \tilde{\gamma}$ or $a=\beta$. If $\gamma= \tilde{\gamma}$ and $a\neq \beta$, then it has $2s^{\ell-2}$ solutions for half of the values of $(u,v)$ and no solutions for the other half of the values of $(u,v)$.

Finally, counting gives the desired result.
\end{proof}

\begin{proof}[Proof of Theorem~\ref{thm:AKconstructionB} for odd $s$]
The proof strategy is similar to that of Theorem~\ref{thm:AKconstruction}. We will take two columns of $A$ and reduce the problem of how many times $(u,v)\in\bbF^2_s$ appears to that of counting solutions to a polynomial system over $\bbF_s$. However, we will get now two systems, one for the upper-half of $A$ and one for the lower-half. Hence we will count how many $(x,y)\in\bbF^{\ell}$ are solutions of
\[
\left\{
\begin{array}{rl}
     p_1(x,y)&=u\\
     p_2(x,y)&=v
\end{array}
\right.
~\text{ and }~
\left\{
\begin{array}{rl}
     \tilde{p}_1(x,y)&=u\\
     \tilde{p}_2(x,y)&=v
\end{array}
\right.
\]
and add the two counts to get how many times $(u,v)\in\bbF_s^2$ appears as a row in the two selected columns of $A$. Note that, in this case, the columns of $A$ are indexed by the disjoint union of $\{\ast\}$, $\bbF_s\times\Gamma$ ($L$-column), $\bbF_s\times\Gamma$ ($Q$-column) and $\Xi$ ($E$-column). Hence, we have 9 possible cases combining column types.

\bigskip\noindent\emph{Case 1}. Column indexed by $\{\ast\}$ and column indexed by $\bbF_s\times\Gamma$ ($L$-column).

The systems take the form
\begin{equation}
\left\{
\begin{array}{rl}
     x&=u\\
     ax+\gamma^Ty&=v
\end{array}
\right.
~\text{ and }~
\left\{
\begin{array}{rl}
     x&=u\\
     ax+\gamma^Ty+\left(1-\frac{1}{\omega}\right)\frac{a^2}{4}&=v
\end{array}
\right.
\end{equation}
where $(a,\gamma)\in\bbF_s\times \Gamma$. In both systems, arguing as in case 1 of proof of Theorem~\ref{thm:AKconstruction}, we have that there are $s^{\ell-2}$ solutions for each $(u,v)$, so a total of $2s^{\ell-2}$ in total.

\bigskip\noindent\emph{Case 2}. Column indexed by $\{\ast\}$ and column indexed by $\bbF_s\times\Gamma$ ($Q$-column)

The systems take the form
\begin{equation}
\left\{
\begin{array}{rl}
     x&=u\\
     (x-\beta)^2+\gamma^Ty&=v
\end{array}
\right.
~\text{ and }~
\left\{
\begin{array}{rl}
     x&=u\\
     \omega(x-\beta)^2+\gamma^Ty&=v
\end{array}
\right.
\end{equation}
where $(\beta,\gamma)\in\bbF_s\times \Gamma$. Again, arguing as in case 4 of proof of Theorem~\ref{thm:AKconstruction}, we have that each system has $s^{\ell-2}$ solutions for each $(u,v)$, so a total of $2s^{\ell-2}$.

\bigskip\noindent\emph{Case 3}. Column indexed by $\{\ast\}$ and column indexed by $\Xi$ ($E$-column)

The systems take the form
\begin{equation}
\left\{
\begin{array}{rl}
     x&=u\\
     (x-\beta)^2+ax&=v
\end{array}
\right.
~\text{ and }~
\left\{
\begin{array}{rl}
     x&=u\\
     \omega(x-\beta)^2+ax-\left(1-\frac{1}{\omega}\right)\frac{a^2}{4}&=v
\end{array}
\right.
\end{equation}
where $(a,\beta)\in\Xi$. In both systems, for each $(u,v)$, either there are $s^{\ell-1}$ solutions (corresponding to the free vector of variables $y\in\bbF_s^{\ell-1}$) or none at all. Now, the left-hand side system has solution for $(u,v)$ of the form
\[
\left(u,(u-\beta)^2+au\right),
\]
while the right hand side has, for $(u,v)$, solution of the form
\[
\left(u,\omega(u-\beta)^2+au-\left(1-\frac{1}{\omega}\right)\frac{a^2}{4}\right).
\]
Now, none of these two lists of $(u,v)$ intersect, since otherwise
\[
(u-\beta)^2=\frac{a^2}{4\omega}
\]
would have a solution in $\bbF_s$ contradicting that $\omega$ is a non-square in $\bbF_s$---$a$ is non-zero by election. Hence we have $2s$ $(u,v)$ for which the combined number of solutions is $s^{\ell-1}$ and $s^2-2s$ for which the combined number of solutions is zero.

\bigskip\noindent\emph{Case 4}. Both columns indexed by $\bbF_s\times\Gamma$ ($L$-column)

In this case, we get systems
\begin{equation}
\left\{
\begin{array}{rl}
     ax+\gamma^Ty&=u\\
     \tilde{a}x+\tilde{\gamma}^Ty&=v
\end{array}
\right.
~\text{ and }~
\left\{
\begin{array}{rl}
     ax+\gamma^Ty+\left(1-\frac{1}{\omega}\right)\frac{a^2}{4}&=u\\
     \tilde{a}x+\tilde{\gamma}^Ty+\left(1-\frac{1}{\omega}\right)\frac{\tilde{a}^2}{4}&=v
\end{array}
\right.
\end{equation}
with $(a,\gamma),(\tilde{a},\tilde{\gamma})\in\bbF_s\times\Gamma$. Arguing as in case 2 of proof of Theorem~\ref{thm:AKconstruction}, we have that for each $(u,v)$, each system has $s^{\ell-2}$ solutions, and so a combined number of $2s^{\ell-2}$ solutions.

\bigskip\noindent\emph{Case 5}. Column indexed by $\bbF_s\times\Gamma$ ($L$-column) and column indexed by $\bbF_s\times\Gamma$ ($Q$-column)

The systems take the form
\begin{equation}
\left\{
\begin{array}{rl}
     ax+\gamma^Ty&=u\\
     (x-\beta)^2+\tilde{\gamma}^Ty&=v
\end{array}
\right.
~\text{ and }~
\left\{
\begin{array}{rl}
     ax+\gamma^Ty+\left(1-\frac{1}{\omega}\right)\frac{a^2}{4}&=u\\
     \omega(x-\beta)^2+\tilde{\gamma}^Ty&=v
\end{array}
\right.
\end{equation}
with $(a,\gamma),(\beta,\tilde{\gamma})\in\bbF_s\times\Gamma$. We argue similar to case 4 of proof of Theorem~\ref{thm:AKconstruction}, depending our analysis on whether $\gamma$ and $\tilde{\gamma}$ are equal or different.

If $\gamma\neq\tilde{\gamma}$, then, for each $x\in\bbF_s$, we have, in both the left-hand side and right-hand side system, an affine system of codimension 2 with $s^{\ell-3}$ solutions in $y\in\bbF_s^{\ell-1}$. Hence, for every $(u,v)$, each system has $s^{\ell-2}$ solutions, and so a combined number of $2s^{\ell-2}$ solutions. 

If $\gamma=\tilde{\gamma}$, then we substitute, in each system, the first equation into the second, obtaining the equation
\[
(x-\beta)^2-a(x-\beta)+u-v-a\beta=0
\]
for the left-hand side system, and the equation
\[
\omega(x-\beta)^2-a(x-\beta)+u-v-a\beta-\left(1-\frac{1}{\omega}\right)\frac{a^2}{4}=0
\]
for the right-hand side system. Now, the discriminat (in $x-\beta$) of the first equation is
\[
a^2-4(u-v-a\beta)
\]
while the discriminant of the second one is
\[
\omega\left(a^2-4(u-v-a\beta)\right).
\]
Hence, for each $(u,v)$, one of this quadratic equations has two solutions in $\bbF_s$ if and only if the other has zero, has exactly one if and only if the other has one, and it has zero if and only if the other has two. 

Going back to the original system, this means that for each $(u,v)$, one of the following three things happen: (1) the left-hand side system has $2s^{\ell-2}$ solutions and the right-hand side one zero; (2) the left-hand side and right-hand side systems have both $s^{\ell-2}$ solutions; and (3) the left-hand side system has no solutions and the right-hand side one $2s^{\ell-2}$ solutions. Hence, for each $(u,v)$, no matter the case, the combined number of solutions is $2s^{\ell-2}$.

\bigskip\noindent\emph{Case 6}. Column indexed by $\bbF_s\times\Gamma$ ($L$-column) and column indexed by $\Xi$ ($E$-column)

We get systems of the form
\begin{equation}
\left\{
\begin{array}{rl}
     ax+\gamma^Ty&=u\\
     (x-\beta)^2-\tilde{a}x&=v
\end{array}
\right.
~\text{ and }~
\left\{
\begin{array}{rl}
     ax+\gamma^Ty+\left(1-\frac{1}{\omega}\right)\frac{a^2}{4}&=u\\
     \omega(x-\beta)^2-\tilde{a}x-\left(1-\frac{1}{\omega}\right)\frac{\tilde{a}^2}{4}&=v
\end{array}
\right.
\end{equation}
with $(a,\gamma)\in\bbF_s\times\Gamma$ and $(\tilde{a},\beta)\in\Xi$. To solve these systems, we first solve the second equation and then the corresponding linear equation in $y$.

When solving the quadratic equations, we have the discriminant (with respect $x-\beta$)
\[
\tilde{a}^2+4(\tilde{a}\beta+v)
\]
on the left-hand side, and the discriminant
\[
\omega\left(\tilde{a}^2+4(\tilde{a}\beta+v)\right)
\]
on the right-hand side. This is then the situation of the previous case, having three cases for the system: (1) the left-hand side system has $2s^{\ell-2}$ solutions and the right-hand side one zero; (2) the left-hand side and right-hand side systems have both $s^{\ell-2}$ solutions; and (3) the left-hand side system has no solutions and the right-hand side one $2s^{\ell-2}$ solutions.

Hence, for any $(u,v)$, we conclude that the combined number of solutions is always $2s^{\ell-2}$.

\bigskip\noindent\emph{Case 7}. Both columns indexed by $\bbF_s\times\Gamma$ ($Q$-column)

The systems take the form
\begin{equation}
\left\{
\begin{array}{rl}
     (x-\beta)^2+\gamma^Ty&=u\\
     (x-\tilde{\beta})^2+\tilde{\gamma}^Ty&=v
\end{array}
\right.
~\text{ and }~
\left\{
\begin{array}{rl}
     \omega(x-\beta)^2+\gamma^Ty&=u\\
     \omega(x-\tilde{\beta})^2+\tilde{\gamma}^Ty&=v
\end{array}
\right.
\end{equation}
with $(\beta,\gamma),(\tilde{\beta},\tilde{\gamma})\in\bbF_s\times\Gamma$. This is a similar case to the case 4 above, having two cases depending on whether $\gamma$ and $\tilde{\gamma}$ are equal.

If $\gamma\neq\tilde{\gamma}$, then the same argument as in case 4 work out, having that each system has $s^{\ell-2}$ solutions for each $(u,v)$. Hence we get a combined count of $2s^{\ell-2}$ for each $(u,v)$.

If $\gamma=\tilde{\gamma}$, then we substract the first equation from the second, obtaining the systems
\begin{equation} \label{ec:b17}
\left\{
\begin{array}{rl}
     (x-\beta)^2+\gamma^Ty&=u\\
     2(\beta-\tilde{\beta})x+\tilde{\beta}^2-\beta^2&=v-u
\end{array}
\right.
~\text{ and }~
\left\{
\begin{array}{rl}
     \omega(x-\beta)^2+\gamma^Ty&=u\\
     2\omega(\beta-\tilde{\beta})x+\omega(\tilde{\beta}^2-\beta^2)&=v-u
\end{array}
\right.
\end{equation}
Now, we are picking different columns, so $\beta\neq\tilde{\beta}$. Hence the second equation has always a solution in $x$, which can be substituted in the first equation to get a linear system with $s^{\ell-2}$ solutions.

Hence, for each $(u,v)$, each system has $s^{\ell-2}$ solutions, having a combined number of $2s^{\ell-2}$ solutions.

\bigskip\noindent\emph{Case 8}. Column indexed by $\bbF_s\times\Gamma$ ($Q$-column) and column indexed by $\Xi$ ($E$-column)

The system takes the form
\begin{equation}
\left\{
\begin{array}{rl}
     (x-\beta)^2+\gamma^Ty&=u\\
     (x-\tilde{\beta})^2-ax&=v
\end{array}
\right.
~\text{ and }~
\left\{
\begin{array}{rl}
     \omega(x-\beta)^2+\gamma^Ty&=u\\
     \omega(x-\tilde{\beta})^2-ax-\left(1-\frac{1}{\omega}\right)\frac{a^2}{4}&=v
\end{array}
\right.
\end{equation}
with $(\beta,\gamma)\in\bbF_s\times\Gamma$ and $(a,\tilde{\beta})\in\Xi$. This is a case similar to case 6 above, where we solve the second equation and then the corresponding linear equation in $y$. Even more, we have the same second equation, up to changing $\beta$ by $\tilde{\beta}$, so the same discussion applies.

Hence, arguing as in case 6, we conclude that for each $(u,v)$, the combined total of solution is $2s^{\ell-2}$.

\bigskip\noindent\emph{Case 9}. Both columns indexed by $\Xi$ ($E$-column)

The system takes the form
\begin{equation}
\left\{
\begin{array}{rl}
     (x-\beta)^2-ax&=u\\
     (x-\beta)^2-\tilde{a}x&=v
\end{array}
\right.
~\text{ and }~
\left\{
\begin{array}{rl}
     \omega(x-\beta)^2-ax-\left(1-\frac{1}{\omega}\right)\frac{a^2}{4}&=u\\
     \omega(x-\beta)^2-\tilde{a}x-\left(1-\frac{1}{\omega}\right)\frac{\tilde{a}^2}{4}&=v
\end{array}
\right.
\end{equation}
with $(a,\beta),(\tilde{a},\beta)\in\Xi$---note that the second component is the same by construction of $\Xi$. Moreover, without loss of generality, we can assume that $\beta=0$ after changing $x$ by $x+\beta$, $u$ by $u-a\beta$, and $v$ by $v-\tilde{a}\beta$. Hence the pair of systems becomes:
\begin{equation}
\left\{
\begin{array}{rl}
     x^2-ax&=u\\
     x^2-\tilde{a}x&=v
\end{array}
\right.
~\text{ and }~
\left\{
\begin{array}{rl}
     \omega x^2-ax-\left(1-\frac{1}{\omega}\right)\frac{a^2}{4}&=u\\
     \omega x^2-\tilde{a}x-\left(1-\frac{1}{\omega}\right)\frac{\tilde{a}^2}{4}&=v
\end{array}
\right.
\end{equation}

Now, for each $(u,v)$, the combined number of solutions is a multiple of $S^{\ell-1}$ since $y$ is free in both systems. Thus we only need to count the number of solutions in $x$.

For each system, the number of solutions in $x$ is either one or zero. Otherwise, we would have a system with the same quadratic equation twice, which is not the case as $a\neq\tilde{a}$ by construction. Now, by looking at the discriminants, we have that if the left system has a solution, then
\[
a^2-u\text{ and }\tilde{a}^2-v
\]
are squares in $\bbF_s$, and if the right system has a solution, then
\[
\omega(a^2-u)\text{ and }\omega(\tilde{a}^2-v)
\]
are squares in $\bbF_s$. Hence, if both systems have solutions, we necessary have
\[
u=a^2\text{ and }v=\tilde{a}^2.
\]
But then, for this choice of $(u,v)$, neither the left nor the right system have a solution---since $a\neq\tilde{a}$ by construction. Hence either the left system has a solution or the right system has one, but not both.

Now, the left system has a solution in $x$ if and only if the quadratic polynomials have a common solution. By Sylvester's resultant~\cite[pp.~155--156]{coxlittleoshea2015}, the latter happens if and only if
\begin{equation}
Q(u,v):=\det
\begin{pmatrix}
1&-a&-u&\\
&1&-a&-u\\
1&-\tilde{a}&-v&\\
&1&-\tilde{a}&-v
\end{pmatrix}=(u-v)^2+(a-\tilde{a})(\tilde{a}u-av)
\end{equation}
vanishes---this polynomial vanishes if and only if $x^2-ax=u$ and $x^2-\tilde{a}x=v$ have a common zero.

Now, a simple computation, show that $Q$ defines a non-empty non-singular projective quadratic curve. Hence, parameterizing it using the straight lines passing through its point $(0,0)$, we deduce that it has exactly $s$ in $\bbF_s^2$. In other words, there are $s$ $(u,v)$s for which the left system has $s^{\ell-1}$ solutions (the solution in $x$ gives rise to $s^{\ell-1}$ in $(x,y)$ as $y$ can take $s^{\ell-1}$ values); and for the rest of the $(u,v)$, the left system does not have any solution at all.

Similarly, one can show that that the left system has solution if and only if 
\begin{equation}
\tilde{Q}(u,v)=Q_1\left(u+\left(1-\frac{1}{\omega}\right)\frac{a^2}{4},v+\left(1-\frac{1}{\omega}\right)\frac{\tilde{a}^2}{4}\right)
\end{equation}
vanishes. $\tilde{Q}$ defines again a non-empty non-singular projective quadratic planar curve---its just a translate of the curve defined by $Q$ Hence, we have $s$ $(u,v)$s for which the right system has $s^{\ell-1}$ solutions; and for the rest of the $(u,v)$, the left system does not have any solution at all.

Finally, since both systems cannot have solutions at the same time, there are $2s$ values of $(u,v)$ for which the combined number of solutions is $s^{\ell-1}$ and $s(s-2)$ for which the combined number of solutions is zero.

We have gone through all possible combinations of $2$ columns of $A$. Now, there are only two cases in which not all $(u,v)$ appear the same number of times as a row: case 3 and case 9.

Hence, doing the counting as in the proof of Theorem~\ref{thm:AKconstruction}, we have that, in case 3,
\[
\sum_{(a,\beta)}\sum_{(u,v)}\left|n(A,(\ast,(a,\beta)),(u,v))-2s^{\ell-2}\right|^p=\tilde{\kappa}\left(2s\left((s-2)s^{\ell-2}\right)^{p}+s(s-2)(2s^{\ell-2})^{p}\right);
\]
and, in case 9,
\begin{multline*}
    \sum_{(a,\beta),(\tilde{a},\tilde{\beta})}\sum_{(u,v)}\left|n(A,(\ast,(a,\beta)),(u,v))-2s^{\ell-2}\right|^p\\=\binom{\tilde{\kappa}}{2}\left(2s\left((s-2)s^{\ell-2}\right)^{p}+s(s-2)(2s^{\ell-2})^{p}\right).
\end{multline*}
Adding all up, and simplifying, gives the desired unbalance.
\end{proof}

\begin{proof}[Proof of Theorem~\ref{thm:AKconstructionB} for even $s$]
As in the proof of Theorem~\ref{thm:AKconstruction} for even $s$, we just re-analyze the cases of the proof of Theorem~\ref{constructionAKB} for odd $s$ in which changes have to me made to the arguments.

Observe that a change to the argument has to be made only in the cases where we have to analyze the solutions of a quadratic polynomial. In the remainder of the cases, this is not needed, as characteristic two does not play a role. Said this, the cases that we will re-analyze are the cases 3, 5, 6, 7, 8 and 9 of the Proof of Theorem~\ref{thm:AKconstructionB} for odd $s$, since these are the ones that are significantly different. As this proof will be similar, we will not give as many details as in that proof.

Before continuing with the proof, we recall the map $T$ and $L$ given respectively by the formulas~\eqref{eq:trace} and \eqref{eq:cotrace} in the proof of Theorem~\ref{thm:AKconstruction} for even $s$. In this proof, we proved the following fact:
\begin{quote}
    For $\delta\neq 0$, the quadratic polynomial $x^2+\delta x+\varepsilon\in\bbF_s[x]$ has either two roots in $\bbF_s$ or none. Moreover, it has two roots in $\bbF_s$ if and only if $T\left(\varepsilon/\delta^2\right)=0$.
\end{quote}
We will use this fact to analyze all the quadratic equations that appear in the analysis. Now, note that the choice of $\zeta$ forces $T(\zeta)=1$. Otherwise, $\mathrm{span}_{\bbF_2}\{\zeta^a\mid a\geq 0\}$ would be a propper subfield of $\bbF_s$ contradicting the fact that $\zeta^{s/2}\neq \zeta$.

\bigskip\noindent\emph{Case 3}. Column indexed by $\{\ast\}$ and column indexed by $\Xi$ ($E$-column)

The systems take the form
\begin{equation}
\left\{
\begin{array}{rl}
     x&=u\\
     x^2+\delta x&=v
\end{array}
\right.
~\text{ and }~
\left\{
\begin{array}{rl}
     x&=u\\
     x^2+\delta x+\delta^2\zeta&=v
\end{array}
\right.
\end{equation}
where $\delta\in\Xi$. By substituting the first equation into the second one, we get that the first system has $s^{\ell-1}$ (the variable $y$ is free) for values of the form
\begin{equation}
    \left(u,u^2+\delta u\right)
\end{equation}
and zero otherwise. Similarly, the second system has $s^{\ell-1}$ solutions for values of the form
\begin{equation}
    \left(u,u^2+\delta u+\delta^2\zeta\right)
\end{equation}
and zero otherwise. Now, $\delta\neq 0$, so these two families of $(u,v)$ pairs are disjoint. Hence for $2s$ values of $(u,v)$, the combined number of solutions is $s^{\ell-1}$; and for the remaining $s(s-2)$ values of $(u,v)$, the combined number of solutions is zero.

\bigskip\noindent\emph{Case 5}. Column indexed by $\bbF_s\times\Gamma$ ($L$-column) and column indexed by $\bbF_s\times\Gamma$ ($Q$-column)

The systems take the form
\begin{equation}
\left\{
\begin{array}{rl}
     ax+\gamma^Ty&=u\\
     x^2+\beta x+\tilde{\gamma}^Ty&=v
\end{array}
\right.
~\text{ and }~
\left\{
\begin{array}{rl}
     ax+\gamma^Ty+a^2\zeta&=u\\
     x^2+\beta x+\tilde{\gamma}^Ty+\beta^2\zeta&=v
\end{array}
\right.
\end{equation}
with $(a,\gamma),(\beta,\tilde{\gamma})\in\bbF_s\times\Gamma$. There are two cases: $\gamma\neq\tilde{\gamma}$ and $\gamma=\tilde{\gamma}$.

If $\gamma\neq \tilde{\gamma}$, then, as done in many cases before, we solve the corresponding linear system in $y$ for each of the possible values of $x$. Hence we get $2s^{\ell-2}$ combined solutions for each $(u,v)$.

If $\gamma=\tilde{\gamma}$, we add the first equation to the second obtaining, respectively\footnote{Recall that we are in characteristic two, so we can ignore signs.},
\[
x^2+(a+\beta) x+u+v=0
\]
and
\[
x^2+(a+\beta) x+u+v+(a+\beta)^2\zeta=0.
\]
Now, the first quadratic equation has two solutions if and only if $T((u+v)/(a+\beta)^2)=0$, and the second one if and only if $T((u+v)/(a+\beta)^2)+1=T((u+v)/(a+\beta)^2+\zeta)=0$. Thus one of the quadratic equations has two solutions if and only if the other one has zero. Therefore we conclude that there are $2s^{\ell-2}$ combined solutions for each $(u,v)$.

\bigskip\noindent\emph{Case 6}. Column indexed by $\bbF_s\times\Gamma$ ($L$-column) and column indexed by $\Xi$ ($E$-column)

We get systems of the form
\begin{equation}
\left\{
\begin{array}{rl}
     ax+\gamma^Ty&=u\\
     x^2+\delta x&=v
\end{array}
\right.
~\text{ and }~
\left\{
\begin{array}{rl}
     ax+\gamma^Ty+a^2\zeta&=u\\
     x^2+\delta x+\delta^2\zeta&=v
\end{array}
\right.
\end{equation}
with $(a,\gamma)\in\bbF_s\times\Gamma$ and $\delta\in\Xi$. To solve these systems, we first solve the second equation and then the corresponding linear equation in $y$.

Now, the first quadratic equation has two solutions if and only if $T(v/\delta^2)=0$, and the second one if and only if $T(v/\delta^2)+1=T(v/\delta^2+\zeta)=0$. So one has two solutions if and only if the other has zero. Thus one system has $2s^{\ell-2}$ solutions if and only if the other one has zero, and we conclude that we have a combined number of $2s^{\ell-2}$ combined solutions for each $(u,v)$.

\bigskip\noindent\emph{Case 7}. Both columns indexed by $\bbF_s\times\Gamma$ ($Q$-column)

The systems take the form
\begin{equation}
\left\{
\begin{array}{rl}
     x^2+\beta x+\gamma^Ty&=u\\
     x^2+\tilde{\beta} x+\tilde{\gamma}^Ty&=v
\end{array}
\right.
~\text{ and }~
\left\{
\begin{array}{rl}
     x^2+\beta x+\gamma^Ty+\beta^2\zeta&=u\\
     x^2+\tilde{\beta} x+\tilde{\gamma}^Ty+\tilde{\beta}^2\zeta&=v
\end{array}
\right.
\end{equation}
with $(\beta,\gamma),(\tilde{\beta},\tilde{\gamma})\in\bbF_s\times\Gamma$. This case is analogous to the case 5. If $\gamma\neq \tilde{\gamma}$, we solve the corresponding linear system in $y$ for each $x$ getting a total of $2s^{\ell-2}$ combined solutions for each $(u,v)$. If $\gamma=\tilde{\gamma}$, we add equations to reduce to the quadratic case and deduce that one system has $2s^{\ell-2}$ solutions if and only if the other one has zero solutions. Hence we conclude again that we have $2s^{\ell-2}$ combined solutions for each $(u,v)$.

\bigskip\noindent\emph{Case 8}. Column indexed by $\bbF_s\times\Gamma$ ($Q$-column) and column indexed by $\Xi$ ($E$-column)

The system takes the form
\begin{equation}
\left\{
\begin{array}{rl}
     x^2+\beta x+\gamma^Ty&=u\\
     x^2+\delta x&=v
\end{array}
\right.
~\text{ and }~
\left\{
\begin{array}{rl}
     x^2+\beta x+\gamma^Ty+\beta^2\zeta&=u\\
     x^2+\delta x+\delta^2\zeta&=v
\end{array}
\right.
\end{equation}
with $(\beta,\gamma)\in\bbF_s\times\Gamma$ and $\delta\in \Xi$. This cases is analogous to case 6, we first solve the quadratic equation in $x$ and then the linear equation in $y$. Almost the same argument gives that one system has $2s^{\ell-2}$ solutions if and only if the other one has zero. Hence the combined total number of solutions is $2s^{\ell-2}$ independently of $(u,v)$.

\bigskip\noindent\emph{Case 9}. Both columns indexed by $\Xi$ ($E$-column)

The system takes the form
\begin{equation}
\left\{
\begin{array}{rl}
     x^2+\delta x&=u\\
     x^2+\tilde{\delta} x&=v
\end{array}
\right.
~\text{ and }~
\left\{
\begin{array}{rl}
     x^2+\delta x+\delta^2\zeta&=u\\
     x^2+\tilde{\delta} x+\tilde{\delta}^2\zeta&=v
\end{array}
\right.
\end{equation}
with $\delta,\tilde{\delta}\in\Xi$ such that $\delta\neq\tilde{\delta}$. 

Each system has a number of solutions that is $s^{\ell-2}$ (number of possible values that the free $y$ can take) times the number of solutions in $x$. Moreover, each system has either one or zero solutions in $x$, since the two quadrics that form the system are different, and so cannot have more than one root in common. Even more, both systems cannot have a solution in $x$ at the same time, otherwise
\[
T(u/\delta^2)=T(v/\tilde{\delta}^2)=T(u/\delta^2+\zeta)=T(v/\tilde{\delta}^2+\zeta)=0,
\]
which is not possible as $T(\zeta)\neq 0$. Hence, as in the odd case, we have just to count for how many $(u,v)$, the left system has one solution in $x$, and for how many $(u,v)$, the right system has one solution in $x$. 

By looking at each of the quadratic equations in each of the systems (as we have done before), we have that for each $(u,v)$, at most one system has $s^{\ell-1}$ solutions (note that they have at most a common solution, since they are different quadratic polynomials; and $s^{\ell-1}$ come from the free variable $y$), being possible that none of the systems have solutions. 

Now, any of the systems have a solution in $x$ if and only if the conforming quadratics have a common root. Thus, by the Sylvester resultant~\cite[pp.~155--156]{coxlittleoshea2015}, the left system has a solution in $x$ if and only if 
\[
Q(u,v)=\det\begin{pmatrix}
    1&\delta&u&\\
    &1&\delta&u\\
    1&\tilde{\delta}&v&\\
    &1&\tilde{\delta}&v
\end{pmatrix}=(u+v)^2+(\delta+\tilde{\delta})(\tilde{\delta}u+\delta v)
\]
vanishes; and the right system has a solution in $x$ if and only if
and
\[
\tilde{Q}(u,v):=Q(u+\delta^2\zeta,v+\tilde{\delta}^2\zeta)=\det\begin{pmatrix}
    1&\delta&u+\delta^2\zeta&\\
    &1&\delta&u+\delta^2\zeta\\
    1&\tilde{\delta}&v+\tilde{\delta}^2\zeta&\\
    &1&\tilde{\delta}&v+\tilde{\delta}^2\zeta
\end{pmatrix}
\]
vanishes.

One can easily check that both $Q$ and $\tilde{Q}$ define non-empty non-singular projective quadratic planar curves. Thus, considering all straight lines passing through a point\footnote{Just consider all straight lines passing through a point of the curve, each will intersect the curve at another point except the tangent curve. This gives a bijection between the projective line and the projective points of the curve, and so between $\bbF_s$ with the affine points of the curve.}, we get that both these curves have $s$ points in the affine plane $\bbF_s$. 

Hence, by the above paragraphs, there are $s$ values of $(u,v)$ for which the left system has $s^{\ell-1}$ solutions; and $s$ values of $(u,v)$ for which the right system has $s^{\ell-1}$ solutions. As this cannot happen at the same time, there are $2s$ values of $(u,v)$ for which the combined number of solutions is $s^{\ell-1}$ and $s(s-2)$ for which the combined number of solutions is zero.


Now, the proof is analogous to that of Theorem~\ref{thm:AKconstructionB} for odd $s$, since the same cases (3 and 9) contribute to the unbalance and they do so with the same number in terms of $s$, $\ell$ and $p$. 
\end{proof}

\section{Proofs of Appendices~\ref{sec:unbalancebounds} and \ref{sec:comparative}}\label{sec:unbalanceproofs}

\begin{proof}[Proof of Theorem~\ref{thm:HamminSimilarityvsUnbalance}]
We include this proof only for completeness. We follow the proof of \cite[Lemma~3.1]{UDpaper2}, which is an application of double counting. Fix the array $A\in S^{N\times k}$. For $r\in [N]$, $j\in [k]$ and $x\in S$, consider the numbers
\begin{equation}
    z_{r,j}(x)=\begin{cases}
    1&\text{if }A_{r,j}=x\\
    0&\text{if }A_{r,j}\neq x.
\end{cases}
\end{equation}
Then we have that
\begin{equation}\label{eq:indicatorsum1}
    \HammingSimilarity(A;r,\tilde{r})=\sum_{j\in [k]}\sum_{x\in S}z_{r,j}(x)z_{\tilde{r},j}(x)
\end{equation}
and that
\begin{equation}\label{eq:indicatorsum2}
    \binom{\HammingSimilarity(A;r,\tilde{r})}{t}=\sum_{j\in T_{t,k}}\sum_{x\in S^t}z_{r,j_1}(x_1)z_{\tilde{r},j_1}(x_1)\cdots z_{r,j_t}(x_t)z_{\tilde{r},j_t}(x_t).
\end{equation}
To see this, note that $z_{r,j_1}(x_1)z_{\tilde{r},j_1}(x_1)\cdots z_{r,j_t}(x_t)z_{\tilde{r},j_t}(x_t)=1$ if and only if the $r$th and $\tilde{r}$th runs of $A[j]$ equal $x$. Therefore $\sum_{x\in S^t}z_{r,j_1}(x_1)z_{\tilde{r},j_1}(x_1)\cdots z_{r,j_t}(x_t)z_{\tilde{r},j_t}(x_t)=1$ if the $r$th and $\tilde{r}$th runs of $A[j]$ are equal, i.e., if $j$ corresponds to a subset of the factors where the $r$th and $\tilde{r}$th runs of $A$ agree.

Now, we also have that
\begin{align*}
    \Unb_{2,t}(A)&=\sum_{x\in S^{t}}\sum_{j\in T_{t,k}}\left|n(A,x,j)-N/s^t\right|^2\\
    &=\sum_{x\in S^{t}}\sum_{j\in T_{t,k}}n(A,x,j)^2-2\sum_{x\in S^{t}}\sum_{j\in T_{t,k}}n(A,x,j)\frac{N}{s^t}+\sum_{x\in S^{t}}\sum_{j\in T_{t,k}}\left(\frac{N}{s^t}\right)^2\\
    &=\sum_{x\in S^{t}}\sum_{j\in T_{t,k}}n(A,x,j)^2-2\binom{k}{t}N\frac{N}{s^t}+s^t \binom{k}{t}\left(\frac{N}{s^t}\right)^2\\
    &=\sum_{x\in S^{t}}\sum_{j\in T_{t,k}}n(A,x,j)^2-\binom{k}{t}\frac{N^2}{s^t}
\end{align*}
where, in the fourth equality, we have used that for $j\in T_{t,k}$, $\sum_{x\in S^t}n(A,x,j)=N$. Hence, we only need to show that
\begin{equation}
    \sum_{x\in S^{t}}\sum_{j\in T_{t,k}}n(A,x,j)^2=\sum_{r,\tilde{r}}\sum_{j\in T_{t,k}}\sum_{x\in S}z_{r,j_1}(x_1)z_{\tilde{r},j_1}(x_1)\cdots z_{r,j_t}(x_t)z_{\tilde{r},j_t}(x_t).
\end{equation}
But, for each $x\in S^t$ and $j\in T_{t,k}$,
\begin{align*}
    \sum_{r,\tilde{r}}z_{r,j_1}(x_1)z_{\tilde{r},j_1}(x_1)\cdots z_{r,j_t}(x_t)z_{\tilde{r},j_t}(x_t)&=\sum_{r,\tilde{r}}\left(z_{r,j_1}(x)\cdots z_{r,j_t}(x_t)\right)\left(z_{\tilde{r},j_1}(x_1)\cdots z_{\tilde{r},j_t}(x_t)\right)\\
    &=\left(\sum_{r}z_{r,j_1}(x)\cdots z_{r,j_t}(x_t)\right)^2\\
    &=n(A,x,j)^2,
\end{align*}
since $z_{r,j_1}(x)\cdots z_{r,j_t}(x_t)=1$ if and only if the $r$th run of $A[j]$ is $x$, meaning we are counting the runs of $A[j]$ equal to $x$. Hence the desired identity follows.
\end{proof}

\begin{proof}[Proof of Corollary~\ref{cor:HamminSimilarityvsUnbalance1}]
This is a consequence of \cite[Lemma~2]{UDpaper8liuhickernell} applied to
\begin{equation}
    U_{2,2}(A)=2\sum_{(r,\tilde{r})\in T_{2,\lambda s^2}}\binom{\HammingSimilarity(A;r,\tilde{r})}{2}-\lambda (\lambda -1)s^2\binom{k}{2},
\end{equation}
and noting that $\sum_{(r,\tilde{r})\in T_{2,\lambda s^2}}\HammingSimilarity(A;r,\tilde{r})=\frac{1}{2}\Unb_{2,1}(A)+\binom{\lambda s^2}{2}k$.
The right-hand side first sum is lower bounded by
\begin{equation}\label{eq:intminimum}
    \min\left\{2\sum_{(r,\tilde{r})\in T_{2,N\lambda s^2}}\binom{h_{r,\tilde{r}}}{2}\mid h_{r,\tilde{r}}\in \bbZ\text{ and }\sum_{(r,\tilde{r})\in T_{2,\lambda s^2}}h_{r,\tilde{r}}=\frac{1}{2}\Unb_{2,1}(A)+\binom{\lambda s^2}{2}k\right\}
\end{equation}
with equality if and only if, for
\begin{equation}\label{eq:tildegamma}
    \tilde{\gamma}=\frac{\Unb_{1,2}(A)}{\lambda s^2(\lambda s^2-1)}+\gamma,
\end{equation}
there is a subset $R\subset T_{N,2}$ of size $\binom{N}{2}(1+\left\lfloor\tilde{\gamma}\right\rfloor-\tilde{\gamma})$ such that for $(r,\tilde{r})\in R$, $\HammingSimilarity(r,\tilde{r})=\left\lfloor\tilde{\gamma}\right\rfloor$ and for $(r,\tilde{r})\in T_{N,2}\setminus R$, $\HammingSimilarity(A;r,\tilde{r})=\left\lfloor\tilde{\gamma}\right\rfloor+1$.

Now, the expression for the minimum in~\eqref{eq:intminimum} is still
\begin{equation}
    \lambda s^2(\lambda s^2-1)\left(\binom{\left\lfloor\tilde{\gamma}\right\rfloor}{2}+\left\lfloor\tilde{\gamma}\right\rfloor\left(\tilde{\gamma}-\left\lfloor\tilde{\gamma}\right\rfloor\right)\right)
\end{equation}
which (surprisingly!) is a strictly growing continuous function on $\tilde{\gamma}$. Hence the final corollary follows from noting that $\tilde{\gamma}\geq \gamma$ with equality if and only if $A$ is an $OA(N,k,s,1)$. 
\end{proof}

\begin{remark}
The only reason the above proof does not extend for $t>2$ is because the map $x\mapsto \binom{x}{t}$ is strictly convex only for $t=2$. For $t=2$, it is strictly convex on $x\geq t-1$; but it is not clear that we can force this condition on the $\HammingSimilarity(A;r,\tilde{r})$.
\end{remark}

\begin{proof}[Proof of Corollary~\ref{cor:HamminSimilarityvsUnbalance2}]
We just substitute the possible values in Corollary~\ref{cor:HamminSimilarityvsUnbalance1} and compute.
\end{proof}

\begin{proof}[Proof of Proposition~\ref{prop:columnorthogonality}]
Without loss of generality, after scaling $f$ by a constant, assume $\sum_{x\in S}f(x)^2=1$. Note that scaling $f$ by a constant doesn't affect the matrix $X(A,f)$.

If $A$ is an $OA(N,k,s,2)$, then $A$ is also an $OA(N,k,s,1)$ and so
\begin{equation}\label{eq:normf}
\sum_{j=1}^N f(A_{i,j})^2=\sum_{x\in S} n(A,x,j)\,f(x)^2=\sum_{x\in S} N/s\,f(x)^2=N/s,
\end{equation}
and, for $i\neq j$,
\begin{align*}
\begin{pmatrix}
    X(A,f)^TX(A,f)
\end{pmatrix}_{i,j}
&=\frac{s}{N}\sum_{k=1}^N f(A_{k,i})f(A_{k,j})&\text{(\eqref{eq:Xmatrix} \& \eqref{eq:normf})}\\
&=\frac{s}{N}\sum_{(x,y)\in S^2} n(A,(x,y),(i,j))f(x)f(y)\\
&=\frac{s}{N}\sum_{(x,y)\in S^2} \frac{N}{s^2}\, f(x)f(y)&(A OA(N,k,s,2))\\
&=\frac{1}{s}\sum_{(x,y)\in S^2}  f(x)f(y)\\
&=\frac{1}{s}\left(\sum_{x\in S} f(x)\right)\left(\sum_{y\in S}f(y)\right)\\
&=0&\text{\eqref{eq:fnullaverage}}
\end{align*}
Hence $X(A,f)^TX(A,f)=\bbI$, as we wanted to show.
\end{proof}

\begin{proof}[Proof of Proposition~\ref{prop:Dcriterion}]
If $A$ is an $OA(N,k,s,2)$, then, by Proposition~\ref{prop:columnorthogonality}, we have that $X(A,f)^TX(A,f)=\bbI$ and so $\mcD_f(A)=1$.
\end{proof}

\begin{proof}[Proof of Theorem~\ref{theo:DcriterionvsUnbalance}]
As in the proof of Proposition~\ref{prop:columnorthogonality}, assume that $\sum_{x\in S}f(x)^2=1$. Then \eqref{eq:normf} holds, because $A$ is an $OA(N,k,s,1)$, and so, for $i\neq j$,
\begin{align*}
\begin{pmatrix}
    X(A,f)^TX(A,f)
\end{pmatrix}_{i,j}
&=\frac{s}{N}\sum_{k=1}^N f(A_{k,i})f(A_{k,j})&\text{(\eqref{eq:Xmatrix} \& \eqref{eq:normf})}\\
&=\frac{s}{N}\sum_{(x,y)\in S^2} n(A,(x,y),(i,j))f(x)f(y)\\
&=\frac{s}{N}\sum_{(x,y)\in S^2} \left(n(A,(x,y),(i,j))-\frac{N}{s^2}\right)f(x)f(y)\\&\hspace{3cm}+\frac{s}{N}\sum_{(x,y)\in S^2} \frac{N}{s^2}\, f(x)f(y)\\
&=\frac{s}{N}\sum_{(x,y)\in S^2} \left(n(A,(x,y),(i,j))-\frac{N}{s^2}\right)f(x)f(y)\\&\hspace{3cm}+\frac{1}{s}\left(\sum_{x\in S} f(x)\right)\left(\sum_{y\in S}f(y)\right)&\\
&=\frac{s}{N}\sum_{(x,y)\in S^2} \left(n(A,(x,y),(i,j))-\frac{N}{s^2}\right)f(x)f(y)&\text{\eqref{eq:fnullaverage}},
\end{align*}
and so, by the Cauchy-Schwarz inequality and \eqref{eq:normf}, for $i\neq j$,
\begin{equation}
    \left|\begin{pmatrix}
    X(A,f)^TX(A,f)
\end{pmatrix}_{i,j}\right|\leq \frac{s}{N}\sqrt{\sum_{(x,y)\in S^2} \left|n(A,(x,y),(i,j))-\frac{N}{s^2}\right|^2}
\end{equation}
and thus
\begin{multline}
    \left\|X(A,f)^TX(A,f)-\bbI\right\|_F\\\leq \frac{s}{N}\sqrt{2\sum_{(i,j)\in T_{2,k}}\sum_{(x,y)\in S^2} \left|n(A,(x,y),(i,j))-\frac{N}{s^2}\right|^2}=\frac{s}{N}\sqrt{2\Unb_{2,2}(A)},
\end{multline}
which proves~\eqref{eq:XorthvsUnb}, since $s/N=1/(\lambda s)$.

To prove~\eqref{eq:XorthvsTolgen}, assume, for now, that $\sum_{x\in S}f(x)^2=1$. As shown above, for $i\neq j$, we have
\begin{equation}
    \begin{pmatrix}
    X(A,f)^TX(A,f)
\end{pmatrix}_{i,j}=\frac{s}{N}\sum_{(x,y)\in S^2} \left(n(A,(x,y),(i,j))-\lambda\right)f(x)f(y).
\end{equation}
Now, add and substract $\alpha$ to both $f(x)$ and $f(y)$, so that we get
\begin{multline}\label{eq:auxexpansion}
    \begin{pmatrix}
    X(A,f)^TX(A,f)
\end{pmatrix}_{i,j}
=\frac{s}{N}\sum_{(x,y)\in S^2} \left(n(A,(x,y),(i,j))-\lambda\right)(f(x)-\alpha)(f(y)-\alpha)\\
+\frac{s}{N}\sum_{(x,y)\in S^2} \left(n(A,(x,y),(i,j))-\lambda\right)\alpha(f(y)-\alpha)\\
+\frac{s}{N}\sum_{(x,y)\in S^2} \left(n(A,(x,y),(i,j))-\lambda\right)(f(x)-\alpha)\alpha\\
+\frac{s}{N}\sum_{(x,y)\in S^2} \left(n(A,(x,y),(i,j))-\lambda\right)\alpha^2
\end{multline}
Now, since $A$ is an $OA(N,k,s,1)$,
\begin{equation}\label{eq:auxauxsum1}
    \sum_{x\in S}n(A,(x,y),(i,j))=n(A,y,j)=\lambda s
\end{equation}
and
\begin{equation}\label{eq:auxauxsum2}
    \sum_{y\in S}n(A,(x,y),(i,j))=n(A,x,i)=\lambda s.
\end{equation}
Thus, by \eqref{eq:auxauxsum1},
\begin{multline}\label{eq:auxsum1}
    \sum_{(x,y)\in S^2} \left(n(A,(x,y),(i,j))-\lambda\right)\alpha(f(y)-\alpha)\\
    =\sum_{y\in S}\left(\sum_{x\in S}n(A,(x,y),(i,j))-\lambda\right)\alpha(f(y)-\alpha)
    =0
\end{multline}
and, by \eqref{eq:auxauxsum2},
\begin{multline}\label{eq:auxsum2}
    \sum_{(x,y)\in S^2} \left(n(A,(x,y),(i,j))-\lambda\right)(f(x)-\alpha)\alpha\\
    =\sum_{x\in S}\left(\sum_{y\in S}n(A,(x,y),(i,j))-\lambda\right)(f(x)-\alpha)\alpha
    =0.
\end{multline}
Also, by \eqref{eq:countingsum},
\begin{equation}\label{eq:auxsum3}
    \sum_{(x,y)\in S^2} \left(n(A,(x,y),(i,j))-\lambda\right)=0
\end{equation}

Hence, combining \eqref{eq:auxexpansion} with \eqref{eq:auxsum1}, \eqref{eq:auxsum2} and \eqref{eq:auxsum3}, we get
\begin{equation}\label{eq:zerodoublesum}
    \begin{pmatrix}
    X(A,f)^TX(A,f)
\end{pmatrix}_{i,j}
=\frac{s}{N}\sum_{(x,y)\in S^2} \left(n(A,(x,y),(i,j))-\lambda\right)(f(x)-\alpha)(f(y)-\alpha).
\end{equation}
Hence
\begin{equation}
    \left|\begin{pmatrix}
    X(A,f)^TX(A,f)
\end{pmatrix}_{i,j}\right|
\leq \frac{s}{\lambda}\Tol_2(A)\left(\frac{1}{s}\sum_{x\in S}|f(x)-\alpha|\right)^2,
\end{equation}
where optimizing with respect $\alpha$ gives the squared average absolute deviation, $\sigma_1(f)^2$, on the right hand side. 

Now, for general $f$, we don't have $\sum_{x\in S}f(x)^2=1$. However, $\frac{1}{\sqrt{s}\sigma_2(f)}f$ would satisfy that condition, and $\sigma_1\left(\frac{1}{\sqrt{s}\sigma_2(f)}f\right)=\frac{1}{\sqrt{s}}\frac{\sigma_1(f)}{\sigma_2(f)}$. From here, the claim follows, since $X(A,f)$ does not vary under scaling of $f$ by constants.
\end{proof}

\begin{proof}[Proof of Corollary~\ref{cor:DcriterionvsUnbalance}]
By assumption, we have that
\begin{equation}\label{eq:XAfspecialform}
    X(A,f)^TX(A,f)=\begin{pmatrix}
        1&v^T\\v&\bbI
    \end{pmatrix}
\end{equation}
with $v$ having at most $r$ non-zero components of absolute value at most $\frac{1}{\lambda}\left(\frac{\sigma_1(f)}{\sigma_2(f)}\right)^2\Tol_2(A)$. Hence, by~\cite[Proposition 1]{brentosbornsmith2015det}, we have that
\begin{equation}
    \det (X(A,f)^TX(A,f))= 1-\sum_i v_i^2\geq 1-\frac{r}{\lambda^2}\left(\frac{\sigma_1(f)}{\sigma_2(f)}\right)^4\Tol_2(A)^2,
\end{equation}
and so, given the assumption~\eqref{eq:DcritvsTolcond}, the result follows.
\end{proof}
\begin{proof}[Proof of Proposition~\ref{prop:sigmaquotient}]
The right-hand side of \eqref{eq:sigmaquotient} follows from the inequality between the arithmetic and quadratic means, which implies $\sigma_1(f)\leq \sigma_2(f)$. 

To prove the left-hand side of \eqref{eq:sigmaquotient}, assume that this is false and there is $f$ such that $\left(\sigma_1(f)/\sigma_2(f)\right)^2<1/(s-1)$. Then, let $A$ be the array of Proposition~\ref{prop:trivialconstruction} with $s$ levels. Then, by Corollary~\ref{cor:DcriterionvsUnbalance}, $\mcD_f(A)>0$. But, $\mcD_f(A)=0$, because, by construction, $A$ was constructed repeating one column. Hence a contradiction, proving the desired inequality.

To achieve the lower bound of \eqref{eq:sigmaquotient} take $f:[s]\rightarrow \bbR$ such that $f(1)=-(s-1)$ and $f(a)=1$ for $a\neq 1$. To achieve the upper bound, take any $f:[s]\rightarrow \{-1,0,1\}$ such that $f^{-1}(-1)$ and $f^{-1}(1)$ have the same size.
\end{proof}
\begin{proof}[Proof of Proposition~\ref{prop:phithetacriterionvsunbalance}]
Recall that
\[
D_{\alpha,\beta}(A)=\sum_{j\in T_{2,k}}\left(\sum_{x\in S^2}|n(A,x,j)-N/s^2|^{\alpha}\right)^\beta.
\]
Now, we have that
\begin{multline}\label{eq:auxmeanineq}
    \sum_{x\in S^2}|n(A,x,j)-N/s^2|^{\alpha\beta}\leq \left(\sum_{x\in S^2}|n(A,x,j)-N/s^2|^{\alpha}\right)^\beta \\\leq s^{2(\beta-1)}\sum_{x\in S^2}|n(A,x,j)-N/s^2|^{\alpha\beta},
\end{multline}
by the usual inequalities between the $\alpha$-mean and the $\alpha\beta$-mean. Now, adding \eqref{eq:auxmeanineq} over $j\in T_{2,k}$ gives $\Unb_{\alpha\beta,2}(A)$ on the left-hand side, and $s^{2(\beta-1)}\Unb_{\alpha\beta,2}(A)$ on the right-hand side. Hence the desired inequalities follow. For the identities with $D_1$ and $D_2$, we use that they the proven inequalities are equalities when $\beta=1$.
\end{proof}

\begin{proof}[Proof of Proposition~\ref{prop:bwvstol}]
By the triangle inequality, for any $x,y\in S^t$ and $i,j\in T_{t,k}$, we have
\[
\left|n(A,x,i)-n(A,y,i)\right|\leq \left|n(A,x,i)-N/s^t\right|+\left|n(A,y,i)-N/s^t\right|\leq 2\Tol_t(A).
\]
Therefore, maximizing on the left-hand side, $\mathrm{BW}_t(A)\leq 2\Tol_t(A)$.

Fix $x\in S^t$ and $i\in T_{t,k}$, then
\begin{align*}
\left|n(A,x,i)-N/s^t\right|&=\left|n(A,x,i)-\frac{1}{\binom{k}{t}s^t}\sum_{j\in T_{t,k}}\sum_{y\in S^t}n(A,y,j)\right|&\\
&=\left|\frac{1}{\binom{k}{t}s^t}\sum_{j\in T_{t,k}}\sum_{y\in S^t}\left(n(A,x,i)-n(A,y,j)\right)\right|\\
&\leq \frac{1}{\binom{k}{t}s^t}\sum_{j\in T_{t,k}}\left|n(A,x,i)-n(A,y,j)\right|\\&\leq \mathrm{BW}_t(A),
\end{align*}
where the first equality follows from $\sum_{j\in T_{t,k}}\sum_{y\in S^t}n(A,y,j)=N\binom{k}{t}$. Hence, maximizing the left-hand side, we get $\Tol_t(A)\leq\mathrm{BW}_t(A)$.
\end{proof}

\begin{proof}[Proof of Theorem~\ref{thm:discrepancyvsunbalance}]
The equation~\ref{eq:discrepancyvsunbalance2} follows from the definition of DD (Definition \ref{dfn:DD}), see \cite[Theorem~2]{UDpaper13}. For the other identity~\eqref{eq:discrepancyvsunbalance1}, note that
\begin{align*}
    \sum_{r,\tilde{r}}\rho^{\HammingSimilarity(A;r,\tilde{r})}&=\sum_{t=0}^{\infty}\sum_{r,\tilde{r}}\binom{\HammingSimilarity(A;r,\tilde{r})}{t}(\rho-1)^t&\text{(Newton's binomial)}\\
    &=\sum_{t=0}^k\binom{k}{t}\left(\frac{\rho-1}{s}\right)^t+\sum_{t=0}^{\infty}\Unb_{2,t}(A)(\rho-1)^t&\text{(Theorem~\ref{thm:HamminSimilarityvsUnbalance})}\\
    &=\left(1+\frac{\rho-1}{s}\right)^k+\sum_{t=0}^{\infty}\Unb_{2,t}(A)(\rho-1)^t&\text{(Newton's binomial)}\\
    &=\left(1+\frac{\rho-1}{s}\right)^k+\sum_{t=1}^{\infty}\Unb_{2,t}(A)(\rho-1)^t.&(\Unb_{2,0}(A)=0)
\end{align*}
For the remaining of the identities, we use the formulas for CD, MD and WD in \cite[(3) and (4)]{UDpaper2} and \cite[(4)]{ke2015}, together with easily obtainable bounds for points of the form $\left(\frac{2x_i-1}{2s}\right)$.
\end{proof}

\begin{proof}[Proof of Corollary~\ref{cor:discrepancyvsunbalance}]
The proof of this corollary is analogous to that of Corollary~\ref{cor:HamminSimilarityvsUnbalance1}. We apply again just \cite[Lemma~2]{UDpaper8liuhickernell} and the proof is just similar: We show the inequality with $\tilde{\gamma}$ of \eqref{eq:tildegamma} instead of $\gamma$. Then we argue that the lower bound is a growing function of $\tilde{\gamma}$ and so deduce the bound with $\gamma$ instead of $\tilde{\gamma}$.
\end{proof}

\end{document}